\documentclass[a4paper, 12pt]{article}
\usepackage{amsfonts}
\usepackage {amssymb}
\usepackage {amsmath}
\usepackage {amsmath}
\usepackage {amsthm}
\usepackage{graphicx}
\usepackage{xypic}
\usepackage {amscd}
\usepackage{mathrsfs}
\usepackage{multirow}
\usepackage[colorlinks, linkcolor=blue, anchorcolor=black, citecolor=red]{hyperref}
\usepackage{enumerate}

\usepackage{geometry}  

\newcommand{\Fut}{{\rm Fut}}

\newcommand{\bL}{{\bf L}}

\newcommand{\mcL}{{\mathcal{L}}}
\newcommand{\mcX}{{\mathcal{X}}}

\newcommand{\vol}{{\rm vol}}
\newcommand{\ord}{{\rm ord}}
\newcommand{\lct}{{\rm lct}}
\newcommand{\vphi}{\varphi}

\newcommand{\bC}{{\mathbb{C}}}
\newcommand{\mcC}{{\mathcal{C}}}

\newcommand{\FS}{{\rm FS}}
\newcommand{\bP}{{\mathbb{P}}}

\newcommand{\mcW}{\mathcal{W}}
\newcommand{\chw}{{\rm CW}}
\newcommand{\Ch}{{\rm Ch}}

\newcommand{\bG}{\mathbb{G}}

\newcommand{\bT}{\mathbb{T}}

\newcommand{\NA}{{\rm NA}}
\newcommand{\MA}{{\rm MA}}

\newcommand{\triv}{{\rm triv}}
\newcommand{\bQ}{{\mathbb{Q}}}

\newcommand{\bR}{{\mathbb{R}}}

\newcommand{\mcH}{{\mathcal{H}}}
\newcommand{\mcO}{{\mathcal{O}}}
\newcommand{\mcF}{{\mathcal{F}}}
\newcommand{\bZ}{{\mathbb{Z}}}

\newcommand{\bN}{{\mathbb{N}}}
\newcommand{\bA}{{\mathbb{A}}}
\newcommand{\la}{\langle}
\newcommand{\ra}{\rangle}

\newcommand{\bfL}{{\bf L}}
\newcommand{\mcI}{\mathcal{I}}

\newcommand{\Val}{{\rm Val}}

\newcommand{\DHM}{{\rm DH}}
\newcommand{\mcR}{\mathcal{R}}
\newcommand{\ccX}{\check{\mcX}}
\newcommand{\ccL}{\check{\mcL}}
\newcommand{\wt}{{\rm wt}}

\newcommand{\bfD}{{\bf D}}

\newcommand{\rVal}{\mathring{\Val}}

\newcommand{\bfH}{{\bf H}}
\newcommand{\bfE}{{\bf E}}

\newcommand{\lc}{{\rm lc}}
\newcommand{\bfF}{{\bf F}}
\newcommand{\ac}{{\rm ac}}

\newcommand{\Aut}{{\rm Aut}}

\newcommand{\cZ}{\mathcal{Z}}

\newcommand{\cF}{\mathcal{F}}

\newcommand{\bfS}{{\bf S}}
\newcommand{\btD}{\tilde{\bfD}}

\newcommand{\Gam}{\Gamma}
\newcommand{\rk}{{\rm rk}}
\newcommand{\btS}{\tilde{\bfS}}
\newcommand{\tbeta}{\tilde{\beta}}

\newcommand{\cW}{\mathcal{W}}
\newcommand{\bin}{{\bf in}}
\newcommand{\mult}{{\rm mult}}
\newcommand{\ddc}{{\rm dd^c}}
\newcommand{\CH}{{\rm CH}}
\newcommand{\bbM}{\mathbb{M}}

\newcommand{\cX}{\mathcal{X}}
\newcommand{\bhL}{\hat{\bfL}}

\newcommand{\bhD}{\hat{\bf D}}
\newcommand{\bfQ}{{\bf Q}}

\newcommand{\fa}{\mathfrak{a}}
\newcommand{\cE}{\mathcal{E}}
\newcommand{\cN}{\mathcal{N}}
\newcommand{\bcL}{\bar{\mcL}}
\newcommand{\bcX}{\bar{\mcX}}

\newcommand{\hvol}{\widehat{\rm vol}}
\newcommand{\col}{{\rm colen}}
\newcommand{\ree}{\mathcal{R}}
\newcommand{\Gr}{{\rm Gr}}
\newcommand{\fv}{\mathfrak{v}}
\newcommand{\mcN}{\mathcal{N}}
\newcommand{\cC}{\mathcal{C}}
\newcommand{\bm}{{\bf m}}
\newcommand{\cB}{\mathcal{B}}
\newcommand{\bfV}{{\bf V}}

\newtheorem{thm}{Theorem}[section]
\newtheorem{prop}[thm]{Proposition}
\newtheorem{defn}[thm]{Definition}

\newtheorem{cor}[thm]{Corollary}
\newtheorem{rem}[thm]{Remark}
\newtheorem{conj}[thm]{Conjecture}
\newtheorem{exmp}[thm]{Example}
\newtheorem{lem}[thm]{Lemma}

\newtheorem{defn-prop}[thm]{Definition-Proposition}

\begin{document}

\title{Algebraic uniqueness of K\"{a}hler-Ricci flow limits and optimal degenerations of Fano varieties}
\author{Jiyuan Han, Chi Li}%, Gang Tian, Feng Wang}
\date{}

\maketitle

\abstract{
We prove that for any $\bQ$-Fano variety $X$, the special $\bR$-test configuration that minimizes the $\bfH^\NA$-functional is unique and has a K-semistable $\bQ$-Fano central fibre $(W, \xi)$. Moreover there is a unique K-polystable degeneration of $(W, \xi)$. 
As an application, we confirm the conjecture of Chen-Sun-Wang about the algebraic-uniqueness for K\"{a}hler-Ricci flow limits on Fano manifolds which implies that the Gromov-Hausdorff limit of the flow does not depend on the choice of initial K\"{a}hler metrics. The results are achieved by studying algebraic optimal degeneration problems via new functionals for real valuations, which are analogous to the minimization problem for normalized volumes. 
}

\setcounter{tocdepth}{1}
\tableofcontents

\section{Introduction}

Let $X$ be a smooth Fano manifold. It is now known that $X$ admits a K\"{a}hler-Einstein metric if and only if $X$ is K-polystable (see \cite{Tia97, Berm15, Tia12, CDS15, Tia15}). 
In this paper, we are interested in the case when $X$ is not K-polystable. If $X$ is strictly K-semistable, then $X$ admits a unique 
K-polystable degeneration by \cite{LWX18}. If $X$ is K-unstable (i.e. not K-semistable), several kinds of optimal degenerations were studied. For example there is a unique special degeneration which arises in the study of Aubin's continuity method and whose associated valuation minimizes the $\delta$-invariant (see \cite{Sze16, BLZ19}). There is also a unique destabilizing geodesic ray which arise in the study of inverse-Monge-Amp\`{e}re flow (resp. Calabi flow) and minimize an $L^2$-normalized non-Archimedean Ding-invariant (resp. $L^2$-normalized radial Calabi-functional) (see \cite{His19, Xia19}).
In this paper we are interested in optimal degenerations that arise in the study of Hamilton-Tian conjecture about the long time behavior of K\"{a}hler-Ricci flows. The latter conjecture states that starting from any K\"{a}hler metric $\omega\in c_1(X)$, the normalized K\"{a}hler-Ricci converges in the Gromov-Hausdorff sense to a K\"{a}hler-Ricci soliton on a $\bQ$-Fano variety $X_\infty$.  The Hamilton-Tian conjecture been solved (see \cite{TZ16,CW14, Bam18}) and applied to give a proof of the Yau-Tian-Donaldson conjecture in \cite{CSW18}. 

It is known that $X_\infty$ coincides with $X$ if and only if there is already a K\"{a}hler-Ricci soliton on $X$ (\cite{TZ07, DS16}).
In general,  \cite{CSW18} Chen-Sun-Wang proved the following phenomenon. The metric degeneration from $X$ to $X_\infty$ induces a finitely generated filtration $\cF$ on $R=\bigoplus_{m} H^0(X, -mK_X)$ and there is a two-step degeneration:
\begin{enumerate}
\item The filtration $\cF$ as an $\bR$-test configuration (see Definition \ref{def-RTC}) degenerates $X$ to a normal Fano variety $W$ with a torus $\bT$-action generated by a holomorphic vector field $\xi$.
For simplicity, we call this step the semistable degeneration.
\item There is an $\bT$-equivariant test configuration of $(W, \xi)$ to $(X_\infty, \xi)$. We call this step the polystable degeneration.
\end{enumerate}
As explained in \cite{CSW18}, this picture is a global analogue of the picture in Donaldson-Sun's study of metric tangent cones on Gromov-Hausdorff limits of Fano K\"{a}hler-Einstein manifolds in \cite{DS17}. 
In \cite{DS17}, Donaldson-Sun conjectured that metric tangent cones depend only on the algebraic structure near the singularity. This conjecture has been confirmed in a series of works of the second-named author with his collaborators (see \cite{Li18, LX16, LX18, LWX18}) which depends on the study of minimization problem of normalized volume functional over the space of valuations centered at the singularity (see \cite{LLX20} for a survey). 
Analogous to this conjecture on metric tangent cones made, the following conjecture was proposed in \cite{CSW18}:
\begin{conj}\label{conj-CSW}
The data $\cF$, $W$ and $X_\infty$ depend only on the algebraic structure of $X$ but not on the initial metric for the K\"{a}hler-Ricci flow.
\end{conj}

In this paper we will confirm Conjecture \ref{conj-CSW}. The idea and method to prove this conjecture are in some sense parallel to the the study of minimizing normalized volumes. The second purpose of this paper is to study an analogous minimization problem in the global setting which can be studied for all $\bQ$-Fano varieties possibly singular, and prove various results about it. 

The functional we want to minimize is called $\bfH^\NA$-functional of $\bR$-test configurations.\footnote{We will mostly use the notation of non-Archimedean functionals as advocated in \cite{BHJ17}. However, note that $\bfH^\NA$ here is not the non-Archimedean entropy functional used in \cite{BHJ17}. We will not use the non-Archimedean entropy in this article.} Tian-Zhang-Zhang-Zhu (see \cite[Proposition 5.1]{TZZZ}) first introduced the $\bfH^\NA$-functional for holomorphic vector fields in their study of K\"{a}hler-Ricci flow on Fano manifolds. 
This invariant was generalized to any {\it special} $\bR$-test configuration by
Dervan-Sz\'{e}kelyhidi \cite{DS16}, who then used the result of \cite{CW14, CSW18, He16} to prove that the semistable degeneration mentioned above minimizes the $\bfH^\NA$-functional among all special $\bR$-test configurations (see Remark \ref{rem-DS}). 
For general test configurations, such $\bfH^\NA$-functional is a non-linear version of the non-Archimedean Berman-Ding functional, and 
was first explicitly used by Hisamoto in \cite{His19a} to reprove Dervan-Sz\'{e}kelyhidi's result by pluripotential theory. Note that in this paper, for the convenience of our argument and comparison with the case of $\delta$-invariant (or with the $\beta$-invariant see \eqref{eq-betag}), we will use the negative of sign convention in these previous works. 
%For general test configurations, such $\bfH^\NA$-functional.  
%In this paper, we will study the algebraic minimization problem for $\bfH^\NA$-functional among all $\bR$-test configurations.
%called $\bfH^\NA$, which is a non-linear version of non-Archimedean Berman's non-Archimedean Ding-invariant $\bfD^\NA$. 
%$\bfH^\NA$ was already used by Hisamoto in \cite{His19a} to . Again in the case of special test configurations, both $\bfH^\NA$-invariant and $\bfH^\NA$-invariant are nothing but the $\bfH^\NA$-invariant for the induced holomorphic vector field on the central fibre defined in \cite[section 5]{TZZZ}.

Conjecture \ref{conj-CSW} follows from two purely algebro-geometric statements for each step of the semistable and polystable degnerations.
\begin{thm}\label{thm-semi}
For any $\bQ$-Fano variety, the special $\bR$-test configuration that minimizes $\bfH^\NA$ is unique and its central fibre $(W, \xi)$ is K-semistable (see Definition \ref{def-Dstable}).
\end{thm}
%Now let $\bT=(\bC^*)^r$ be a complex torus. Assume $\bT$ acts effectively on a $\bQ$-Fano variety $X$ with moment polytope denoted by $P$. Let $g$ be a smooth positive function on $P$.
\begin{thm}\label{thm-poly}
If $(X, \xi)$ is K-semistable, then there exists a unique K-polystable degeneration.
\end{thm}
\begin{cor}\label{cor-conj}
The conjecture \ref{conj-CSW} is true for any smooth Fano manifold. In particular, the Gromov-Hausdorff limit $X_\infty$ for the K\"{a}hler-Ricci flow does not depend on the initial metric of the flow.
\end{cor}
 %As in the usual $\bfD^\NA$-invariant, one advantage of $\bfH^\NA$-invariant over the $\bfH^\NA$-invariant is that it can be naturally defined for any filtrations, not necessarily finitely generated. 
To prepare for the proof of such results, we will first carry out an algebraic study of $\bfH^\NA$-functional, which is analogous to the study of minimization problem for normalized volume or the $\delta$-invariant.
We will show in Theorem \ref{thm-MMP} that the MMP process devised in \cite{LX14}  decrease the $\bfH^\NA$-invariant of test configurations. This requires us to derive new intersection formula (Proposition \ref{eq-Ekint}) and derivative formula for the $\bfH^\NA$ invariant, whose proof is motivated by a similar derivative formula in our previous work \cite{HL20}.
 
We will then introduce the following $\tbeta$-functional on $\Val(X)$ the space of valuations on $X$: for any $v\in \Val(X)$ with $A_X(v)<+\infty$, we define
\begin{equation}\label{eq-tbeta0}
\tbeta(v)=A_X(v)+\log\left(\frac{1}{(-K_X)^{\cdot n}}\int_\bR e^{-\lambda} (-d\vol(\cF^{(\lambda)}_v))\right).
\end{equation}
If $A_X(v)=+\infty$, then we define $\tbeta(v)=+\infty$. This functional is a non-linear version of the $\beta$-functional in the literature of K-stability, and is a global analogue of the normalized volume. Unlike the case of normalized volume functional, the $\tbeta$ invariant is not invariant under rescaling of valuations. Indeed, when restricted to the ray of multiple of a fixed valuation $v\in \Val(X)$ with $A(v)<+\infty$, there is a unique minimizer which is non-trivial if and only if $\beta(v)<0$ (Proposition \ref{prop-vscaling}). 
The above MMP result implies that the minimum can be approached by a sequence of special divisorial valuations.
%which have been shown to be related to lc places of $\bQ$-complements in \cite{BLX19}. 
 As a consequence, one can adapt the method developed in \cite{BLX19} to show that there is minimizing valuation which is quasi-monomial (see Theorem \ref{thm-exist}). On the other hand, $\bfH^\NA$-invariant for special test configurations is expressed as the $\tbeta$-invariant (see Lemma \ref{lem-tDtbeta}). Combine these discussion, we will prove (see section \ref{sec-filtration} and section \ref{sec-NAfunctional} for relevant notations):
 \begin{thm}
 For any $\bQ$-Fano variety $X$, we have the identity:
 \begin{equation}
 \inf_{\cF \text{ filtration }}\bfH^\NA(\cF)=\inf_{(\mcX, \mcL, a\eta) \text{ special }}(\bfH^\NA(\mcX, \mcL, a\eta))=\inf_{v\in \Val(X)}\tbeta(v).
 \end{equation}
 Moreover the last infimum is achieved by a quasi-monomial valuation.
 \end{thm} 
 As in the cases of normalized volume, we conjecture that the minimizer is unique and induces a special $\bR$-test configuration (see Conjecture \ref{conj-special}) whose central fibre (with the induced vector field) must then be K-semistable by the following result. 
When $X$ is smooth, by the result of Dervan-Sz\'{e}kelyhidi \cite{DS16} the existence of such special minimizing valuation is implied by the work \cite{CW14, CSW18}. We also note that optimal degenerations (of various kinds) in the toric case are well-studied (see \cite{WZ04} for the toric result for K\"{a}hler-Ricci flow).
\begin{thm}[{=Theorem \ref{thm-minsemi}}]
A special $\bR$-test configuration minimizes $\bfH^\NA$ if and only if its central fibre is K-semistable.
\end{thm}

The uniqueness in Theorem \ref{thm-semi} about the semistable degeneration is nothing but the result on the uniqueness of the minimizer of $\tbeta$ among all quasi-monomial valuations associated to special $\bR$-test configurations. The proof of this fact uses the technique of initial term degeneration, which is motived by study of normalized volumes (see \cite{Li17, LX16, LX18}). This process essentially reduces the question to the uniqueness of minimizer of $\bfH^\NA$ (actually a variant of $\bfH^\NA$ after the work of Xu-Zhuang \cite{XZ19}) along an interpolation between a fixed filtration and a weight filtration (induced by a holomorphic vector field) on the central fibre. The interpolation is constructed by using the rescaling of twist of the fixed filtration (the twist of filtration is in the sense of \cite{Li19} generalizing \cite{His16b}), which we can deal with using the technique of Newton-Okounkove bodies and Boucksom-Jonsson's work on the characterization of equivalent filtrations. 
The valuative formulation is useful because filtrations associated to valuations are equivalent if and only if they are the same (Proposition \ref{prop-equiv}). 
 Again unlike the case of normalized volumes or the case of $\delta$-invariant in \cite{BLZ19}, the minimizing valuation in the current global setting is expected to be absolutely unique, not just up to rescaling or twisting. This is because of a strict convex property of $\bfH^\NA$ functional, which goes back to Tian-Zhu's work in \cite{TZ02} on the uniqueness of K\"{a}hler-Ricci vector fields among the Lie algebra of a torus.

To deal with the polystable step, we first introduce the equivariant version of normalized volumes. Most results about normalized volumes can be generalized for the equivariant version.  
Finally we complete the proof of Theorem \ref{thm-poly} by adapting the argument in \cite{LWX18} about uniqueness of K-polystable degeneration of K-semistable Fano varieties. 
 
To end this introduction, we summarize in the following table the quantities that are used in each of the two steps. 
%  the result in Theorem \ref{thm-semi} essentially prove the uniqueness of regular minimizers. 

\renewcommand{\arraystretch}{1.5}
\begin{center}
\begin{tabular}{ |l|l|l| } 
 \hline
%Truth of & \multirow{2}{*}
Degenerations & Semistable & Polystable \\
%\\$\bfF'^\infty=\bfF^\NA$ & & (including destabilizing geodesic rays) \\
\hline
Valuations & $\Val(X)$ & $\Val_{C, o}^{\bC^*\times\bT}$ \\
\hline
Anti-derivative & $\bfH^\NA$, $\hat{\bfH}^\NA$, $\tbeta$ & $\hvol_g$   \\
\hline
Derivative &$\bfD_\xi^\NA$, $\Fut_\xi$ & $\bfD_\xi^\NA$, $\beta_g$ \\ 
%\hline 
%Convexity for torus valuations & Tian-Zhu & Martelli-Sparks-Yau \\
%\hline
%Integral formula &  & \\
\hline
Derivative formula & \eqref{eq-der2} & \eqref{eq-derhvg} \\
\hline
%\multirow{2}{*}{$\bfH$, $\bfM$} &  smooth positive metric (\cite{Tia97, Tia14}, \cite{BHJ19}) & $``\ge"$ (Theorem \ref{thm-grslope}.1) \\ 
%& associated geodesic ray (Theorem \ref{thm-grslope}.2) &  $``\le "$ (open, Conjecture \ref{conj-Hslope})   \\ 
% \hline
\end{tabular}
\end{center}

{\bf Postscript Note:}  After we finish the paper, we are informed by F. Wang and X. Zhu that they are using analytic methods to prove related uniqueness results for K\"{a}hler-Ricci flow limits based on their recent work on Hamilton-Tian conjecture (see \cite{WZ20a, WZ20b}).

\vskip 1mm
\noindent {\bf Acknowledgements: } 
J. Han would like to thank Jeff Viaclovsky for his teaching and
support over many years. C. Li is partially supported by NSF (Grant No. DMS-1810867) and an
Alfred P. Sloan research fellowship. We would like to thank G. Tian and X. Zhu for their interest and comments, and F. Wang and X. Zhu for informing us their work. We also thank M. Jonsson for helpful comments.

\section{Preliminaries}

\subsection{Some notations}\label{sec-notation}
Let $X$ be a $\bQ$-Fano variety. In this paper for the simplicity of notations, we assume that $-K_X$ is Cartier. The modification to the general $\bQ$-Cartier case is straightforward (see e.g. \cite{Li19}). For any $m, \ell\in \bN$, set
\begin{eqnarray}
&&R_m:=H^0(X, -mK_X), \quad R=\bigoplus_{m=0}^{+\infty} R_m, \\
&& N_m=\dim R_m, \quad \bfV=(-K_X)^{\cdot n}=\lim_{m\rightarrow+\infty}\frac{N_m}{m^n/n!}\\
&&R^{(\ell)}_m:=H^0(X, -m\ell K_X), \quad R^{(\ell)}=\bigoplus_{m=0}^{+\infty} R^{(\ell)}_m.
\end{eqnarray}
We will denote by $\Val(X)$ the space of real valuations on $\bC(X)$, by $\rVal(X)$ the set of real valuations $v$ with $A_X(v)<+\infty$ and by $X^{\rm div}_\bQ$ the set of divisorial valuations (i.e. the valuations of the form $a\cdot \ord_E$ with $a\ge 0$ and $E$ a prime divisor over $X$). A valuation $v\in \Val(X)$ is quasi-monomial if there exist a birational morphism $Y\rightarrow X$ and a simple normal crossing divisors $E=\cup_{i=1}^d E_i\subset Y$ such that $v$ is a monomial valuation on $Y$ with respect to the local coordinates defining $E_i$, whose center of $v$ over $Y$ is an irreducible component of $\cap_{i\in J} E_i$ where $J\subseteq\{1,\dots,d\}$ is a subset. We denote by ${\rm QM}(Y, E)$ the set of such quasi-monomial valuations.
We refer to \cite{ELS03, JM12} for the more details about such quasi-monomial (or equivalently the Abhyankar) valuations.

In this paper, $\bT$ denotes a complex torus $(\bC^*)^r=((S^1)^r)_\bC$ that acts effectively on a $\bQ$-Fano variety $X$. There is a canonical action of $\bT$ on (any multiple of) $-K_X$. 
Set 
\begin{equation}
N_\bZ={\rm Hom}(\bC^*, \bT), \quad N_\bR=N_\bZ\otimes_\bZ\bR, \quad M_\bZ={\rm Hom}(\bT, \bC^*), \quad M_\bR=M_\bZ\otimes_\bZ\bR.
\end{equation}
For any $\xi\in N_\bR$, we have a valuation $\wt_\xi\in \Val(X)$ as follows. For any $f\in \bC(X)=\bigoplus_{\alpha\in M_\bZ} \bC(X)_\alpha$, 
\begin{equation}\label{eq-wtxi}
\wt_\xi(f)=\min\left\{\la\alpha, \xi\ra; f=\sum_{\alpha} f_\alpha, f_\alpha \neq 0 \right\}.
\end{equation}
Moreover for any $m\in \bN$, we have a weight decomposition induced by the canonical $\bT$-action on $(X, -mK_X)$:
\begin{equation}
R_m=\bigoplus_{\alpha\in M_\bZ} (R_m)_\alpha=(R_m)_{\alpha^{(m)}_1}\oplus \cdots\oplus  (R_m)_{\alpha^{(m)}_{N_m}}.
\end{equation}

Moreover, we will use the following notations for any $\bQ$-Fano variety.
Let $e^{-\tilde{\vphi}}$ be an $(S^1)^r$-invariant smooth positively curved Hermitian metric on $-K_X$ (e.g. as the restriction of a Fubini-Study metric under an equivariant embedding of $X$ into projective space). We identity any $\eta\in N_\bR$ with the corresponding holomorphic vector field on $X$. Because $\bT$-action canonically lifts to act on $-K_X$, we can set 
\begin{equation}\label{eq-thetaeta}
\theta_{\tilde{\vphi}}(\eta)=-\frac{\mathfrak{L}_\eta e^{-\tilde{\vphi}}}{e^{-\tilde{\vphi}}}.
\end{equation}
Then $\theta_{\tilde{\vphi}}(\eta)$ is a Hamiltonian function of $\eta$ with respect to $\ddc\tilde{\vphi}=\frac{\sqrt{-1}}{2\pi}\partial\bar{\partial}\tilde{\vphi}\ge 0$: 
\begin{equation}
\iota_\eta \ddc\tilde{\vphi}=\frac{\sqrt{-1}}{2\pi}\bar{\partial} \theta_{\tilde{\vphi}}(\eta).
\end{equation} 
Moreover, $(z, \eta)\mapsto \theta_{\tilde{\vphi}}(\eta)(z)$ is equivalent to the moment map $\bm_{\tilde{\vphi}}: X\rightarrow M_\bR$ whose image is the moment polytope $P$ of $\bT$-action on $(X, -K_X)$ which does not depend on the choice of $\tilde{\vphi}$. It is  known that the measure 
\begin{equation}
\frac{n!}{m^n}\sum_{i} \dim (R_m)_{\alpha^{(m)}_i}\cdot \delta_{\frac{\alpha^{(m)}_i}{m}} 
\end{equation}
converges weakly to the Duistermaat-Heckman measure $(\bm_{\tilde{\vphi}})_*(\ddc\tilde{\vphi})^n$ (see \cite{Bri87} ,\cite[Proposition 4.1]{BWN}).

For any subset $S\subseteq \bR^n$, we will use $dy_S$ or just $dy$ to denote the Lebesgue measure of $S$.

\subsection{$\bR$-test configuration and filtrations}\label{sec-filtration}
We will use extensively the language of filtrations: 
%We first recall that definition of test configurations and filtrations, which have been much studied in the subject of K-stability.
%The twists of test configurations first appeared in the work of Hisamoto (\cite{His16a, His16b}).
\begin{defn}[\cite{BC11}]\label{def-filtr}
A filtration $\cF:=\mcF R_\bullet$ of the graded $\bC$-algebra $R=\bigoplus_{m=0}^{+\infty}R_m$ consists of a family of subspaces $\{\mcF^\lambda  R_m\}_x$ of $R_m$ for each $m\ge 0$ satisfying:
\begin{itemize}
\item (decreasing) $\mcF^\lambda  R_m\subseteq \mcF^{\lambda'}R_m$, if $\lambda\ge \lambda'$;
\item (left-continuous) $\mcF^\lambda R_m=\bigcap_{\lambda'<\lambda}\mcF^{\lambda'}R_m$; 
\item (multiplicative) $\mcF^\lambda  R_m\cdot \mcF^{\lambda'} R_{m'}\subseteq \mcF^{\lambda+\lambda'}R_{m+m'}$, for any $\lambda, \lambda'\in \bR$ and $m, m'\in \bZ_{\ge 0}$;
\item (linearly bounded) There exist $e_-, e_+\in \bZ$ such that $\mcF^{m e_-} R_m=R_m$ and $\mcF^{m e_+} R_m=0$ for all $m\in \bZ_{\ge 0}$. %We will call $e_+$ to be a shifting number.
\end{itemize}
Similarly one define filtration on $R^{(\ell)}$ for any $\ell\ge 1\in \bN$.
%We say that $\mcF$ is a $\bZ$-filtration if $\mcF^\lambda  R_m=\mcF^{\lceil x\rceil} R_m$ for each $x\in \bR$ and $m\in \bZ_{\ge 0}$. 

%Given such a filtration $\mcF$, for any $\theta\in \bR$, the $\theta$-shifting of $\mcF$, denoted by $\mcF(\theta)$ is defined to be the filtration given by:
%\begin{equation}\label{eq-Fshift}
%\mcF(\theta)^x R_m:=\mcF^{x-m\ell_0 \theta} R_m.
%\end{equation}
\end{defn}

\begin{exmp}\label{ex-va2fil}

Given any valuation $v\in \rVal(X)$, we have an associated filtration $\mcF=\mcF_v$:
\begin{equation}\label{eq-filval}
\mcF_v^\lambda R_m:=\{s\in R_m; v(s)\ge \lambda\}.
\end{equation}
In particular, if there is a $\bT$-action on $X$, for any $\xi\in N_\bR$, we have a filtration $\cF_{\wt_\xi}$ associated to the valuation $\wt_\xi$ in \eqref{eq-wtxi}.

%In this case, 
%\begin{equation}
%\bfE^\NA(\cF_v)=S(v)=\frac{1}{\bfV}\int_0^{+\infty} \vol(\mcF_v^{(x)}R)dx=:\frac{1}{\bfV}\int_0^{+\infty} \vol(L-x v) dt,
%\end{equation}
%The following quantity plays an important role in recent studies of K-stability (see e.g. \cite{Fuj19a, Li17, BlJ17}):
%\begin{equation}\label{eq-defSLv}
%S_L(v)=\frac{1}{\ell_0^{n+1} L^{\cdot n}}\int_0^{+\infty} \vol(\mcF_v^{(x)}R)dx=:\frac{1}{L^{\cdot n}}\int_0^{+\infty} \vol(L-t v) dt,
%\end{equation}
%where we have denoted by $\vol(L-t v)$ the quantity $\vol(\mcF_v^{(t\ell_0)}R^{(\ell_0)})/\ell_0^{n+1}$.
%Note that for any $a>0$, $\cF_{av}=a\cF_v$. %, \quad \vol(\cF^{(x)}_{cv})=\vol(\cF^{(x/c)}).$
The trivial filtration $\cF_\triv$ is the filtration associated to the trivial valuation: $\cF_\triv^x R_m$ is equal to $R_m$ if $x\le 0$, and is equal to $0$ if $x>0$.
%By Izumi's inequality (see \cite{JM12, Li18}), there exist $c_1,c_2>0$ such that $c_1\cdot \ord_{W}\le v\le c_2 A_{(Z,Q)}(v) \ord_W$ where $W$ is the center of $v$. So we get $\lambda_{\min}(\mcF_v)=0$. Then by integration by parts we get:
%\begin{equation}\label{eq-ENAFv}
%\bfE^\NA(\mcF_v)=-\frac{1}{\ell_0^n L^{\cdot n}}\int_0^{+\infty} \frac{x}{\ell_0} \cdot d\vol(\mcF^{(x)}R)=S_L(v).
%\end{equation}
%Moreover, by \cite[Proposition 2.1]{Fuj19b} (see also \cite[(5.3)]{BoJ18b}), we have a very useful inequality:
%\begin{equation}\label{eq-SvsJ}
%\frac{1}{n} S_L(v) \le \bfJ^\NA(\mcF_{v})=\Lam^\NA(\mcF_{v})-S_L(v)\le n S_L(v).
%\end{equation}

\end{exmp}

\begin{exmp}
For any filtration $\cF$, we will denote by $\cF_\bZ$ the filtration defined by $\cF_\bZ^\lambda R_m=\cF^{\lceil \lambda \rceil}R_m$.
\end{exmp}

%\begin{exmp}
%Let $\fa_\bullet=\{\fa_k\}_{k\in\bZ}$ be a graded sequence of ideal sheaves on $X$. We define:
%\begin{equation}
%\cF^\lambda_{\fa_\bullet}R_m=H^0(X, \mcO_X(-mK_X)\otimes \fa_{\lceil \lambda\rceil}).
%\end{equation}
%Conversely, let $\cF$ be any filtration. Set $\fb_\bullet=\{\fb_k\}$ with
%\begin{equation}
%\fb_k(\cF)=\lim_{m\rightarrow+\infty} {\rm Image}(\cF^k R_m \otimes \mcO_X(-mL)\rightarrow \mcO_X) . 
%\end{equation}
%Then $\cF^\lambda R_m\subseteq \cF_{\fb_\bullet}^{\lambda-1} R_m$.
%\end{exmp}
%Following \cite{LM09} and \cite{KK14}, we introduce:
\begin{defn}[\cite{KK12, KK14, LM09}]\label{def-good}
We say a valuation $\mathfrak{v}: \bC(X)\rightarrow \bZ^n$ (where $\bZ^n$ is ordered lexicographically) is a faithful valuation if 
$\fv(\bC(X))\cong \bZ^n$. Note that such valuation always has at most one dimensional leaves (in the sense of \cite{KK12}): if $\fv(f)=\fv(g)$ for $f,g\in \bC(X)$ then there exists $c\in \bC^*$ satisfies $\fv(f+cg)>\fv(f)$.
\end{defn}
Fix such a faithful valuation $\fv$. For any $t\in \bR$, define the Newton-Okounkov body of the graded linear series
\begin{equation}
\cF^{(t)}:=\cF^{(t)} R_\bullet:=\{\cF^{tm}R_m\}.
\end{equation}
as closed convex hull of unions of rescaled values of elements from $\cF^{(t)}$:
\begin{equation}
\Delta(\cF^{(t)})=\overline{\bigcup_{m=1}^{+\infty} \frac{1}{m}\fv\left(\cF^{tm}R_m\right)}.
\end{equation} 
By the theory of Newton-Okounkov bodies, we know that (\cite{Oko96, LM09, KK14}):
\begin{equation}
n!\cdot \vol\left(\Delta(\mcF^{(t)})\right)=\vol(\cF^{(t)}R_\bullet)=\lim_{m\rightarrow+\infty}\frac{\dim_\bC \cF^{mt}R_m}{m^n/n!}.
\end{equation}
When $t\ll 0$, $\Delta(\cF^{(t)})=:\Delta_\fv(X, -K_X)=\Delta(X)$ is associated to the complete graded linear series $\{R_m\}_m$.
Following \cite{BC11}, define the concave transform 
\begin{eqnarray}
&&G^\cF: \Delta(X)\longrightarrow \bR\\
&&G^\cF(y)=\sup\{t; y\in \Delta(\cF^{(t)})\}.
\end{eqnarray}

Given any filtration $\cF=\{\mcF^{\lambda}R_m\}_{\lambda\in\bR}$ and $m\in \bZ_{\ge 0}$, the successive minima on $R_m$ is the decreasing sequence 
\[
\lambda^{(m)}_{\max}=\lambda^{(m)}_1\ge \cdots \ge \lambda^{(m)}_{N_m}=\lambda^{(m)}_{\min}
\]
defined by:
\[
\lambda^{(m)}_j=\max\left\{\lambda \in \bR; \dim_{\bC} \mcF^{\lambda} R_m \ge j \right\}.
\]

\begin{thm}[\cite{BC11}]\label{thm-BoCh}
\begin{enumerate}
\item The funciton $x\mapsto \vol(\cF^{(x)} R_\bullet)^{1/n}$ is concave on $(-\infty, \lambda_{\max})$ and vanishes on $(\lambda_{\max}, +\infty)$.
\item
As $m\rightarrow+\infty$, the Dirac type measure
\begin{equation}
\nu_m=\frac{n!}{m^n}\sum_i \delta_{\frac{\lambda^{(m)}_i}{m}}=-\frac{d}{d t} \frac{{\rm dim}_{\bC} \mcF^{mt}H^0(Z, m\ell_0 L)}{m^n/n!}
\end{equation}
converges weakly to a measure with total mass $V=(-K_X)^{\cdot n}$:
\begin{equation}
\DHM(\cF):=n!\cdot (G^\cF)_*dy=-d\vol(\cF^{(t)})
\end{equation}
where $dy$ is the Lebesgue measure on $\Delta(X)$.

%\begin{prop}[{\cite{BC11}, \cite[Corollary 5.4]{BHJ17}}]\label{BHJvol}
%converges weakly as $m\rightarrow+\infty$ to the probability measure:
%\[
%\DHM(\mcF):=-\frac{1}{\ell_0^n L^{\cdot n}} d\; \vol\left(\mcF^{(t)}\right)=-\frac{1}{\ell_0^n L^{\cdot n}} \frac{d}{d t} \vol\left(\mcF^{(t)}\right) dt.
%\]
\item The support of the measure $\DHM(\mcF)$  is given by ${\rm supp}(\DHM(\mcF))=[\lambda_{\min}, \lambda_{\max}]$ with 
\begin{align}
&
\displaystyle \lambda_{\min}:= \lambda_{\min}(\mcF):=\inf\left\{t\in \bR; \vol\left(\mcF^{(t)}\right)< \bfV\right\}; \label{lambdamin} \\
&
\lambda_{\max}:=\lambda_{\max}(\mcF):=\lim_{m\rightarrow+\infty}\frac{\lambda_{\max}^{(m)}}{m}=\sup_{m\ge 1}\frac{\lambda_{\max}^{(m)}}{m}.\label{lambdamax}
\end{align}
Moreover, $\DHM(\cF)$ is abosolutely continuous with respect to the Lebesgue measure, except perhaps for a point mass at $\lambda_{\max}$. 
\end{enumerate}
%Moreover, $\DHM(\mcF)$ is absolutely continuous with respect to the Lebesgue measure, except perhaps for a Dirac mass at $\lambda_{\max}$.
\end{thm}

\begin{exmp}
If $v\in \Val(X)$ is quasi-monomial, it is shown in \cite{BKMS} that $\DHM$ is absolutely continuous with respect to the Lebesgue measure on $\bR$, i.e. there is no Dirac mass at $\lambda_{\max}(\cF_v)$.

\end{exmp}

\begin{defn}
Let $\cF$ be any filtration. For any $a>0$ the $a$-rescaling of $\mcF$ is given by:
\begin{equation}
(a\cF)^\lambda R_m=\cF^{\lambda/a} R_m.
\end{equation}
For any $b\in \bR$, the $b$-shift is given by:
\begin{equation}
\cF(b)^\lambda  R_m=\cF^{\lambda-bm} R_m.
\end{equation} 
Set 
\begin{equation}\label{eq-aFb}
a\cF(b)=(a\cF)(b)=a(\cF(b/a)), \quad \text{ i.e. } a\cF(b)^xR_m=\cF^{\frac{x-bm}{a}}R_m.
\end{equation}
\end{defn}
We have the easy identities:
\begin{equation}\label{eq-Gbasic}
\Delta(a\cF(b)^{(t)})=\Delta(\cF^{(\frac{t-b}{a})}), \quad 
G_{a\cF(b)}=a G_\cF+b, \quad \vol(a\cF(b)^{(t)})=\vol(\cF^{(\frac{t-b}{a})}).
\end{equation}

For any $f_m\in R_m$, set:
\begin{equation}\label{eq-vcF}
\bar{v}_\cF(f_m)=\sup\{\lambda; f_m\in \cF^\lambda  R_m\}=\max\{\lambda; f_m\in \cF^\lambda  R_m\}
\end{equation}
and for any $f=\sum_m f_m\in \mcR$ with $f_m\in R_m$, set:
\begin{equation}
\bar{v}_\cF(\sum_m f_m)=\min\{\bar{v}_\cF(f_m); f_m\neq 0\in R_m\}.
\end{equation}
Then $\bar{v}_\cF$ is a semi-valuation on $R=\bigoplus_m R_m$, satisfying:
\begin{equation}\label{eq-semival}
\bar{v}_\cF(f+g)\ge \min\{\bar{v}_\cF(f), \bar{v}_\cF(g)\}, \quad \bar{v}_\cF(fg)\ge \bar{v}_\cF(f)\cdot \bar{v}_\cF(g).
\end{equation}
Set 
\begin{equation}
\Gamma^+(\cF):=\left\{\lambda^{(m)}_i-\lambda^{(m)}_{N_m}; m\ge 0, 1\le i\le N_m\right\}.
\end{equation}
Denote by $\Gamma(\cF)$ the group of $\bR$ generated by $\Gamma^+(\cF)$.

\begin{defn}\label{def-RTC}
\begin{itemize}
\item
The extended Rees algebra and associated graded algebra of a filtration $\cF$ are defined as:
\begin{eqnarray}\label{eq-Rees}
\ree(\cF)&=&\bigoplus_{m\ge 0} \bigoplus_{\lambda\in \Gamma(\cF)} t^{-\lambda} \cF^{\lambda} R_m, \\
\Gr (\cF)&=&\bigoplus_{m\ge 0}\bigoplus_{\lambda\in \Gamma(\cF)} t^{-\lambda} \cF^{\lambda} R_m/\cF^{>\lambda} R_m
\end{eqnarray}
where $\cF^{>\lambda}R_m=\{f\in R_m; v_\cF(f)>\lambda\}$.

%\end{defn}
%\begin{defn}\label{def-RTC}
%\begin{itemize}
\item
If $\ree(\cF)$ is a finitely generated, we say that $\cF$ is finitely generated and call $\cF$ an $\bR$-test configuration. 
In this case, $\Gam(\cF)$ is a finitely generated free Abelian group: $\Gam(\cF)\cong \bZ^{\rk (\Gamma(\cF))}$ and we call $\rk (\Gam(\cF))$ the rank of $\cF$. Moreover, $\Gr(\cF)$ is also finitely generated and we call the projective scheme ${\rm Proj}(\Gr(\cF))=:X_{\mcF, 0}$ the central fibre of $\cF$. 

There is an induced filtration $\cF|_{X_{\cF,0}}:=\cF'R'=\{\cF'R'_m\}$ on $R':=\Gr(\cF)$ which the homogeneous coordinate ring of the central fibre:
\begin{equation}\label{eq-cF'R'}
\cF'^\lambda R'_m=\bigoplus_{\lambda^{(m)}_i\ge \lambda} \cF^{\lambda^{(m)}_i}R_m/\cF^{>\lambda^{(m)}_i}R_m.
\end{equation}
The $\Gamma(\cF)$ grading of $\Gr(\cF)$ corresponds to a holomorphic vector field $\eta_\cF$ on the central fibre which generates a $\rk(\Gamma(\cF))$ torus action. 

\item

We say an $\bR$-test configuration $\cF$ is special if its central fibre $X_{\cF, 0}$ is a $\bQ$-Fano variety and there is an isomorphism $\Gr(\cF)\cong R(X_{\cF,0}, -K_{X_{\cF,0}})=:R'$. 
In this case, there is $\sigma\in  \bR$ such that:
\begin{equation}\label{eq-RspecialcF}
\cF' R'=\cF'_{\wt_\eta} R'(-\sigma).
\end{equation}

\end{itemize}

\end{defn}

\begin{rem}
We can naturally extend the above definition to filtrations on $R^{(\ell)}$ for any $\ell\in \bN_{\ge 1}$.
Indeed we will actually identify two filtration if they induce the same non-Archimedean metric on $(X^\NA, L^\NA)$ with $L=-K_X$. See Definition \ref{defn-phiF}.
\end{rem}

There are two equivalent geometric description of $\bR$-test configurations which we now explain.

\begin{enumerate}
\item ({\bf Geometric $\bR$-TC I:}) Let $\iota: X\rightarrow \bP^{N_\ell}$ be a Kodaira embedding by a basis of $R_\ell=H^0(X, \ell (-K_X))$ for some $\ell>0$ and $\eta$ be a holomorphic vector field on $\bP^{N_\ell-1}=\bP(H^0(X, \ell (-K_X)^*)$ that generates an effective holomorphic action on $\bP^{N_\ell-1}$ by a torus $\bT$ of rank $r$. Then we get a weight decomposition $R_\ell=\bigoplus_{\alpha\in \bZ^r}R_{\ell,\alpha}$ and a filtration on $R_\ell$ by setting:
\begin{equation}\label{eq-etafil}
\cF^\lambda R_\ell=\bigoplus_{\la \alpha, \eta\ra\ge \lambda} R_{\ell,\alpha}.
\end{equation}
The filtration $\cF R_\ell$ generates a filtration on $\cF R^{(\ell)}$ which is an $\bR$-test configuration $\cF$. 

\begin{lem}
Any $\bR$-test configuration, which by definition is a finitely generated filtration, is obtained in this way. 
\end{lem}
\begin{proof}
To see this, we assume again that $\cF$ is generated by $\cF R_\ell$. For simplicity of notations, set $V=R_\ell$ and $\lambda_i=\lambda^{(\ell)}_i$. By shifting the filtration, we can normalize $\lambda_{N_\ell}=0$ and assume that we have the relation:
\begin{eqnarray*}
&&\lambda_1=\cdots=\lambda_{i_1}=:w_1\\
&>&\lambda_{i_1+1}=\cdots=\lambda_{i_2}=:w_2\\
&>&\cdots \\
&>&\lambda_{i_{k-2}+1}=\cdots=\lambda_{i_{k-1}}=:w_{k-1}\\
&>&\lambda_{i_{k-1}+1}=\cdots=\lambda_{N_\ell}=:w_k=0.
\end{eqnarray*}
In other words, $\{w_1,\dots, w_{k}\}$ is the set of distinct values of successive minima and we have a usual filtration:
\begin{equation}\label{eq-flag}
\{0\}\subsetneq \mcF^{w_1}V\subsetneq \mcF^{w_2}V \subsetneq \cdots \subsetneq \mcF^{w_k}=V.
\end{equation}
In other words, we can equivalently describe an $\bR$-filtration by the language of weighted flags.
Fix a reference Hermitian inner product $H_0$ on $V=R_\ell$, we can assign a decomposition to the flag \eqref{eq-flag}:
\begin{equation}\label{eq-wtdec}
V=V_1\oplus V_2\oplus\cdots\oplus V_{k}
\end{equation}
where $V_1=\mcF^{w_1}V$ and $V_j$ is the $H_0$-orthogonal complement of $\mcF^{w_{j-1}}V$ inside $\mcF^{w_{j}}V$, which has dimension $i_j-i_{j-1}=: d_j$. 

Fix a maximal $\bQ$-linearly independent subset of $\{w_1, \dots, w_k\}$ to be
\begin{equation}
0> w_2=:\zeta_1 > \cdots > w_{p_r}=:\zeta_r.
\end{equation}

So for each $w_j$ we can find a vector of rational numbers $\vec{r}_j=(r_{j1}, \dots, r_{jr})\in \bQ$ such that
$w_j=\sum_{p=1}^r r_{jp} \zeta_p$.
Find a common multiple of the denominator $D$ of $\{ r_{jp}; 1\le j\le k, 1\le p\le r\}$, we set $\eta=\zeta/D$ and $\alpha_j=D \vec{r}_j$ so that:
\begin{equation}
w_j=\sum_{p=1}^r \alpha_{jp} \zeta_p=\la \alpha_j, \eta\ra.
\end{equation} 
In this way we get a $(\bC^*)^r$ representation $V$ whose weight decomposition is given by \eqref{eq-wtdec}, where $V_j$ consists of elements of weight $\alpha_j$ and
%Conversely, if we have a $H_0$-unitary $(S^1)^r$ complex representation on $V$, we get a $T\cong (\bC^*)^r$ representation on $V$. Set $N_\bZ={\rm Hom}(\bC^*, T)\cong \bZ^r$ and $N_\bR=N_\bZ\otimes_\bZ \bR$. Let $V=V_1\oplus \cdots V_k$ be the weight decomposition where $V_j$ consists of elements of weight $\alpha_j$. For any $\xi\in N_\bR$, we get an $\bR$-filtration:
\begin{eqnarray*}
\mcF^\lambda V&=&\bigoplus_{\la\alpha_j, \eta\ra\ge a}V_j=\left\{v=\sum_{j=1}^k v_j; \min\{\la \alpha_j, \eta\ra; v_j\neq 0\}\ge \lambda\right\}.
\end{eqnarray*}

\end{proof}

From another point of view, let $\mcI_X\subset \bC[Z_1,\dots, Z_{N_\ell}]=S$ be the homogeneous ideal of $X$. For each $d\in \bN$, the $\bT$-action induces a representation of $\bT$ on $S_d$, the set of degree $d$ homogeneous polynomials of degree. The holomorphic vector field $\eta$ induces an order on the weights of these $\bT$-representations. 
Choosing a set of homogeneous generator of $\mcI_X$, then the initial term with respect to this order generates the ideal of $X_{\cF, 0}$. If $\sigma_\eta$ denotes the one-parameter $\bR$-group generated by $\eta$. Then we have the convergence of algebraic cycles (or schemes):
\begin{equation}
\lim_{s\rightarrow+\infty}\sigma_\eta(s)\circ [X]=[X_{\cF,0}].
\end{equation}
So we say that the $\bR$-action generated by $\eta$ degenerates $X$ into a projective scheme $X_{\cF,0}$. 

By perturbing $\eta\in N_\bR$, we can find a sequence of rational vector $\eta_k\in N_\bQ$ converging to $\eta$. For $k\gg 1$, $\eta_k$ induces an $\bR$-test configuration of rank 1 with the same central fibre $X_{\cF, 0}$. 

%This is well-known when $\rk(\cF)=1$, i.e. for the case of test configurations (see Definition \ref{def-TC}).
%We will denote such an $\bR$-test configuration by $\cF_{(\iota,\eta)}$.
%Conversely, any $\bR$-test configuration $\cF$ is equal to $\cF_{(\iota,\eta)}(b)$ with some $b\in \bR$.

\item ({\bf Geometric $\bR$-TC II:}) This description is essentially contained in \cite[section 2]{Tei03}. 
%It is not needed in our later discussion for general $\bR$-test configurations but is given here for completeness. %Taking $\ell\gg 1$, we are free to consider filtration on $R^{(\ell)}$.
%Replacing $\cF$ by $\cF(d)$ for $d\gg 1$, we can assume $\Gamma^+(\cF)\subset \bR_{\ge 0}$. 
%Let $\bC(\bar{\Gamma}^+(\cF))$ the normalization of $\bC(\Gamma(\cF))$. 
For any $\bR$-test configuration, we set $B={\rm Spec}(\bC({\Gamma}^+(\cF))\cong \bC^r$. Then there is a flat family 
\begin{equation}\label{eq-proj1}
\mcX={\rm Proj}_{\bC^r}(\ree(\cF)) \rightarrow B
\end{equation} such that the generic fibre is isomorphic to $X$ and a special fibre isomorphic to $X_{\cF, 0}$. 
Set $\mcL$ to be the relative ample line bundle $\mcO_{\mcX/\bC^r}(1)$. Fix $m\ge 0$. For any $\lambda \in \bR$, we set $\lceil \lambda \rceil=\min\{\lambda^{(m)}_i; \lambda^{(m)}_i\ge \lambda\}=\la \alpha, \eta_\cF\ra$ for $\alpha\in M_\bZ$. Then for any $\tau=(\tau_1,\dots, \tau_r)\in \bC^r$, we set
$\tau^{-\lceil \lambda \rceil}=\prod_{i=1}^r \tau_i^{\alpha_i}$ to get:
\begin{equation}\label{eq-cFTC1}
\cF^\lambda R_m=\{s\in R_m; \tau^{-\lceil \lambda \rceil }\bar{s} \text{ extends to a holomorphic section of } m\mcL\rightarrow \mcX\}
\end{equation}
where $\bar{s}$ is the meromorphic section of $m \mcL$ defined as the pull-back of $s$ via the projection $(\mcX, \mcL)\times_B(\bC^*)^r\cong (X, -K_X)\times (\bC^*)^r \rightarrow X$. 

\end{enumerate}

\begin{lem}\label{lem-integral}
If $\Gr(\cF)$ is an integral domain, then the semi-valuation $\bar{v}_\cF$ in \eqref{eq-vcF} defines a valuation on the quotient field of $R$. Denote by $v_\cF$ the restriction of $\bar{v}_\cF$ to $\bC(X)$: for $f=s_1/s_2\in \bC(X)$ with $s_1, s_2\in R_m$, set:
\begin{equation}\label{eq-vcF}
v_{\cF}(f)=\bar{v}_\cF(s_1)-\bar{v}_\cF(s_2).
\end{equation}
Then there exists $\sigma>0$ such that $\cF=\cF_{v_\cF}(-\sigma)$. In particular, this statement applies to any special $\bR$-test configuration.
\end{lem}
\begin{proof}
Fix any two homogeneous elements $s_i \in R_{m_i}, i=1,2$. Assume that $\bar{v}_\cF(f_i)=s_i$. Then $s'_i\in R'_{m_i, x_i}$. Because $\Gr(\cF)$ is integral, $s'_1 s'_2\neq 0\in R'_{m_1+m_2, x_1+x_2}$ which implies that $\bar{v}_\cF(s_1 s_2)=x_1+x_2=\bar{v}_\cF(s_1)+\bar{v}_\cF(s_2)$. From this, we easily see that $\bar{v}_\cF$ is a real valuation.

Assume $f=\frac{s_1}{s_2}=\frac{\tilde{s}_1}{\tilde{s}_2}$. Then $s_1\cdot \tilde{s}_2=s_2\cdot \tilde{s}_1$ and hence $\bar{v}_\cF(s_1)-\bar{v}_\cF(s_2)=\bar{v}_\cF(\tilde{s}_1)-\bar{v}_\cF(\tilde{s}_2)$. So $v_\cF$ in \eqref{eq-vcF} is well-defined. 

For any $s_i\neq 0\in R_m, i=1,2$, by construction $\bar{v}_\cF(s_1)-v_\cF(s_1)=\bar{v}_\cF(s_2)-\bar{v}(s_2)$. This means $b_m:=v_\cF-\bar{v}_\cF$ is constant on $R_m\setminus \{0\}$. It is easy to see that $\sigma_{m_1} \sigma_{m_2}=\sigma_{m_1+m_2}$. So we can set $\sigma=\frac{\sigma_m}{m}$ to get the conclusion.

\end{proof}

An $\bR$-test configuration with $\rk (\Gamma(\cF))=1$ is, up to rescaling, associated to the usual test configuration, a notion that plays a basic role in the subject of $K$-stability.
\begin{defn}[{see \cite{Tia97, Don02, LX14}}]\label{def-TC}
A test configuration of $(X, L)$ is a triple $(\mcX, \mcL, \eta)$, sometimes just denoted by $(\mcX, \mcL)$, that consists of:
\begin{itemize}
\item 
A variety $\mcX$ admitting a $\bC^*$-action generated by a holomorphic vector field $\eta$ and a $\bC^*$-equivariant morphism $\pi: \mcX\rightarrow \bC$, where the action of $\bC^*$ on $\bC$ is given by the standard multiplication generated by $-t\partial_t$.
\item A $\bC^*$-equivariant $\pi$-semiample $\bQ$-Cartier divisor $\mcL$ on $\mcX$ such that there is an $\bC^*$-equivariant isomorphism $\mathfrak{i}_\eta: (\mcX, \mcL)|_{\pi^{-1}(\bC\backslash\{0\})}\cong (X, L)\times \bC^*$.
\end{itemize}
We denote by $(\bar{\mcX}, \bar{\mcL})$ the natural compactification of $(\mcX, \mcL)$ by adding a trivial fibre at infinity using the isomorphism $\mathfrak{i}_\eta$.

$(X_\bC, (-K_{X})_\bC, \eta_\triv):=(X\times\bC, -K_X\times \bC, -t\partial_t)$ is called the trivial test configuration.
$(\mcX, \mcL, \eta)$ is a normal test configuration if $\mcX$ is a normal variety.

A normal test configuration $(\mcX, \mcL, \eta)$ is a special test configuration (resp. weakly special) if $(\mcX, \mcX_0)$ is plt (resp. if $(\mcX, \mcX_0)$ is log canonical) and $\mcL=-K_\mcX+c \mcX_0$ for some $c\in \bQ$. Note that by inversion of adjunction, $(\mcX, \mcL, \eta)$ being special is equivalent to the condition that $(\mcX_0, -K_{\mcX_0})$ is $\bQ$-Fano.

Assume that $\bG$ is a reductive group acting on $X$. A $\bG$-equivariant test configuration of $(X, L)$ is a test configuration $(\mcX, \mcL, \eta)$ satisfying
\begin{itemize}
\item
There is a
$\bG$ action on $(\mcX, \mcL)$ that commutes with the $\bC^*$-action generated by $\eta$ and the action of $\bG$ on $(\mcX, \mcL)\times_\bC {\bC^*}\stackrel{\mathfrak{i}_\eta}{\cong} (X, L)\times\bC^*$ coincides with the fiberwise action of $\bG$ on (the first factor of) $(X, L)\times \bC^*$.
\end{itemize}
%\item The isomorphism $(\mcZ, \mcQ, \mcL)|_{\bC^*}\cong (Z, Q, L)\times \bC^*$ is $G\times\bC^*$-equivariant.
\end{defn}

%\begin{exmp}\label{ex-tc2fil}

As mentioned above, by the work of \cite{WN12, Sze16, BHJ17}, for any $\bR$-test configuration $\mcF$ with $\rk(\cF)=1$, there exists a normal test configuration $(\mcX, \mcL, \eta)$ and $a>0$ such that 
$\cF=a \cF_{(\mcX, \mcL, \eta)}$.
%Moreover, by \cite[]{BHJ17}, there exists such a unique normal ample test configuration.
The identity \eqref{eq-proj1} becomes
\begin{eqnarray*}\label{eq-mcXProj}
\mcX={\rm Proj}_{\bC[t]}\left(\bigoplus_{m\ge 0}\bigoplus_{j\in \bZ} t^{-aj}\cF^j R_m\right).
\end{eqnarray*}
Conversely assume $(\mcX, \mcL)$ is a test configuration of $(X, L:=-K_X)$. Then we associate a filtration $\mcF=\mcF_{(\mcX, \mcL)}$ as in \eqref{eq-cFTC1}:
\vskip 1mm
$s\in \mcF^\lambda  R_m$ if and only if $t^{-\lceil \lambda\rceil}\bar{s}$ extends to a holomorphic section of $m \mcL$.
%, where $\bar{s}$ is the meromorphic section of $m \mcL$ defined as the pull-back of $s$ via the projection $(\mcX, \mcL)\times_\bC\bC^*\cong (X, L)\times\bC^*  \rightarrow X$. 
Assume that there is a $\bC^*$-equivariant birational morphism $\rho: \mcX\rightarrow X_\bC:=X\times\bC$ and write $\mcL=\rho^*L_{\bC}+D$ where $L_\bC=p_1^*L$. Then by \cite[Lemma 5.17]{BHJ17}, the filtration $\cF$ has the following more explicit description:
\begin{equation}\label{eq-filTCvan}
\cF^\lambda  R_m=\bigcap_{E}\{s\in H^0(X, m L); r(\ord_E)(s)+m\ell_0\;\ord_E(D)\ge x b_E\},
\end{equation}
where $E$ runs over the irreducible components of the central fibre $\mcX_0$, $b_E=\ord_E(\mcX_0)=\ord_E(t)$ and $r(\ord_E)$ denotes the restriction of $\ord_E$ to $\bC(Z)$ under the inclusion $\bC(Z)\subset \bC(X\times\bC^*)=\bC(\mcX)$.

When $\cF=\cF_{(\mcX, -K_{\mcX}, \eta)}$ is associated to a special test configuration, the Lemma \ref{lem-integral} applies. In fact, by \cite{BHJ17} $v_{\mcF}=v_{\mcX_0}=r(\ord_{\mcX_0})$ and by \cite{Li17}, $\sigma=A_X(v_{\mcX_0})$:
%\begin{equation}%\label{eq-specialmin}
$\cF_{(\mcX, -K_{\mcX}, \eta)}=\mcF_{v_{\mcX_0}}(-A(v_{\mcX_0})).$ 
%\end{equation}
As a consequence, for any $a>0$, we have the following identity (by \eqref{eq-aFb}):
\begin{equation}\label{eq-specialmin}
\cF_{(\mcX, -K_{\mcX}, a\eta)}=\mcF_{a v_{\mcX_0}}(-A(a v_{\mcX_0})).
\end{equation}
Note that following the definition \ref{def-RTC}, for any $a>0$ we say that $(\mcX, \mcL, a\eta)$ is a special test configuration if $(\mcX, \mcL, \eta)$ is a special test configuration.

%For this filtration, we have $\bfF^\NA(\mcF)=\bfF^\NA(\mcZ, \mcQ, \mcL)$ for $\bfF$ being the functionals defined in \eqref{eq-ENAcF}-\eqref{eq-bfDNAcF}. For $m$ sufficiently divisible we have (see \cite[Theorem 5.18 and Lemma 7.7]{BHJ17})
%\begin{equation}
%\Lam^\NA(\mcZ, \mcL)=\frac{\lambda_{\max}(\mcF_{(\mcZ, \ell_0\mcL)})}{\ell_0}=\frac{\lambda_{\max}^{(m)}(\mcF_{(\mcZ, \ell_0\mcL)})}{\ell_0 m}=\frac{1}{\bfV}\rho^*(L\times\bP^1)^{\cdot n}\cdot \overline{\mcL}.
%\end{equation} 
%Moreover, because $\mcF_{(\mcZ, \ell_0\mcL)}$ is finitely generated (see \cite{WN12, Sze15, BHJ17}), for $m$ sufficiently divisible, the $m$-th approximating test configurations $(\ccZ_m, \ccQ_m, \ccL_m)$ are equivalent to $(\mcZ, \mcQ, \mcL)$.

Note that we use the negative sign $-t\partial_t$ in our definition \ref{def-TC}. This sign convention will be convenient for our following computation as illustrated in the following simple example.
\begin{exmp}\label{exmp-P1sign}
%The following example illustrates our use of negative sign $-t\partial_t$ in the Definition \ref{def-TC}.
Consider the product test configuration $(\mcX, \mcL)$ of $\bP^1$ induced by the $\bC^*$-action:
\begin{equation}
t\circ [Z_0, Z_1]=[Z_0, t Z_1].
\end{equation}

Let $s_i, i=0, 1$ be two holomorphic section of $H^0(\bP^1, \mcO(1))$ corresponding to the homogeneous coordinate $Z_i, i=0,1$.  
Then $t$ acts on the holomorphic section by $t\cdot s_0=s_0$ and $t\circ s_1=t^{-1}s_1$. The corresponding filtration is given by (cf \eqref{eq-etafil}):
\begin{equation}
\cF^\lambda R_m={\rm Span}\{s_0^{m-i} s_1^{i}; 0\ge -i\ge  \lambda \}.
\end{equation}
The natural compactification $\bar{\mcX}$ can identified with the Hirzebruch surface $\bP(\mcO_{\bP^1}(1)\oplus \mcO_{\bP^1})$ and $\bar{\mcL}$ is given by $\mathcal{O}_{\bar{\mcX}}(D_\infty)$ where $D_\infty$ is the divisor at infinity (see \cite[Example 3]{LX14}).
The successive minima is given by $\{\lambda^{(m)}_i\}=\{-m, -m+1, \dots, 0\}$. In particular, we have:
\begin{equation}
\sum_i \lambda^{(m)}_i=-\frac{m^2}{2}-\frac{m}{2}=\frac{\bar{\mcL}^2}{2}m^2+\left(\frac{K_{\bar{\mcX}}^{-1}\cdot \bar{\mcL}}{2}-1\right)m.
\end{equation}
Moreover $\eta=-z \frac{\partial}{\partial z}$ whose Hamiltonian function is given by $\theta(\eta)=-\frac{|Z_1|^2}{|Z_1|^2+|Z_2|^2}$. Note that $\theta(\eta)_* \omega_{\FS}=dy_{[-1,0]}=\DHM(\cF)$.

\end{exmp}

\begin{exmp}
If $\cF$ is an $\bR$-test configuration, then $a \cF(b)$ is an $\bR$-test configurations for any $(a, b)\in \bR_{>0}\times\bR$.

Assume $\cF=\cF_{(\mcX, \mcL, \eta)}$ for a test configuration $(\mcX, \mcL, \eta)$. Then for simplicity of notations, we will identify $a \cF(b)$ with the data $(\mcX, \mcL+b \mcX_0, a\eta)$.

For any $d>0\in \bN$, we can consider the normalization of the base change:
\begin{equation}
(\mcX, \mcL, \eta)^{(d)}:=((\mcX, \mcL, \eta)\times_{\bC, t\rightarrow t^d}\bC)^{\rm normalization}=:(\mcX^{(d)}, \mcL^{(d)}, \eta^{(d)}).
\end{equation}
On the other hand, $\bZ_d=\la e^{2\pi \sqrt{-1}/d}\ra\hookrightarrow \bC^*$ naturally acts on the $(\mcX, \mcL)$ we can take a quotient:
\begin{equation}
(\mcX, \mcL, \eta)/_{\bZ_d}=(\mcX^{(1/d)}, \mcL^{(1/d)}, \eta^{(1/d)})
\end{equation}
to get a test configuration with a non-reduced central fiber in general.

With the this notations, for any $a>0\in \bQ$, we then have the natural identification:
\begin{equation}\label{eq-FTCbase}
\cF_{(\mcX, \mcL, \eta)^{(a)}}=a\cdot \cF_{(\mcX, \mcL, \eta)}=\cF_{(\mcX, \mcL, a \eta)}.
\end{equation}

\end{exmp}

For a filtration $\mcF R_\bullet$, choose $e_-$ and $e_+$ as in the definition \ref{def-filtr}. For convenience, we can choose $e_+=\lceil \lambda_{\rm max}(\mcF R)\rceil \in \bZ$. 
Set $e=e_+-e_-$ and define (fractional) ideals:
\begin{eqnarray}
I_{m,x}&:=&I^\mcF_{m,x}:={\rm Image}\left(\mcF^\lambda  R_m\otimes \mcO_X(-m L)\rightarrow \mcO_X\right); \label{eq-Imx}\\
\tilde{\mcI}_m&:=&\tilde{\mcI}^{\mcF}_m:= I^{\mcF}_{(m, m e_+)}t^{-m e_+}+I^{\mcF}_{(m,me_+-1)}t^{1-m e_+}+\cdots\nonumber \\
&&\hskip 4cm \cdots+ I^{\mcF}_{(m, me_-+1)}t^{-m e_--1}+\mcO_X\cdot t^{-me_-}; \label{eq-tcIm}\\
\mcI_m&:=&\mcI_m^{\mcF(e_+)}=\tilde{\mcI}^\mcF_m\cdot t^{m e_+}=I^\mcF_{(m, m e_+)}+I^{\mcF}_{(m, m e_+-1)} t^1+\cdots\nonumber\\
&&\hskip 4cm \cdots+I^{\mcF}_{(m, m e_-+1)} t^{me-1}+(t^{me})\subseteq \mcO_{X_\bC}. \label{eq-cIm}
 %\mcI^{\mcF(e_+)}_m t^{-m e_+}
%\ccZ_m&:=&\ccX^{\mcF}_m=\ccX^{\mcF(e_+)}_m=\text{ normalization of } {\rm Bl}_{\mcI^\mcF_m }(X\times\bC); \label{eq-ckmcX} \\
%\ccL_m&:=&\ccL^{\mcF}_m=\ccL^{\mcF(e_+)}_m=\pi^*(-K_X\times\bC)-\frac{1}{m m_0} E_m. \label{eq-ckmcL}
\end{eqnarray}

\begin{defn-prop}[{\cite[Lemma 4.6]{Fuj18}}]\label{defn-ckTCm}
With the above notations, for $m$ sufficiently divisible, define the $m$-th approximating test configuration $(\ccX^{\mcF}_m,  \ccL^{\mcF}_m)$ as:
\begin{enumerate}[(1)]
\item $\ccX^{\mcF}_m$ is the normalization of blowup of $X\times\bC$ along the ideal sheaf $\mcI^{\mcF(e_+)}_m$;
%\item $\ccQ^{\mcF}_m$ is the closure of $Q\times\bC^*$ under the $\bC^*$-equivariant inclusion $Q\times\bC^*\subset Z\times\bC^*\subset \mcZ$;
\item The semiample $\bQ$-divisor is given by:
\begin{equation} \label{eq-ckmcL}
\ccL^{\mcF}_m=\pi^*((-K_X)\times\bC)-\frac{1}{m } E_m+e_+\ccX_0,
\end{equation}
where $E_m$ is the exceptional divisor of the normalized blow up.
\end{enumerate}
For simplicity of notations, we also denote the data by $(\ccX_m, \ccL_m)$ if the filtration is clear. %Note that $m\ell_0 \ccL_m$ is Cartier over $\ccZ_m$.
\end{defn-prop}
It is easy to see that the filtration $\cF_{(\ccX_m, \ccL_m)}$ on $R^{(m)}$ is induced by $\cF_{\bZ}R_m$ under the canonical map $S^k R_m\rightarrow R_{km}$. By \cite[Proof of Theorem 4.13]{BoJ18b}, we have the following approximation result:
\begin{prop}[{\cite[Proof of Theorem 4.13]{BoJ18b}}]\label{prop-FBJ}
With the notations in \ref{defn-ckTCm}, the Duistermaat-Heckmann measures $\DHM(\ccX_m, \ccL_m)$ converges weakly to $\DHM(\cF)$ as $m\rightarrow+\infty$.
\end{prop}

Following Boucksom-Jonsson, it is very convenient to use the non-Archimedean metric defined by filtrations. Any filtration (in the sense of Definition \ref{def-filtr}) defines a non-Archimedean metric on $L^\NA\rightarrow X^\NA$. If we denote by $\phi_\triv$ the non-Archimedean metric associated to the trivial filtration. Then any non-Archimedean metric $\phi$ on $L^\NA$ is represented by the real valued function $\phi-\phi_\triv$ on $X^{\rm div}_\bQ$. 
\begin{defn}\label{defn-phiF}
Let $\mcF=\mcF R_\bullet$ be a filtration. For any $w\in \rVal(X)$, define the non-Archimedean metric associated to $\mcF$ by:
\begin{eqnarray}
(\phi^{\mcF}_m-\phi_\triv)(w)&=&-\frac{1}{m} G(w)\left(\tilde{\mcI}^\mcF_m %^{\frac{1}{\ell_0}}
\right)=-\frac{1}{m} G(w)\left(\mcI^{\mcF(e_+)}_m t^{-m e_+} \right)\nonumber\\
&=&
-\frac{1}{m}G(w)\left(\mcI^{\mcF(e_+)}_m\right)+e_+; 
\label{eq-NAphimF}\\
(\phi^{\mcF}-\phi_\triv)(w)&=&
-G(w)\left(\tilde{\mcI}^\mcF_\bullet %^{\frac{1}{\ell_0}}
\right)=\lim_{m\rightarrow+\infty} \phi^\mcF_m(w). \label{eq-NAphiF} %-\frac{1}{\ell_0}\lim_{m\rightarrow+\infty}G(w)\left(\mcI^\mcF_m\right)+\frac{e_+}{\ell_0}. 
\end{eqnarray}
In particular, if $v\in \rVal(Z)$ and $\mcF=\mcF_v$, then we denote $\phi_v=\phi^{\mcF_v}$. \end{defn}
%Note that from the definition \ref{defn-phiF} and \ref{defn-ckTCm} we see that:
%\begin{equation}\label{eq-phiFTC}
%\phi^\mcF_m=\phi_{(\ccZ^{\mcF}_m, \ccL^{\mcF}_m)}.
%\end{equation}

Note that $\phi^\cF_m=\phi_{\cF(\ccX_m, \ccL_m)}$ converges to $\phi$ as $m\rightarrow+\infty$. Moreover if (for simplicity) we assume that $S^k R_m\rightarrow R_{km}$ is surjective for all $k, m\ge 1$, then it is an increasing sequence in the sense that if $m_1| m_2$, then $\phi^{\cF}_{m_1}\le \phi^{\cF}_{m_2}$. If $\phi^\cF$ is continuous, then $\phi_m$ converges to $\phi$ uniformly by Dini's theorem.

The following transformation rule can be easily verified:
\begin{lem}
For any filtration $\cF$ and any $(a, b)\in \bR_{>0}\times \bR$ and $v\in X^{\rm div}_\bQ$, 
\begin{equation}\label{eq-phiab}
(\phi_{a\cF(b)}-\phi_\triv)(v)=a(\phi_\cF-\phi_\triv)(\frac{v}{a})+b.
\end{equation}
\end{lem}

%\subsection{Twists of filtrations}

\subsection{Twist of filtrations}
%We end this section with discussion of twists of filtrations following \cite{Li19}.  

%For any $\xi\in N_\bR$, 
%let $\kappa^{(m)}_j=\la \alpha^{(m)}_j, \xi\ra, j=1,\dots, N_m$ be the weight of $\xi$ on $R_m$. The Chow weight of $\xi$ on $L$ is then defined as:
%\begin{equation}\label{eq-defCW}
%\chw_L(\xi):=\lim_{m\rightarrow+\infty} \frac{1}{N_m}\sum_{j}\frac{\kappa^{(m)}_j}{m\ell_0}.
%\end{equation}
%In our set-up we have $L=-K_Z-Q$ with the canonical $\bT$-action, then (e.g. from \eqref{eq-CMstc})
%\begin{equation}\label{eq-CWFut}
%\chw_L(\xi)=-\Fut_{(Z, Q)}(\xi).
%\end{equation}
%Now let $W$ be the center of $\wt_\xi$ and $U$ be a $\bT$-invariant Zariski open set such that $U\cap W\neq \emptyset$. Let $\mathfrak{e}$ be an $\bT$-equivariant non-vanishing generator of $\mcO_Z(\ell_0 L)$ and let ${\bf w}=\frac{\mathscr{L}_\xi \mathfrak{e}}{\mathfrak{e}}$. Then we have:
%\begin{equation}
%\bfE^\NA(\mcF_{\wt_\xi})=\frac{1}{N_m}\lim_{m\rightarrow+\infty}\sum_{j} \frac{\kappa^{(m)}_j}{m\ell_0}-\frac{\mathscr{L}_\xi \mathfrak{e}}{\mathfrak{e}}=\chw_L(\xi)-{\bf w}.
%\end{equation}

Let $\mcF=\mcF R_\bullet $ be a $\bT$-equivariant filtration,  % of $R=R^{(\ell_0)}=\bigoplus_{m=0}^{+\infty} H^0(Z, m\ell_0L)$.
which means that $\mcF^\lambda  R_m$ is a $\bT$-invariant subspace of $R_m$ for any $x\in \bR$.
For $\alpha\in M_\bZ=N_\bZ^\vee$, denote the weight space
\begin{equation}
(R_m)_\alpha=\{s\in R_m; \tau\circ s=\tau^\alpha s \text{ for all } \tau\in (\bC^*)^r \}.
\end{equation}
Then we have:
\begin{equation}
(\mcF^\lambda  R_m)_\alpha:=\{s\in \mcF^\lambda  R_m; \tau\circ s=\tau^\alpha s\}=\mcF^\lambda  R_m \cap (R_m)_\alpha,
\end{equation}
and the decomposition:
\begin{equation}
\mcF^\lambda  R_m=\bigoplus_{\alpha\in M_\bZ} (\mcF^\lambda  R_m)_\alpha.
\end{equation}

\begin{defn}[\cite{Li19}]
For any $\xi\in N_\bR$, the $\xi$-twist of $\mcF$ is the filtration $\mcF_\xi R_\bullet$ defined by:
\begin{equation}\label{eq-Ftwist}
\mcF_\xi^\lambda R_m=\bigoplus_{\alpha\in M_\bZ} (\mcF_\xi^\lambda R_m)_\alpha, \quad \text{where}\quad
(\mcF_\xi^\lambda R_m)_\alpha:=(\mcF^{\lambda-\la \alpha, \xi\ra} R_m)_\alpha.
\end{equation}
\end{defn}

%\begin{exmp}
%Let $(\mcZ, \mcQ, \mcL)$ be a test configuration of $(Z, Q, L)$, which determines a filtration $\mcF:=\mcF_{(\mcZ, \ell_0 \mcL)}$ of $R^{(\ell_0)}$ (see Example \ref{ex-tc2fil}). Recall that $s\in \mcF^\lambda  R_m$ if and only if $t^{-\lceil x\rceil}\bar{s}$ extends to a holomorphic section.
%Let $\xi\in N_\bZ$. If $s\in (\mcF^\lambda  R_m)_\alpha$, then $\bar{\sigma}_\xi^*\bar{s}=t^{\la \alpha, \xi\ra}\bar{s}$ which implies $s\in \left(\mcF^{x-\la \alpha, \xi\ra}_{(\mcZ_\xi, \ell_0 \mcL_\xi)} R_m\right)_\alpha$.
%So we get the identification:
%$\mcF^\lambda _{(\mcZ_\xi, \ell_0 \mcL_\xi)}R_m=\mcF^\lambda _{(\mcX, \ell_0 \mcL),\xi} R_m$.
%\end{exmp}
%The following proposition deals with twists of filtrations associated to valuations.
%\begin{defn}[\cite{His16b}]
%For any $\xi\in N_\bR$, the $\xi$-twist of $(\mcZ, \mcQ, \mcL, \eta)$ is the data $(\mcZ, \mcQ, \mcL, \eta+\xi)$, which, for simplicity, will also be denoted by $(\mcZ_\xi, \mcQ_\xi, \mcL_\xi)$. If $\xi\in N_\bZ$, then $(\mcZ_s, \mcQ_\xi, \mcL_\xi)=(\mcZ, \mcQ, \mcL, \eta+\xi)$ is a test configuration. In general, we shall call $(\mcZ, \mcQ, \mcL, \eta+\xi)$ to be an $\bR$-test configuration.
%\end{defn}

\begin{exmp}
If $\cF$ is a $\bT$-equivariant $\bR$-test configuration. Then $\cF_\xi$ is also an $\bR$-test configuration.

If $\cF=\cF_{(\mcX, \mcL, a\eta)}$ for a test configuration, then we can identify the data $\cF_\xi$ with the data $(\mcX, \mcL, a\eta+\xi)$ (see \cite{His16b}).
If $\xi\in N_\bZ$, then $(\mcX, \mcL, a\eta)$ equivalent to the birational image of the $(\mcX, \mcL)$ via the birational transform $\sigma_\xi: \mcX\dasharrow \mcX$, $(z, t)\rightarrow (\sigma_\xi(t)\cdot z, t)$ (see \cite{Li19}).

Moreover if we start with the trivial filtration $\cF_\triv=\cF_{(X_\bC, (-K_{X})_\bC, t\partial_t)}$, then $(\cF_\triv)_\xi$ is equal to $\cF_{\wt_\xi}$.

\end{exmp}

\begin{defn}\label{def-adapt}
We say that a faithful valuation $\fv$ in the sense of Definition \ref{def-good} is adapted to the torus action if for any $f\in \bC(X)_\alpha$ we have:
$\fv(f)=(\alpha, \fv^{r+1}(f),\dots, \fv^{n}(f))\in \bZ^r\times\bZ^{n-r}$. 

\end{defn}

There always exists a faithful valuation that is adapted to the torus action. This can be constructed as follows. First we choose an $\bT$-invariant Zariski open set $U$ of $X$ as in \cite{AHS08}. Then by the theory of affine $T$-varieties as developed in \cite{AH06}. There exists a variety $Y$ of dimension $n-r$ and a polyhedral divisor $\mathfrak{D}$ such that:
\begin{equation}
U={\rm Spec}_{\alpha\in M_\bZ}H^0(Y, \mcO(\mathfrak{D}(\alpha))).
\end{equation}
We can choose a faithful valuation $\fv_Y$ on $Y$ (for example via a flag of varieties as in \cite{LM09}) and define: for any $f\in H^0(Y, \mcO(\mathfrak{D}(\alpha))$:
\begin{equation}
\fv(f)=(\alpha, \fv_Y(f)).
\end{equation}
Let $\fv$ be such a valuation and $\Delta=\Delta_\fv(X, -K_X)\subset \bR^n$ be the associated Newton-Okounkov body. If $p: \bR^n=\bR^r\times\bR^{n-r}\rightarrow \bR^r$ denote the natural projection, then we have
\begin{equation}\label{eq-proj}
p(\Delta)=P=\text{ moment map of the $\bT$-action on $(X, -K_X)$ }.
\end{equation}

The following lemma was already observed in \cite{Yao19} in which a faithful valuation adapted to the torus action was constructed using equivariant infinitesimal flags in the sense of \cite{LM09}. Here we give a different and direct proof for the reader's convenience.

For simplicity of notations, we denote $y=(y_1,\dots, y_n)=(y', y'')\in \bR^r\times \bR^{n-r}$ and set:
\begin{equation}\label{eq-yxipair}
\la y', \xi\ra=\sum_{i=1}^r y'_i \xi^i=:\la y, \xi\ra. 
\end{equation}
In the last identity, we identify $\xi\in N_\bR=\bR^r$ with $(\xi, 0)\in \bR^n$. 
\begin{lem}[see \cite{Yao19}]\label{lem-Gtwist}
If $\fv$ is a $\bZ^n$-valued valuation adapted to the torus action, then for any $y\in \Delta(-K_X)$:
\begin{equation}\label{eq-Gtwist}
G_{\cF_\xi}(y)=G_{\cF}(y)+\la y', \xi\ra.
\end{equation}
\end{lem}
\begin{proof}
For any $t> G^\cF(y)=\lambda$, there exists $\epsilon>0$ such that $y\not\in \Delta(\cF^{(t-\epsilon)})$. Let $\delta_1={\rm dist}(y, \Delta(\cF^{(t-\epsilon)}))$.

Choose any $f\in \cF^{(t+\la y', \xi\ra) m}_\xi R_{m,\alpha}=\cF^{(t+\la y', \xi\ra)m-\la \alpha, \xi\ra}R_{m,\alpha}$. Consider two cases:
\begin{enumerate}
\item $\la \frac{\alpha}{m}, \xi\ra-\la y', \xi\ra<\epsilon$. Then $\fv(f)\in \Delta^{(t-\epsilon)}$. So $|\frac{\fv(f)}{m}-y|\ge \delta_1$.

\item $\la \frac{\alpha}{m}, \xi\ra-\la y', \xi \ra\ge \epsilon$. Then $|\frac{\fv(f)}{m}-y|\ge |\frac{\alpha}{m}-y'|\ge \frac{\epsilon}{|\xi|}$.

\end{enumerate}
The two cases together imply that $y\not\in \Delta(\cF^{(t+\la y',\xi\ra)}_\xi)$. So we get the inequality $G_{\cF_\xi}\le G_{\cF}+\la y', \xi\ra$. 

On the other hand, since $\cF=(\cF_\xi)_{-\xi}$, we also get $G_{\cF}\le G_{\cF_\xi}-\la y', \xi\ra$. So we get the wanted identity.

\end{proof}

\subsection{Equivalent filtrations}
In this section we recall Boucksom-Jonsson's characterization of equivalent filtrations in \cite{BoJ18b}.

For a filtration $\cF R_m$ of $R_m$, we say that a basis $\cB=\{s_1,\dots, s_{N_m}\}$ of $R_m$ is adapted to $\cF R_m$ if for any $\lambda\in \bR$ there exists a subset of $\cB$ that spans $\cF^\lambda R_m$.

Let $\cF_i=\{\cF_i R_m\}, i=0,1$ be two filtrations. For each $m$, we can find a common basis $\cB:=\{s_1,\dots, s_{N_m}\}$ of $R_m$ that is adapted to both $\cF_i R_m, i=0,1$ (see \cite{BE18} and the discussion in section \ref{sec-initial}).  
Assume that for each $i=0,1$, we have $s_k\in \cF_i^{\mu_{k,i}}\setminus \cF^{>\mu_{k,i}}$. Then $\cB$ is an orthogonal basis for the non-Archimedean norm $\|\cdot \|_{m,i}$ corresponding to $\cF_i$: for any $s=\sum_k a_k s_k\in R_m$
\begin{equation}
\|s\|_{m,i}=e^{-\max\{\lambda;\; s\in \mcF^\lambda R_m\}}=\max_k |a_k|_0 e^{-\mu_{k,i}}
\end{equation}
where $|\cdot |_0$ is the trivial norm on $\bC$.

Following \cite{CM15, BoJ18b}, we define the set of successive minima of $\cF_1$ with respect to $\cF_0$ is the set $\{\mu_{k,1}-\mu_{k,0}\}$. The following result was proved in \cite{BE18, CM15}.
\begin{thm}[\cite{BE18, CM15}]
As $m\rightarrow+\infty$, the measures 
\begin{equation}
\frac{n!}{m^n}\sum_{k=1}^{N_m}\delta_{\frac{\mu_{k,1}-\mu_{k,0}}{m}}
\end{equation}
converge weakly as $m\rightarrow+\infty$ to a relative limit measure denoted by $d\nu:=d\nu(\cF_0, \cF_1)$.
\end{thm}
\begin{cor}
For any $p\in [1, \infty)$, the limit
\begin{equation}
d_p(\cF_0, \cF_1):=\lim_{m\rightarrow+\infty}\left(\frac{n!}{m^n}\sum_{k=1}^{N_m}m^{-1}|\mu_{k,1}-\mu_{k,0}|^p\right)^{1/p}
\end{equation}
exists and is given by:
\begin{equation}
d_p(\cF_0, \cF_1)=\left(\int_\bR |\lambda|^p d\nu(\lambda)\right)^{1/p}.
\end{equation}
\end{cor}

\begin{defn}[{\cite[section 3.6]{BoJ18b}}]\label{def-Fequiv}
$\cF_0$ and $\cF_1$ are equivalent if $d_2(\cF_0, \cF_1)=0$.
\end{defn}
In fact, by \cite{BoJ18b} $d_p$ are comparable to each other for all $p\in [1, \infty)$, and the above equivalence can be defined by using any $p\in [1, +\infty)$.

The key result we will need from \cite{BoJ18b} is:
\begin{thm}[{\cite[Theorem 4.16]{BoJ18b}}]\label{thm-equiv}
Let $\cF_0$ and $\cF_1$ be two filtrations on $R$. Then $\cF_0$ and $\cF_1$ are equivalent if and only if $\phi_{\cF_1}=\phi_{\cF_2}$.
\end{thm}
We also need:
\begin{prop}\label{prop-equiv}
Moreover if $\cF_{v_i}, i=0,1$ are two $\bR$-test configuration associated to two valuations $v_i\in \Val(X), i=0,1$. Then $\phi_{\cF_{v_1}}=\phi_{\cF_{v_2}}+c$ for a constant $c\in \bR$ if and only if $v_1=v_2$ (and hence $\cF_{v_1}=\cF_{v_2}$).
\end{prop}
This statement follows from $\MA^\NA(\phi_{\cF_{v_i}}+c)=\delta_{v_i}$ (for any $c\in \bR$) by \cite[Theorem 5.13]{BoJ18b}. We can give a direct proof here.
\begin{proof}
Recall that $\phi_{\cF_{v_i}}=\lim_{m\rightarrow+\infty}\phi^{\cF_{v_i}}_{m}$ is an increasing limit along the subsequence $m=2^{k}$, where for any $w\in \Val(X)$,
\begin{equation}
\phi^{\cF_{v_i}}_{m}(w)=-\frac{1}{m}G(w)\left(\sum_{\lambda\in \bN} I^{\cF_{v_i}}_{m,\lambda}t^{-\lambda}\right)
\end{equation}
where $I^{\cF_{v_i}}_{m,\lambda}$ is the base ideal of the sublinear system $\cF_{v_i}^{\lambda} R_m$. 
Note that it is easy to see that $v_1=v_2$ if and only if $\fa_\lambda(v_1)=\fa_\lambda(v_2)$ for any $\lambda\in \bN$, where $\fa_\lambda(v_i)=\{f\in \mcO_X; v_i(f)\ge \lambda\}$. 

For any $d\in \bN$, by choosing $m\gg 1$, we can assume that $m L\otimes \fa_d(v_1)$ is globally generated. Then we get $I^{\cF_{v_1}}_{m, d}=\fa_d(v_1)$. From this it is also clear that that $\phi_{\cF_{v_i}}(v_i)=0$. So we get:
\begin{eqnarray*}
-c&=&-\phi_{\cF_{v_1}}(v_2)\le -\phi^{\cF_{v_1}}_{2^k}(v_2)= \frac{1}{2^k}G(v_2)(\sum_\lambda I^{\cF_{v_1}}_{2^k, \lambda}t^{-\lambda})\\
&\le& \frac{1}{2^k}(v_2(I^{\cF_{v_1}}_{2^k, d})-d)=\frac{1}{2^k}(v_2(\fa_d(v_1))-d). 
\end{eqnarray*}
Since $k$ can be arbitrarily large, we get $-c\le 0$, i.e. $c\ge 0$. Switching $v_1$ and $v_2$ in the above argument, we get $c\le 0$. So $c=0$. We then have the inequality $v_2(\fa_d(v_1))\ge d$ for any $d\in \bN$. This easily implies $v_2\ge v_1$. Switching $v_1$ and $v_2$, we get $v_1\le v_2$. Hence $v_1=v_2$ as required.
\end{proof}

\subsection{Non-Archimedean invariants of filtrations}\label{sec-NAfunctional}
%In this subsection, we discuss the t of filtrations. the $\bfH^\NA$-invariant of $\bR$-test configurations defined by Dervan-Sz\'{e}kelyhidi to filtrations. 
For any filtration $\cF$ on $R=R(X, -K_X)$, we set:
\begin{eqnarray}
\bfL^\NA(\phi^\cF)=\bfL^\NA(\cF)&=&\bfL^\NA_X(\cF)=\inf_{v\in X^{\rm div}_\bQ} (A_X(v)+(\phi_\cF-\phi_\triv)(v)) \label{eq-defLNA}\\
\btS^\NA(\phi^\cF)=\btS^\NA(\cF)&=&\btS^\NA_X(\cF)=-\log\left(\frac{1}{\bfV}\int_\bR e^{-\lambda} \DHM(\cF)\right) \nonumber \\
&=&-\log \left(\frac{n!}{\bfV}\int_\Delta e^{-G_\cF(y)}dy\right), \label{eq-defSNA}\\
\bfE^\NA(\phi^\cF)=\bfE^\NA(\cF)&=&\bfE^\NA_X(\cF)=\frac{1}{\bfV}\int_\bR \lambda\cdot \DHM(\cF)=\frac{n!}{\bfV}\int_\Delta G_\cF(y) dy,\\
\bfH^\NA(\phi^\cF)=\bfH^\NA(\cF)&=&\bfH^\NA_X(\cF)=\bfL^\NA(\cF)-\btS^\NA(\cF), \label{eq-defbtD}\\
\bfD^\NA(\phi^\cF)=\bfD^\NA(\cF)&=&\bfD^\NA_X(\cF)=\bfL^\NA(\cF)-\bfE^\NA(\cF).
\end{eqnarray}
The above functionals are by-now well-known and we use the notations following those in \cite{BHJ17, His19a}. The formula involving $G_\cF$ follows from Theorem \ref{thm-BoCh}.2.

\begin{prop}[{see \cite{Fuj19a, BoJ18b}}]\label{prop-approx}
For a filtration $\cF$, with the notations from Definition \ref{defn-phiF}, we have the convergence:  from Definition \ref{defn-ckTCm} satisfies, for any $\bfF\in \{\btS, \bfE\}$:
\begin{equation}
\lim_{m\rightarrow+\infty} \bfF^\NA(\phi^\cF_m)=\bfF^\NA(\phi^\cF).
\end{equation}
Moreover, we have:
\begin{equation}\label{eq-Lusc}
\lim_{m\rightarrow+\infty}\bfL^\NA(\phi^\cF_m)\le \bfL^\NA(\phi^\cF).
\end{equation}
\end{prop}
\begin{proof}
By Proposition \ref{prop-FBJ} we know that $\DHM(\cF_{(\mcX_m, \mcL_m, \eta_m)})$ converges weakly to $\DHM(\cF)$ as $m\rightarrow+\infty$, from this we can easily get the convergence of $\btS^\NA$ and $\bfE^\NA$. 

The inequality \eqref{eq-Lusc} follows easily from the inequality $\phi^\cF_m\le \phi^\cF$. %, or the upper semicontinuity of $\bfL^\NA$ under increasing convergence (even strong convergence) of non-Archimedean metrics (see \cite[Lemma 2.7]{BoJ18b}).
%Because $\phi^{\cF}_m\le \phi^\cF$, it is easy to see that $\bfL^\NA(\phi^{\cF}_m)\le \bfL^\NA(\phi^\cF)$. The other direction of inequality follows from the lower semicontinuity of $\bfL^\NA$ under an increasing convergence of non-Archimedean metrics.
%\begin{equation}
%\bfL^\NA(\cF)=\lct(X\times\bC, \tilde{\mcI}^\cF_\bullet; (t))=\lim_{m\rightarrow+\infty}\lct(X\times\bC, \tilde{\mcI}_m^{1/m}; (t))
%\end{equation}
% \cite{Fuj19a} (see also \cite{BBJ18}), we know that $\bfL^\NA(\mcX_m, \mcL_m, \eta_m)$ converges to $\bfL(\cF)$ as $m\rightarrow+\infty$.

\end{proof}

\begin{rem}
When $\cF=\cF_v$ for $v\in \rVal(X)$, one can show that \eqref{eq-Lusc} is an equality for the increasing subsequence $\{\phi^{\cF}_{m!}\}_{m\in \bN}$. Indeed, in this case, $\phi^{\cF}-\phi_\triv$ is continuous on the non-Archimedean space $X^\NA$ (by \cite{BoJ18b}) and the sequence $\{\phi^\cF_{m!}\}$ converges to $\phi^\cF$ uniformly as $m\rightarrow+\infty$ by Dini's theorem, which implies the equality. Alternatively, one could use the method in \cite{JM12} to prove the convergence. 
\end{rem}
%By the uniform convergence of $\check{\phi}_m$ to $\phi$, we have the convergence (see \cite[Definition 2.6]{Fuj18}):
%\begin{eqnarray}\label{eq-LNAlim}
%\bfL^\NA(\cF)+1&=&\lct(X\times\bC, \tilde{\mcI}^\cF_\bullet; (t))-1:=\lim_{m\rightarrow+\infty}\lct(X\times\bC, (\mcI^\cF_m)^{1/m}; (t))-1\nonumber\\
%&=&\lim_{m\rightarrow+\infty}\bfL^\NA(\ccX_m, \ccL_m)
%\end{eqnarray}
%where $\tilde{\mcI}^\cF_\bullet=\{\tilde{\mcI}^\cF_m\}$ is the graded sequence of flag fractional ideal defined in \eqref{eq-tcIm}. 
We will also use a variant of $\bfL^\NA$ functional that was introduced by Xu-Zhuang \cite{XZ19}:
\begin{equation}\label{eq-hatbfL}
\hat{\bfL}^\NA(\cF)=\sup\left\{x\in \bR; \lct(X; I^{\cF^{(x)}}_\bullet)\ge 1\right\},
\end{equation}
where $I^{\cF^{(x)}}_\bullet=\{I^\cF_{m,mx}\}$ is a graded sequence of base ideals defined in  \eqref{eq-Imx}. 
For later use, we also denote: 
\begin{equation}
\hat{\bfH}^\NA(\cF)=\bhL^\NA(\cF)-\btS^\NA(\cF).
\end{equation}
The following comparison estimates were proved by Xu-Zhuang.
\begin{prop}[{\cite[Proposition 4.2, Theorem 4.3]{XZ19}}]\label{prop-XZcomp}
\begin{enumerate}
\item For any prime divisor $E$ over $X$, $A_X(E)\ge \bhL^\NA(\cF_{\ord_E})$ with equality holds true if $\ord_E$ induces a weakly special test configuration.
\item For any filtration $\cF$, $\bhL^\NA(\cF)\ge \bfL^\NA(\cF)$.
\end{enumerate}
\end{prop}

For later purposes, we also introduce (for any $a>0$):
\begin{eqnarray}
\bfE_k^\NA(\cF)&=&\frac{1}{\bfV}\int_\bR x^k \DHM(\cF)=\lim_{m\rightarrow+\infty} \frac{1}{N_m}\sum_{i}\left(\frac{\lambda^{(m)}_i}{m}\right)^k \label{eq-Ek}\\
\bfQ^{(a)}(\cF)&=&\frac{1}{\bfV}\int_\bR e^{-ax} \DHM(\cF)=\frac{1}{\bfV}\sum_{k=0}^{+\infty}\frac{(-1)^k}{k!}a^k \bfE_k^\NA(\cF) \label{eq-bVa}\\
\bfQ(\cF)&:=&\bfQ^{(1)}(\cF).
\end{eqnarray}
Note that $\bfE_1^\NA(\cF)=\bfE^\NA(\cF)$ and $\btS^\NA(\cF)=-\log\bfQ(\cF)$.

For any $v\in \rVal(X)$ (resp. $(\mcX, \mcL, a\eta)$ a test configuration), we will often write $\bfF^\NA(v)$ (resp. $\bfF^\NA(\mcX, \mcL, a\eta)$) for the above various functionals $\bfF^\NA(\cF)$ with $\cF$ being the corresponding filtration. 

\begin{exmp}
If $(\mcX, \mcL, a\eta)$ is a normal test configuration, then we have:
\begin{eqnarray}
\bfE^\NA(\mcX, \mcL, a\eta)&=&a\cdot \frac{\bar{\mcL}^{\cdot n+1}}{(n+1)\bfV}\\
\bfL^\NA(\mcX, \mcL, a\eta)&=&a\cdot\left( \lct(\mcX, -K_{\mcX}-\mcL; \mcX_0)-1\right).
\end{eqnarray}
If $(\mcX, \mcX_0)$ has log canonical singularities and $K_\mcX+\mcL=\sum_i e_i E_i$ (which is centered at $\mcX_0$), then
\begin{equation}
\bfL^\NA(\mcX, \mcL, a\eta)=a \cdot \min_i e_i.
\end{equation}
\end{exmp}
\begin{exmp}
Let $\cF$ be a special $\bR$-test configuration and let $(X_0, \eta)=(X_{\cF,0}, \eta_\cF)$ be the corresponding central fibre. 
Assume that $\cF|_{X_0}=\cF'_{\wt_\eta}R'(-\sigma)$ (see \eqref{eq-RspecialcF}). Let $\tilde{\vphi}$ be any $(S^1)^r$-invariant smooth positively curved Hermitian metric on $-K_X$. 
 Then with the notations as in the paragraph containing \eqref{eq-thetaeta}, we have:
\begin{eqnarray}
%\bfL^\NA_{X_0}(\cF)&=&\bfL^\NA_{X_0}(\cF|_{X_0})=\frac{\int_{X_0}\theta_{\tilde{\vphi}}(\eta)e^{-\tilde{\vphi}}}{\int_{X_0}e^{-\tilde{\vphi}}}=0\\
\bfL^\NA_X(\cF)&=&\bfL^\NA_{X_0}(\cF|_{X_0})=
\frac{\int_{X_0}\theta_{\tilde{\vphi}}(\eta) e^{-\tilde{\vphi}}}{\int_{X_0}e^{-\tilde{\vphi}}}-\sigma=-\sigma\\
\bfE^\NA_X(\cF)&=&\bfE^\NA_{X_0}(\cF|_{X_0})=\frac{1}{\bfV}\int_{X_0}\theta_{\tilde{\vphi}}(\eta) (\ddc\tilde{\vphi})^n-\sigma\\
\btS^\NA_X(\cF)&=&\btS^\NA_{X_0}(\cF|_{X_0})=-\log\left(\frac{1}{\bfV}\int_{X_0} e^{-\theta_\eta}(\ddc\tilde{\vphi})^n\right)-\sigma.
\end{eqnarray}
The above identity is well-known if $\cF$ comes from a special test configuration. For more general $\cF$, one can use a sequence of special test configuration to approximate and get the above formula.
\end{exmp}

Corresponding to \eqref{eq-phiab}, we have the following simple transformation rule:
\begin{lem}\label{lem-abtrans}
For any $(a,b)\in \bR_{>0}\times \bR$, 
\begin{eqnarray}
\bfL^\NA(a\cF(b))&=&a \bfL^\NA(\cF)+b\\
\bhL^\NA(a\cF(b))&=&a \bhL^\NA(\cF)+b\\
\btS^\NA(\cF(b))&=&\btS^\NA(\cF)+b\\
\bfH^\NA(\cF(b))&=&\bfH^\NA(\cF).
\end{eqnarray}
\end{lem}
We will also use the following fact.
\begin{lem}
Let $\cF$ be a $\bT$-equivariant filtration.
For any $\xi\in N_\bR$, we have:%the following identities hold true:
\begin{eqnarray}\label{eq-Ltwist}
\bL^\NA(\mcF_\xi)&=&\bL^\NA(\mcF), \quad \bhL^\NA(\cF_{\xi})=\bhL^\NA(\cF). \label{eq-bLFtwist} %\\
%\bfE^\NA(\mcF_\xi)&=&\bfE^\NA(\mcF)-\Fut_{(Z,Q)}(\xi); \label{eq-EFtwist}. %\\
%\bfD^\NA(\mcF_\xi)&=&\bfD^\NA(\mcF)-\Fut_{(Z,Q)}(\xi).\label{eq-DFtwist}
\end{eqnarray}
As a consequence, we have 
\begin{equation}\label{eq-Lwtwist}
\bfL^\NA(\mcF_{\wt_\xi})=\bhL^\NA(\cF_{\wt_\xi})=0
\end{equation}
%In particular, if $\Fut_{(Z,Q)}\equiv 0$, then $\bfE^\NA(\mcF_\xi)=\bfE^\NA(\mcF)$ and $\bfD^\NA(\mcF_\xi)=\bfD^\NA(\mcF)$.
\end{lem}
\begin{proof}
The first identity in \eqref{eq-Ltwist} is proved in {\cite[Lemma 3.10]{Li19}}. 
Let's prove the second identity using the similar argument as there. By the $\bT$-equivariance, we have the decomposition $I^{\cF}_{m,mx}=\sum_\alpha (I^{\cF}_{m,mx})_\alpha$ and similarly for $I^{\cF_\xi}_{m,mx}$.
By the definition of $\cF_\xi$, we get $(I^{\cF_\xi}_{m,m x})_\alpha=\left(I^{\cF}_{m,mx-\la \alpha, \xi \ra}\right)_\alpha$.  For any $w\in \Val(X)^\bT$, we use the same argument as \cite[Proof of Lemma 3.10]{Li19} to get:
\begin{eqnarray*}
w(I^{\cF_\xi}_{m,mx})&=&\frac{1}{m}\min_\alpha w((I^{\cF_\xi}_{m,mx})_\alpha)=\frac{1}{m}\min_\alpha w((I^{\cF}_{m,mx-\la \alpha,\xi\ra})_\alpha)\\
&=&-A_X(w_\xi)+A_X(w)+\frac{1}{m}\min_\alpha w_\xi((I^\cF_{m,mx})_\alpha)\\
&=&-A_X(w_\xi)+A_X(w)+\frac{1}{m}w_\xi(I^\cF_{m,mx}).
\end{eqnarray*}
From this we easily get the identity:
\begin{equation}
A_X(w)-w\left(I^{\cF_\xi^{(x)}}_\bullet\right)=A_X(w_\xi)-w_\xi\left(I^{\cF^{(x)}}_\bullet\right).
\end{equation}
Assume that $x=\bhL^\NA(\cF)$. Then for any $\epsilon>0$, we have $c:=\lct(X; I^{\cF^{(x+\epsilon)}}_\bullet)< 1$. So there exists $v\in \Val(X)^\bT$ such that
\begin{eqnarray*}
0>A_X(v)-v(I^{\cF^{(x)}}_\bullet)=A_X(v_{-\xi})-v_{-\xi}\left(I^{\cF_{\xi}^{(x)}}_\bullet\right).
\end{eqnarray*}
This implies $\lct(X; I^{\cF_\xi^{(x+\epsilon)}}_\bullet)<1$. So we get $x+\epsilon \ge \bhL^\NA(\cF_\xi)$ and hence $\bhL^\NA(\cF)\ge \bhL^\NA(\cF_\xi)$ since $\epsilon>0$ is arbitrary. Because $(\cF_\xi)_{-\xi}=\cF$, repeating the above argument for $\cF_\xi$, we get the other direction of inequality, which proves the second identity in \eqref{eq-Ltwist}.

Using \eqref{eq-Ltwist}, the identity \eqref{eq-Lwtwist} follows from the facts $\cF_{\wt_\xi}=(\cF_{\triv})_\xi$ and $\bfL^\NA(\cF_\triv)=0=\bhL^\NA(\cF_\triv)$.

\end{proof}

The following lemma is a prototype uniqueness result in this paper, and can be seen as a generalization of the uniqueness of K\"{a}hler-Ricci soliton vector fields shown by Tian-Zhu \cite{TZ02} (the case when $\cF=\cF_\triv$). 
See section \ref{sec-soliton} for more discussion.
\begin{lem}\label{lem-twistconvex}
Let $\cF$ be a $\bT$-equivariant filtration. Then the function $\xi\mapsto \bfH^\NA(\cF_\xi)$ on $N_\bR$ admits a unique minimizer.
\end{lem}
\begin{proof}
By \eqref{eq-Ltwist}, $\bfL^\NA(\cF_\xi)$ is constant in $\xi$. Using the identity \eqref{eq-Gtwist} and \eqref{eq-defSNA},  
\begin{eqnarray*}
-\btS^\NA(\cF_\xi)&=&\log\left(\frac{n!}{\bfV}\int_\Delta e^{-G_{\cF_\xi}(y)} dy\right)=\log\left(\frac{n!}{\bfV}\int_\Delta e^{-G_\cF(y)-\la y, \xi\ra} dy\right).
\end{eqnarray*}
It is easy to use this expression to show that $f(\xi):=-\btS^\NA(\cF_\xi)$ is strictly convex in $\xi\in N_\bR$, which implies the uniqueness of minimizer. To prove the existence of minimizer, we need to show that $f(\xi)$ is proper, i.e. 
$\lim_{|\xi|\rightarrow+\infty}f(\xi)=+\infty$. 
%This follows from the fact that $0$ is in the interior of the moment polytope $P$ of the torus action on $(X, -K_X)$ which is the image of $\Delta$ under the projection $\bR^n\rightarrow \bR^r$. 
To see this recall that we have the vanishing
\begin{equation}
 \int_X \theta_{\tilde{\vphi}}(\xi) e^{-\tilde{\vphi}}=-\int_X \mathfrak{L}_\eta e^{-\tilde{\vphi}}=0.
 \end{equation}
This implies $0>\inf_X \theta_{\tilde{\vphi}}(\xi)=\inf_\Delta \la y, \xi\ra$ if $\xi\neq 0$, which indeed implies the properness.
 
\end{proof}

%The following notion will be useful for us:
\begin{defn}\label{def-normfil}
We say that a filtration $\cF$ is normalized if 
\begin{equation}
\bfL^\NA(\cF)=0.
\end{equation}
A test configuration $(\mcX, \mcL, a\eta)$ is normalized if $\cF_{(\mcX, \mcL, a\eta)}$ is normalized. 
\end{defn}
With the above discussion, the following lemma is easy to prove.
\begin{lem}
\begin{enumerate}
\item
Any special test configuration $(\mcX, -K_{\mcX})$ is normalized. More generally, a special $\bR$-test configuration $\cF$ (see Definition \ref{def-RTC}) if and only if $\sigma=0$ in \eqref{eq-RspecialcF}. 
\item

For any filtration $\cF$, the shift $\cF(-\bfL^\NA(\cF))$ is normalized. If $\cF$ is normalized, then so are $a\cF$ for any $a>0$, and any twist $\cF_\xi$.
\end{enumerate}
\end{lem}
%Recall that
%\begin{equation}
%\DHM=p_*(\omega^n)=\lim_{m\rightarrow+\infty} \frac{n!}{m^n}\sum_i \delta_{\frac{\lambda^{(m)}_i}{m}}.
%\end{equation}
%satisfies $\int_\bR \DHM=\int_{\mcX_0} \omega^n=V$.

As a consequence of this approximation result in Proposition \ref{prop-approx}, it is convenient for us to introduce:
\begin{defn-prop}
For any $\bQ$-Fano variety $X$, we define
\begin{equation}\label{eq-defhX}
h(X)=\inf_{(\mcX, \mcL, a\eta)}\bfH^\NA(\mcX, \mcL, a\eta)=\inf_{\cF} \bfH^\NA(\cF)
\end{equation}
where $(\mcX, \mcL, a\eta)$ ranges over all test configurations and $\cF$ ranges over all filtrations or $\bR$-test configurations.
\end{defn-prop}

The following lemma is similar to \cite[Lemma 2.5]{DS16}. %, follows from the . 
\begin{lem}
For any filtration $\cF$, we have:
\begin{eqnarray}\label{eq-SvsE}
\btS^\NA(\cF)\le \bfE^\NA(\cF), \quad \bfH^\NA(\cF)\ge \bfD^\NA(\cF).
\end{eqnarray}
The identity holds true if and only if $\cF(c)$ is equivalent to the trivial filtration for some $c\in \bR$ (see Definition \ref{def-Fequiv}).
%If $X$ is $\bfD$-semistable, then $X$ is $\btD$-semistable.
\end{lem}
\begin{proof}
The first inequality, which implies the second, follows from the concavity of logarithmic function. When the identity holds, the DH measure $\DHM(\cF)$ is a Dirac measure $V\cdot \delta_c$. $d_2(\cF(c), \cF_\triv)=0$ which by Definition \ref{def-Fequiv} means that $\cF(c)$ is equivalent to the trivial filtration. 
\end{proof}

Based on the work in \cite{CW14, CSW18, He16}, Dervan-Sz\'{e}kelyhidi proved:
\begin{thm}[\cite{DS16}]
Assume that $X$ is a smooth Fano manifold. There is an identity:
\begin{equation}\label{eq-idenKR}
h(X):=\inf_{(\mcX, \mcL, a\eta) \text{special}}\bfH^\NA(\mcX, \mcL, a\eta)=-\inf_{\omega\in c_1(X)}\int_X h_{\omega}e^{h_\omega}\omega^n
\end{equation}
where $\omega$ ranges over smooth K\"{a}hler metrics from $c_1(X)$ and $h_\omega$ is the normalized Ricci potential of $\omega$.
Moreover the infimum is achieved by a special test configuration constructed via the Gromov-Hausdorff limit K\"{a}hler-Ricci soliton from \cite{CW14, CSW18}.
\end{thm}

More recently Hisamoto \cite{His19a} gave a different proof of \eqref{eq-idenKR} based on the destabilising geodesic rays constructed from \cite{DH17}. 

\begin{rem}\label{rem-DS}
Note that our sign convention differs from that Dervan-Sz\'{e}kelyhidi and Hisamoto by a minus. 
Dervan-Sz\'{e}kelyhidi defined a non-Archimedean functional for general $\bR$-test configuration by mimicking Tian's CM weight (or so-called Donaldson-Futaki invariant). But in such generality, their normalization seems not precise. Different with their definition, for any test configuration $(\mcX, \mcL, a\eta)$, one could define:
\begin{equation}
\tilde{\bfH}^\NA(\mcX, \mcL, a\eta)=\frac{a}{V}\left(K_{\bcX/\bP^1}\cdot \bcL^{\cdot n}+\bcL^{\cdot n+1}\right)-\btS^\NA(\mcX, \mcL, a\eta).
\end{equation}
%For completeness, we make some comment of the $\bfH^\NA$ invariant defined by Dervan-Sz\'{e}kelyhidi (\cite{DS16}) for special test configurations, which is just the invariant in \cite{TZZZ} on the central fibre. 
By the same argument as \cite[Proposition 7.32]{BHJ17}, we have:
\begin{equation}
\bfH^\NA(\mcX, \mcL, \eta)\le \tilde{\bfH}^\NA(\mcX, \mcL, \eta)
\end{equation}
with the inequality if $(\mcX, \mcL, \eta)$ is anticanonical.
%The inequality follows from the following inequality:
%\begin{eqnarray*}
%\frac{1}{\bfV}\left(K_{\bar{\mcX}/\bC}\cdot \bar{\mcL}^{\cdot n}+\bcL^{\cdot n+1}\right)&=& \frac{1}{\bfV}(K_{\bar{\mcX}}-\rho^*K_{X_\bC})\cdot \bar{\mcL}^{\cdot n}+\frac{1}{\bfV}(\rho^*K_X+\mcL)\cdot \bar{\mcL}^{\cdot n}\\
%&\ge&\frac{1}{\bfV}(K^{\log}_{\bar{\mcX}}-\rho^*K_{X_\bC})\cdot \bar{\mcL}^{\cdot n}+\frac{1}{\bfV}(\mcL-\rho^*(-K_{X_\bC}))\cdot \bcL^{\cdot n}\\
%&=&\frac{1}{\bfV}\sum_E b_E (\mcL^{\cdot n}\cdot E) (A_X(v_E)+(\phi-\phi_\triv)(v_E))\\
%&\ge& \inf_{v} \left(A_X(v)+(\phi-\phi_\triv)\right)=\bfL^\NA(\mcX, \mcL, a\eta).
%\end{eqnarray*}
%Note that for anticanonical (in particular for special) test configuration, we have:
%\begin{equation}
%\bfH^\NA(\mcX, -K_{\mcX}, \eta)=\bfH^\NA(\mcX, -K_{\mcX}, \eta)=-\btS^\NA(\mcX, -K_{\mcX}, \eta).
%\end{equation}
%See Theorem \ref{thm-MMP} for a similar result for any $\bQ$-Fano variety.
Moreover, by \eqref{eq-SvsE} we also get
\begin{equation}
\tilde{\bfH}^\NA(\mcX, \mcL, \eta)\ge {\rm CM}(\mcX, \mcL, \eta)=\frac{1}{\bfV}\left(K_{\bcX/\bP^1}\cdot \bcL^{n}+\frac{n}{n+1}\bcL^{\cdot n+1}\right) 
\end{equation}
with the identity being true only if $(\mcX, \mcL, \eta)$ is trivial. 
One advantage of $\bfH^\NA$ over $\tilde{\bfH}^\NA$ is that the former can be defined for any filtration, not necessarily finitely generated. Due to this reason, we will not use $\tilde{\bfH}^\NA$ in this paper. 
\end{rem}

%\subsection{Case of metric tangent cones}
 %In the case of metric tangent cones, we prove the above conjecture naturally in two steps. First we prove that the K-semistable cone is unique by proving the uniqueness of regular minimizers of normalized volumes. Second we prove that the degeneration of K-semistable cone to K-polystable cone is unique. 

 \subsection{$g$-Ding-stability and K\"{a}hler-Ricci solitons}\label{sec-soliton}
  
Let $\cF$ be a $\bT$-equivariant filtration. 
For any $\lambda\in \bR$, we have (finite) decomposition:
\begin{equation}
\cF^{\lambda} R_m=\bigoplus_{\alpha\in M_\bZ}\cF^{\lambda}R_{m,\alpha}. %\cF^{\lambda}R_{m,\alpha^{(\lambda)}_1}\oplus \cdots \cF^{\lambda}R_{m,\alpha^{(\lambda)}_{k_{m,\lambda}}}
\end{equation}
% and let $\{\lambda^{(m,\alpha)}_i\}$ be its successive minima. 
%Then similar to the discussion for usual filtrations, the measure:
%\begin{equation}
%\sum_{i,\alpha}\delta_{\frac{(\lambda^{(m,\alpha)}_i, \alpha)}{m}}
%\end{equation}
%converges weakly to a measure $\DHM_{\bT\times\bC^*}(\cF)$ on $\bR^{n+1}$. Let $p: \bR^{n+1}=\bR^n\times\bR\rightarrow \bR^n$ be the projection. 

Let $P$ be the  moment polytope of $(X, -K_X)$ with respect to the $\bT$-action. Let $g$ be a smooth positive function on $P$. 
Fix a faithful $\bZ^n$-valuation that is adapted to the torus action (see Definition \ref{def-adapt}) and let $\Delta\subset \bR^n$ be the Okounkov body that satisfies \eqref{eq-proj}:
$p(\Delta)=P$ where $p: \bR^n\rightarrow \bR^r$ the natural projection.
Still denote by $g$ the function $p^*g$ on $\Delta$.
% and $\DHM_g(\cF)=p_*(g(x)\DHM_{\bT\times\bC^*}(\cF))$. 
%Let $p: \Delta\rightarrow P\subset \bR^r$ be the projection onto the moment polytope of the $\bT$-action. 
%\begin{equation}
%\DHM_{\bT,g}(\cF)=p_*(g(x'')dx_{\rm Leb}).
%\end{equation}
Define the $g$-volume of graded linear series $\{\cF^{(t)}R_m\}$ as:
\begin{eqnarray*}
\vol_g(\cF^{(t)})&:=&\lim_{m\rightarrow+\infty} \sum_\alpha g(\frac{\alpha}{m})\frac{\dim \cF^{mt} R_{m,\alpha}}{m^n/n!}\\
&=&n!\cdot \int_{\Delta(\cF^{(t)})} g(y)dy_{\rm Leb}=:n!\cdot \vol_g(\Delta(\cF^{(t)})).
\end{eqnarray*}
Then as in the $g\equiv 1$ case, we have the convergence:
\begin{eqnarray*}
\DHM_g(\cF):=\lim_{m\rightarrow+\infty} \sum_{\alpha} g(\frac{\alpha}{m})
\delta_{\frac{\lambda^{(m,\alpha)}_i}{m}}&=&-d\vol_g(\cF^{(t)})\\
&=&n!\cdot (G_\cF)_*(g(y)dy_{\rm Leb})
\end{eqnarray*}

We also set:
\begin{eqnarray}
\bfV_g&:=&n!\cdot \vol_g(\Delta)=n!\cdot \int_\Delta g(y)dy_{\rm Leb}=\int_\bR \DHM_g(\cF)\\
\bfE_g^\NA(\cF)&:=&\frac{n!}{\bfV_g}\int_\Delta G_{\cF}(y) g(y)d y_{\rm Leb}=\frac{1}{\bfV_g}\int_{\bR} \lambda\cdot \DHM_g(\cF)\\
 \bfD_g^\NA(\cF)&:=&\bfL^\NA(\cF)-\bfE_g^\NA(\cF).
 \end{eqnarray}
 %\begin{eqnarray*}
%\bfE^\NA_g(\cF)=\lim_{m\rightarrow+\infty} \sum_i \frac{\lambda^{(m)}_i}{m}\sum_{\alpha} g(\frac{\alpha}{m})\frac{\dim \cF^{\lambda^{(m)}_i} R_{m,\alpha}}{m^n/n!}.
%\end{eqnarray*}
 If $(\mcX, \mcL, \eta)$ is a test configuration, then we set $\bfD_g^\NA(\mcX, \mcL, \eta)=\bfD_g^\NA(\cF_{(\mcX, \mcL, \eta)})$.
 We define:
 \begin{defn}
 $(X, \xi)$ is $g$-Ding-semistable if for any $\bT$-equivariant test configuration $(\mcX, \mcL, \eta)$ of $(X, -K_X)$, $\bfD_g^\NA(\mcX, \mcL, \eta)\ge 0$.
 
 $(X, \xi)$ is $g$-Ding-polystable if it is $g$-Ding-semistable and $\bfD_g^\NA(\mcX, \mcL, \eta)=0$ for a $\bT$-equivariant weakly special test configuration (see Definition \ref{def-TC}) only if $(\mcX, \mcL, \eta)$ is a product test configuration.
 \end{defn}
The following result was proved by adapting the techniques of MMP from \cite{LX14, Fuj19a, BBJ18}.
\begin{thm}[see \cite{HL20}]\label{thm-special}
To test the $g$-Ding-semistability, or $g$-Ding-polystability of $(X, \xi)$, it suffices to test over all special test configurations.
\end{thm}
We have the following valuative criterion:
\begin{thm}[\cite{HL20}]\label{thm-valgsemi}
$X$ is $g$-Ding-semistable if and only if for any $v\in (X^{\rm div}_\bQ)^\bT$, we have:
\begin{equation}\label{eq-betag}
\beta_g(v):=A_X(v)-\frac{1}{\bfV_g}\int_0^{+\infty} \vol_g(\cF_v^{(t)})dt\ge 0.
\end{equation}
\end{thm}
Now we use our notations to reformulate holomorphic invariants of Tian-Zhu \cite{TZ02} about the study of K\"{a}hler-Ricci solitons. 
We refer to \cite{TZ02, BWN, HL20} for more details and references. Let $X$ be a $\bQ$-Fano variety with an effective $\bT$-action. We use the same notations such as $(S^1)^r$-invariant smooth Hermitian metric $\tilde{\vphi}$ on $-K_X$, moment polytope $P\subset M_\bR$, $\theta_{\tilde{\vphi}}(\eta)=\frac{\mathfrak{L}_\eta e^{-\tilde{\vphi}}}{e^{-\tilde{\vphi}}}$ etc. We identity any $\eta\in N_\bR$ with the corresponding holomorphic vector field on $X$. 

A K\"{a}hler-Ricci soliton on $(X, \xi)$ is a positively curved bounded Hermitian metric $e^{-\vphi}$ on $-K_X$ that satisfies the equation:
\begin{equation}
e^{\vphi} (\ddc\vphi)^n=e^{\theta_{\vphi}(\xi)},
\end{equation} 
where $\theta_{\vphi}(\xi)=\theta_{\tilde{\vphi}}(\xi)+\xi(\vphi-\tilde{\vphi})$. Over $X^{\rm reg}$, $\vphi$ is smooth (see \cite{BBEGZ, HL20}) and satisfies the identity:
\begin{equation}\label{eq-KRflow}
Ric(\ddc\vphi)-\ddc\vphi=-\ddc\theta_{\vphi}(\xi).
\end{equation}
As a consequence, the family of metrics $\vphi(s):=\sigma_{\xi}(s)^*\vphi$ satisfies the normalized K\"{a}hler-Ricci flow:
\begin{equation}
\frac{d}{ds}\ddc \vphi(s)=-Ric(\ddc\vphi(s))+\ddc\vphi(s).
\end{equation}
For any $\xi\in N_\bR$, we set $g_\xi(x)=e^{-\la x, \xi\ra}=e^{-\sum_{i=1}^r \xi^i x_i}$ which is a smooth positive function on $P$ and write $\bfF_{g_\xi}$ as $\bfF_\xi$ for $\bfF\in \{\bfL, \bfD\} $ etc and $\bfV_\xi:=V_{g_\xi}$. 
Tian-Zhu \cite{TZ02} defined a modified Futaki invariant as an obstruction to the existence of K\"{a}hler-Ricci solitons on $(X, \xi)$: for any $\eta\in N_\bR$, 
\begin{equation}
\Fut_\xi(\eta):=-\frac{1}{\bfV_\xi}\int_X \theta_{\tilde{\vphi}}(\eta)e^{-\theta_{\tilde{\vphi}}(\xi)}(\ddc\tilde{\vphi})^n=\bfD^\NA_\xi(\wt_\eta),
\end{equation} 
where $\bfV_\xi=\int_X e^{-\theta_{\tilde{\vphi}}(\xi)}(\ddc\tilde{\vphi})^n$. The second identity can be obtained by noting that $\bfD^\NA_\xi(\wt_\eta)=-\bfE^\NA_\xi(\wt_\eta)$
because of the vanishing
$\bfL^\NA(\wt_\eta)=0. %\frac{\int_{\mcX_0}\theta_{\tilde{\vphi}}(\eta)e^{-\tilde{\vphi}}}{\int_{\mcX_0}e^{-\tilde{\vphi}}}=0.
$
\begin{rem}
Note again that we use the negative sign convention compared to \cite{TZ02}.
\end{rem}

$\Fut_\xi$ does not depend on the choice of $\tilde{\vphi}$ and $(X, \xi)$ admits a 
KR soliton only if $\Fut_\xi\equiv 0$ on $N_\bR$. Moreover, by \cite[Lemma 2.2]{TZ02} the soliton vector field is a priori uniquely determined by minimizing the strictly convex functional (Tian-Zhu didn't use the logarithm)
on $N_\bR$ (see Lemma \ref{lem-twistconvex}), which is the anti-derivative of $\eta\mapsto \Fut_\xi(\eta)$:
\begin{equation}
\xi \mapsto \log\left( \frac{1}{\bfV}\int_X e^{-\theta_{\tilde{\vphi}}(\xi)} (\ddc\tilde{\vphi})^n\right)=\log\left(\frac{1}{\bfV}\int_\bR e^{-\lambda}\DHM(\cF_{\wt_\xi})\right)=-\btS^\NA(\wt_\xi).
\end{equation}
%Note that we have the identity:
%\begin{eqnarray*}
%\frac{1}{\bfV}\int_X e^{-\theta_{\tilde{\vphi}}(\xi)}(\ddc\tilde{\vphi})^n&=&\frac{1}{\bfV}\int_\bR e^{-\la x', \xi\ra} \DHM_{\bT}\\
%&=&=e^{-\btS^\NA(\wt_\xi)},
%\end{eqnarray*}
%where $\DHM_\bT$ denotes the DH measure on $P$. 
Recall also that $\bfL^\NA(\wt_\eta)=\bhL^\NA(\wt_\eta)\equiv 0$ on $N_\bR$ (see \eqref{eq-Lwtwist}). Combine these discussion we get the derivative identity:
% that will be generalized in Lemma \ref{lem-der}:
\begin{equation}\label{eq-der1}
\frac{d}{ds}\bfH^\NA(\wt_{\xi+s\eta})=\frac{d}{ds}\hat{\bfH}^\NA(\wt_{\xi+s\eta})=\bfD^\NA_\xi(\wt_\eta)=\Fut_\xi(\eta).
\end{equation}

For simplicity of notations, we introduce:
\begin{defn}\label{def-Dstable}
We say that $(X, \xi)$ is K-semistable (resp. K-polystable) if $X$ is $g_\xi$-Ding-semistable (resp. $g_\xi$-Ding-polystable).
\end{defn}
\begin{rem}
Because by Theorem \ref{thm-special} it is enough to test the stability on special test configuration, this definition coincides with the original modified K-(poly)stability adopted by Tian, Berman-Witt-Nystr\"{o}m and others. To respect the original notation, we will just call $(X, \xi)$ to be K-(poly)stability, although we will also freely use the notion of Ding-(poly)stability.
\end{rem}

By \cite{BWN, DS16} when $X$ is smooth, the Yau-Tian-Donaldson conjecture is true, i.e. K-polystability is equivalent to the existence of K\"{a}hler-Ricci solitons. For singular $X$, we proved in \cite{HL20} a version of Yau-Tian-Donainldson conjecture involving $\Aut(X, \xi)_0$-uniform Ding-stability.

%The second identity follows from the fact that 

%If $(\mcX, -K_{\mcX}, \eta)$ is a $\bT$-equivariant special test configuration, then we have:
%\begin{eqnarray}
%\bfL^\NA_\xi(\mcX, -K_{\mcX}, \eta)&=&\frac{\int_{\mcX_0}\theta_{\tilde{\vphi}}(\eta)e^{-\tilde{\vphi}}}{\int_{\mcX_0}e^{-\tilde{\vphi}}}=0, \\
%\bfD^\NA_\xi(\mcX, -K_{\mcX}, \eta)&=&-\frac{1}{\bfV_g}\int_{\mcX_0}\theta_{\tilde{\vphi}}(\eta) e^{-\theta_{\tilde{\vphi}}(\xi)}(\ddc\tilde{\vphi})^n. \label{eq-TZ}
%\end{eqnarray}
%Note that the expression in \eqref{eq-TZ} is nothing but Tian-Zhu's holomorphic invariant on $\mcX_0$ which is an obstruction to the existence of K\"{a}hler-Ricci solitons.

\section{$\bfH^\NA$-invariant and MMP}

\subsection{An intersection formula for higher moments}

Let $(\mcX, \mcL, \eta)$ be any normal ample test configuration. Choose a smooth (semipositive) curvature form $\omega$ in $c_1(\mcL|_{\mcX_0})$. Let $\theta$ be the Hamiltonian function for $\eta$ with respect to $\omega$: $\iota_\eta \omega=\frac{\sqrt{-1}}{2\pi}\bar{\partial}\theta$. By the equivariant Riemann-Roch formula, we get:
\begin{eqnarray*}
\bfE_k^\NA(\mcX, \mcL)&:=&\bfE_k^\NA(\cF_{(\mcX,\mcL)})=
\lim_{m\rightarrow+\infty} \frac{1}{N_m}\sum_{i}\left(\frac{\lambda^{(m)}_i}{m}\right)^k\\
&=&\frac{1}{\bfV}\int_{\mcX_0} \theta^k \omega^n.
\end{eqnarray*}

To motivate our calculations, we will first give direct proof of two identities which can already be derived from above discussion:
\begin{lem}
We have the identity:
\begin{eqnarray}
\bfE^\NA_k(\mcX, \mcL)&=&\frac{1}{\bfV}\int_\bR x^k \DHM(\cF^{(x)})\\
\bfE^\NA(\mcX, \mcL)&=&\bfE_1^\NA(\mcX, \mcL)=\frac{1}{\bfV} \frac{\bar{\mcL}^{\cdot n+1}}{n+1}. \label{eq-E1int}
\end{eqnarray}
\end{lem}
\begin{proof}
When we change $\mcL$ to $\mcL+d\mcX_0$, $\mcF$ is changed to $\mcF(d)$, and both sides of the above identities are added by $d$. So we can assume that $\bcL$ is very ample over $\bcX$. Then we have:
\begin{equation}
\bar{\mcX}={\rm Proj}\left(\bigoplus_{m\ge 0}\bigoplus_{j=0}^{+\infty}t^{-j}\cF^{j} R_m\right)
\end{equation}
and $\bar{\mcL}_d=\mcO_{\bcX}(1)$.

For simplicity of notations, we denote:
\begin{eqnarray}\label{eq-fkm}
f_k(m)&=&\sum_{i=1}^{N_m} (\lambda^{(m)}_i)^k=\sum_{j=0} j^k (\dim \cF^j R_m-\dim \cF^{j+1} R_m)\nonumber \\
&=& \sum_{j=1}(j^k-(j-1)^k)\dim \cF^j R_m\nonumber \\
&=& \sum_{j=1} (k j^{k-1}+O(j^{k-2}))\dim \cF^j R_m.
\end{eqnarray}
We easily get the following identity:
\begin{eqnarray}\label{eq-Eklim}
\bfE^\NA_k&=&\frac{1}{\bfV}\lim_{m\rightarrow+\infty} \frac{n!}{m^{n+k}} f_k(m)\nonumber\\
&=&\frac{1}{\bfV}\int_{0}^\infty k x^{k-1}\vol(\cF^{(x)}R_\bullet)dx=\frac{1}{\bfV}\int_\bR x^k (-d\vol(\mcF^{(x)})).
\end{eqnarray}
Moreover we have the dimension formula
\begin{eqnarray*}
\mcN_m&:=&h^0(\bar{\mcX}, m \bar{\mcL})=\sum_{j=0}^{+\infty} \dim\cF^j R_m\\
&=&\frac{m^{n+1}}{n!} \int_{0}^{+\infty} \vol(\cF^{(x)}R_\bullet)dx+O(m^{n})
\end{eqnarray*}
%Recall that $\cF(d)^{mx}R_m=\cF^{mx-md}R_m$, we have:
%\begin{eqnarray*}
%\frac{\bar{\mcL}_d^{\cdot n+1}}{n+1}&=&\int_{0}^{+\infty} \vol(\cF^{(x-d)})dx\\
%&=&d\cdot V-\int_{0}^{+\infty} x \DHM(\cF).
%\end{eqnarray*}
which, by the Riemann-Roch formula, gives the identity:
\begin{equation}
\frac{1}{\bfV}\frac{\bcL^{\cdot n+1}}{n+1}=\frac{1}{\bfV}\int_0^{+\infty}\vol(\cF^{(x)}R_\bullet)dx=\frac{1}{\bfV}\int_0^{+\infty}x \DHM(\cF).
\end{equation}
\end{proof}

 The formula \eqref{eq-E1int} goes back to Mumford's study of GIT \cite{Mum77}, and has also been used in the study of K-stability. The following result is a generalization of it to higher moments. We will use the following notations as in \cite{HL20}. Let $\bC^*\rightarrow \bC^{k+1}\setminus \{0\} \rightarrow \bP^k$ be the principal $\bC^*$-bundle and set:
 \begin{equation}
(\bar{\mcX}^{[k]}, \bar{\mcL}^{[k]}):=((\bar{\mcX}, \bar{\mcL})\times (\bC^{k+1}\setminus\{0\}))/\bC^*
\end{equation}
Since the $\bC^*$-action on $\bcX$ moves the fibre $\bar{\mcX}\rightarrow\bP^1$, the situation here is different with the situation in \cite[Corollary 3.4]{BHJ17} or
 \cite{HL20} where a similar fibre construction with respect to a vertical torus action is used. 
 \begin{prop}
Let $(\mcX, \mcL)$ be a normal ample test configuration.
For any $k\ge 1$, we have the following intersection formula:
\begin{equation}\label{eq-Ekint}
\bfE _k^\NA(\mcX, \mcL)=\frac{1}{\bfV}\frac{k!n!}{(n+1)!}(\bcL^{[k-1]})^{\cdot n+k}.
\end{equation}
\end{prop}

\begin{proof}
We use the notations from the above proof and without loss of generality assume that $\bcL$ is very ample over $\bcX$.

The weights $\{\mu_\alpha; \alpha=1,\dots, \mcN_m\}$ and multiplicities of $\bC^*$-action on $H^0(\bcX, \bcL)$ are given according to the isomorphism \eqref{eq-mcXProj}. By the identity \eqref{eq-fkm}, the weight of $\bC^*$ on $\det H^0(\bcX, m\mcL)$ is given by:
 \begin{eqnarray}\label{eq-summu}
\sum_{\alpha=1}^{\cN_m}\mu_\alpha^{k-1}=\sum_{j=0}^{+\infty} j^{k-1} \dim \cF^j R_m=k^{-1}f_k(m)+O(m^{n+k-1}).
\end{eqnarray}
Choose a smooth K\"{a}hler metric $\Omega\in c_1(\bcL)$ on $\bcX$ and let $\Theta$ be the Hamiltonian function for $\eta$. Then by the equivariant Riemann-Roch formula, we get %or using the Duistermaat-Heckman measure, we get:
\begin{eqnarray}\label{eq-limsummu}
\lim_{m\rightarrow+\infty} 
&&\frac{(n+1)!}{m^{n+1}}\sum_\alpha \left(\frac{\mu_\alpha}{m}\right)^{k-1}=\int_{\bar{\mcX}}\Theta^{k-1} \Omega^{n+1}\nonumber \\
&=&\frac{(k-1)!(n+1)!}{(k+n)!}\int_{\bar{\mcX}^{[k-1]}}(\Omega+\Theta t)^{n+k}=\frac{(k-1)!(n+1)!}{(k+n)!}(\bcL^{[k-1]})^{\cdot n+k}.
\end{eqnarray}
Combining \eqref{eq-Eklim}, \eqref{eq-summu} and \eqref{eq-limsummu}, we get:
\begin{eqnarray*}
\bfE_k^\NA&=&\frac{1}{\bfV}\lim_{k\rightarrow}\frac{n!}{m^{n+k}}k\sum_\alpha \mu_\alpha^{k-1}=\frac{1}{\bfV}\frac{k}{n+1} \frac{(k-1)!(n+1)!}{(k+n)!}(\bcL^{[k-1]})^{\cdot n+k}\\
&=&\frac{1}{\bfV}\frac{k!n!}{(k+n)!}(\bcL^{[k-1]})^{\cdot n+k}.
\end{eqnarray*}
\end{proof}
Recall from \eqref{eq-bVa} that $\btS^\NA(\mcX, \mcL, a\eta)=-\log\bfQ^{(a)}$ where
\begin{eqnarray*}
\bfQ^{(a)}&=&\frac{1}{\bfV}\int_{\mcX_0} e^{-a\theta}\omega^n=\sum_{k=0}^\infty (-1)^k a^k \frac{1}{\bfV}\int_{\mcX_0}\frac{\theta^k}{k!} \omega^n\\
&=&\sum_{k}(-1)^k \frac{a^k}{k!} \bfE_k^\NA.
\end{eqnarray*}

\begin{prop}\label{prop-dS}
Let $(\mcX, \mcL_\lambda, a\eta)_{\lambda\in (-\epsilon,\epsilon)}$ be a family of normal test configurations of $(X, -K_X)$ (with a fixed total space and varying polarization).
Assume that $\mcX_0=\sum_i b_i E_i$ for irreducible components $E_i$ and $\mcL_\lambda$ is differentiable with respect to $\lambda$. Then 
we have the following derivative formula:
\begin{eqnarray}\label{eq-dbtS}
\frac{d}{d\lambda}\btS^\NA(\mcX, \mcL, a\eta)&=&a\frac{\sum_i e_i\bfQ^{(a)}_i}{\bfQ^{(a)}}
\end{eqnarray}
where $\bfQ^{(a)}_i=\frac{1}{\bfV}\int_{E_i}e^{-a\theta}\omega^n$.
\end{prop}
\begin{proof}
We use the intersection formula \eqref{eq-Ekint} to get:
\begin{eqnarray*}
V\cdot \frac{d}{d\lambda}\bfE_k^\NA&=&\frac{d}{d\lambda}\frac{k!n!}{(k+n)!}(\bcL^{[k-1]})^{\cdot n+k}=\frac{k!n!}{(k+n-1)!}(\bcL^{[k-1]})^{\cdot n+k-1}\cdot \dot{\bcL}^{[k-1]}\\
&=&\frac{k!n!}{(k+n-1)!}\sum_i e_i \int_{E_i^{[k-1]}}(\Omega+\Theta t)^{n+k-1}\\
&=&k\sum_i e_i \int_{E_i} \theta^{k-1}\omega^n.
\end{eqnarray*}
where $E_i^{[k-1]}=(E_i\times \bC^{k-1}\setminus\{0\})/\bC^*$.
So we get the wanted formula:
%\begin{eqnarray*}
%\bfQ^{(a)}_i=\int_{E_i}e^{-a\theta}\omega^n, \quad \bfQ^{(a)}=\sum_i b_i \bfQ^{(a)}_i.
%\end{eqnarray*}
\begin{eqnarray*}
\frac{d}{d\lambda} \bfQ^{(a)}&=&\sum_{k}(-1)^k \frac{a^k}{k!} \frac{d}{d\lambda}\bfE_k^\NA\\
&=&\sum_{k=1} (-1)^k \frac{a^k}{(k-1)!}\sum_i e_i \frac{1}{\bfV}\int_{E_i}\theta^{k-1}\omega^n\\
&=&-a  \sum_i e_i \frac{1}{\bfV}\int_{E_i}\sum_{j=0} \frac{(-1)^j}{j!} (a\theta)^j \omega^{n} \\
&=&-a \sum_i e_i \frac{1}{\bfV}\int_{E_i}e^{-a\theta} \omega^{n}=-a \sum_i \bfQ^{(a)}_i.
\end{eqnarray*}
The term wise differentiation and the change of summation are valid because of absolute convergence.
\end{proof}

\subsection{Decreasing of $\bfH^\NA$ along MMP}

\begin{thm}\label{thm-MMP}
Let $\bG$ be a reductive group and $(\mcX, \mcL, a\eta)$ be a $\bG$-equivariant normal test configuration. There exists a $\bG$-equivariant special test configuration $(\mcX^s, \mcL^s, a^s \eta^s)$ such that:
\begin{equation}
\bfH^\NA(\mcX, \mcL, a\eta)\ge \bfH^\NA(\mcX^s, \mcL^s, a^s \eta^s).
\end{equation}
Moreover if $\mcX_0$ is reduced, then the identity holds true if $\mcX$ is already a special test configuration.

%Moreover, the identity holds true if and only if $(\mcX, \mcL, a\eta)$ is already special.

\end{thm}

%\begin{cor}
%If $(\mcX, \mcL, \eta)$ obtains the minimum defining $h(X)$, then $(\mcX, \mcL, \eta)$ is a special test configuration and in particular, $\mcX_0$ is a $\bQ$-Fano variety.
%\end{cor}
%The key observation to prove Theorem \ref{thm-MMP} is the following
%In the following lemma, $(\mcX, \mcL, \eta)$ is any fixed normal test configuration and $(\mcX, \mcL, \eta)_\ell$ is the base change.
%\begin{lem}
%For any $c>0$, we have the identity:
%\begin{equation}
%\bfL^\NA(\mcX, \mcL, c\eta)=c \bfL^\NA(\mcX, \mcL, \eta).
%\end{equation}
%\end{lem}
%Fix any $a>0$. For $x\in \bR$, set
%\begin{equation}
%g_a=\frac{e^{-ax}-1}{-x}=\frac{1-e^{-ax}}{x}>0, \quad \text{ or equivalently} \quad e^{-ax}=1-x g_a(x).
%\end{equation}
%Then $g_a$ satisfies:
%\begin{equation}
%g_a(x)\le e^{-ax}.
%\end{equation}

\begin{proof}%[Proof of Theorem \ref{thm-MMP}]

For simplicity of notations, we assume $\bG$ is trivial. The general case is obtained by running the $\bG$-equivariant MMP in the following arguments.

{\it Step 1:}
Choose a semistable reduction of $\mcX\rightarrow \bC$. By this, we mean that there is an integer $d$ and a $\bG$-equivariant log resolution of singularities $\tilde{\mcX}\rightarrow \mcX^{(d_1)}:=\mcX\times_{\bC, t\rightarrow t^{d_1}}\bC$ such that $(\tilde{\mcX}, \tilde{\mcX}_0)$ is simple normal crossing. In particular, $\mcX^{(d_1)}_0$ is reduced. By using the identity \eqref{eq-FTCbase} and Lemma \ref{lem-abtrans} we easily get:
\begin{equation}
\bfH^\NA(\mcX^{(d_1)}, \mcL^{(d_1)}, a \eta^{(d_1)}/d_1)=\bfH^\NA(\mcX, \mcL, a\eta). 
\end{equation}

{\it Step 2:}
In this step, we show that there exist $d_1\in \bZ_{>0}$, a projective birational $\bC^*$-equivariant morphism $\pi: \mcX^{\rm lc}\rightarrow \mcX^{(d_1)}$ and a normal, ample test configuration $(\mcX^\lc, \mcL^\lc)/\bC$ for $(X, L)$ such that,
 \begin{equation}\label{eq-dec1}
\bfH^\NA(\mcX^{(d_1)}, \mcL^{(d_1)}, a\eta^{(d_1)}/d_1)\ge \bfH^\NA(\mcX^\lc, \mcL^\lc, a\eta^\lc/d_1).
\end{equation}
Moreover if the equality holds, then $(\mcX^{(d_1)}, \mcL^{(d_1)})$ is isomorphic to $(\mcX^\lc, \mcL^\lc)$, and hence $(\mcX, \mcX_0)$ is already log canonical.

%$\bfD^\NA_\xi(\mcX^{\rm lc}, \mcL^{\rm lc})\le d \cdot \bfD^\NA_\xi(\mcX, \mcL)$ holds. Moreover equality holds true if and only if $(\mcX^{(d)}, \mcX^{(d)}_0)$ is log canonical and $-K_{\mcX^{(d)}}\sim_\bQ \psi^*_d\mcL$

We run a $\bC^*$-equivariant MMP to get a log canonical modification: $\pi^\lc: \mcX^\lc\rightarrow\mcX^{(d_1)}$ such that
 $(\mcX^\lc, \mcX^\lc_0)$ is log canonical and $K_{\mcX^\lc}$ is relatively ample over $\mcX^{(d_1)}$. 
Set $E=K_{\mcX^\lc}+(\pi^\lc)^*\mcL=\sum_{i=1}^k e_i \mcX_{0,i}$ with $e_1\le e_2\le \dots \le e_k$ and $\mcL^\lc_\lambda=(\pi^\lc)^* \mcL^{(d_1)}+\lambda E$.
Then since $E$ is relatively ample over $\mcX^{(d_1)}$, $\mcL_\lambda$ is ample over $\mcX^{\lc}$ for $0<\lambda\ll 1$.
\begin{eqnarray*}
\bfL^\NA(\mcX^\lc, \mcL_\lambda^\lc, a\eta^\lc/d_1)=\frac{a}{d_1}\bfL^\NA(\mcX^\lc, \mcL^\lc_\lambda, \eta^\lc)=\frac{a}{d_1}(1+\lambda)e_1.
\end{eqnarray*}
By definition \eqref{eq-defbtD} we have:
\begin{eqnarray*}
\btS^\NA(\mcX^\lc, \mcL_\lambda^\lc, a\eta^\lc/d_1)&=&-\log \bfQ^{(ad_1^{-1})},\\
 %\bfD_g^\NA(\mcL) &=-(1-\epsilon)\frac{\mcL_g^{\cdot n+1}}{n+1}-\epsilon \mcL \cdot L^{\cdot n}+(1+(1+t)e_1\cdot \bfQ^g)\\
% &= -(1-\epsilon)\frac{\mcL_g^{\cdot n+1}}{n+1}-\epsilon \mcL \cdot L^{\cdot n}+(1+(1+t)e_1\cdot \bfQ^g)\mcX^\lc_0\cdot \mcL_g^n
\bfH^\NA(\mcX^\lc, \mcL^\lc_\lambda, a\eta^\lc/d_1) &=&\frac{a(1+\lambda) e_1}{d_1}+\log \bfQ^{(ad_1^{-1})}.
%&=&-\frac{(\bar{\mcL}_{\lambda}^{\cdot n+1})_g}{n+1}-\epsilon (\bar{\mcL}_\lambda \cdot L_{\bP^1}^{\cdot n})_g+(1+\lambda)e_1\\
% &=& -(1-\epsilon)\frac{(\bar{\mcL}_\lambda^{\cdot n+1})_g}{n+1}-\epsilon (\bar{\mcL}_{s} \cdot L_{\bP^1}^{\cdot n})_g+(1+\lambda)e_1 (\mcX^\lc_0\cdot \bar{\mcL}^{\cdot n})_g
%\bfD_g^\NA(\bar{\mcL}) &=-\frac{\bar{\mcL}_g^{\cdot n+1}}{n+1}+(1+t)e_1\cdot \bfQ^g \\
% &= -\frac{\bar{\mcL}_g^{\cdot n+1}}{n+1}+(1+(1+t)e_1)\mcX^\lc_0\cdot \bar{\mcL}_g^n
\end{eqnarray*}
%{\color{blue}where we have used the fact that: by flatness,  $\mcX^\lc_0\cdot \mcL_\lambda^n = \mcX^\lc_1\cdot \mcL_\lambda^n = \mcX^\lc_1\cdot \mcL_0^n =  (L_{\bP^1})_g^n$;
%and $(L_{\bP^1})_g^n = \int_X g \omega^n = 1$ by normalization.
%To be more precise, we can choose $\varphi(t) = \phi + t$. Then $e^{\varphi(t)}$ extends to a smooth metric on $\mcL$ over $\Delta$.
%Then
%\begin{align*}
%\mcX^\lc_1 \cdot \mcL_0^{\cdot n} &= \sum_{\vk\in N^r_\bZ} a_\vk (\mcX^\lc_1)^\bvk\cdot (\mcL^\bvk)^{\cdot n} \\
%&= \sum_{\vk\in N^r_\bZ} a_\vk \lim_{t\to\infty} \frac{1}{t} \int_{X^\bvk} \bfE^\bvk(\varphi(t)) \\
%&= \sum_{\vk\in N^r_\bZ} a_\vk \int_{X} \varphi' \prod_{i=1}^r \theta_\alpha^{k_\alpha}\frac{(\ddc\vphi)^n}{n!} \\
%&= \int_X 1 g_\vphi  \frac{(\ddc\vphi)^n}{n!} = 1
%\end{align*}
%}
We then use \eqref{eq-dbtS} to calculate:
\begin{eqnarray*}
\frac{d}{d\lambda}\bfH^\NA(\mcX^\lc, \mcL^\lc_\lambda, a\eta^\lc/d_1)&=&\frac{a e_1}{d_1}-\frac{a}{d_1}\frac{\sum_i e_i \bfQ^{(a d_1^{-1})}_i}{\sum_i \bfQ^{(a d_1^{-1})}_i}\\
&=&-\frac{a}{d_1}\frac{\sum_i (e_i-e_1 )\bfQ^{(a d_1^{-1})}_i}{\bfQ^{(a d_1^{-1})}}\le 0.
\end{eqnarray*}
The last identity holds if and only if $e_i\equiv e_1$, and hence $(\mcX^{(d_1)}, \mcL^{(d_1)})\cong (\mcX^\lc, \mcL^\lc)$.  In this case, $(\mcX^{(d_1)}, \mcX^{(d_1)}_0)$ is log canonical, which implies $(\mcX, \mcX_0)$ is already log canonical by the pull-back formula for the log differential (see \cite[pg. 210]{LX14}).

{\it Step 3:}
With the $(\mcX^\lc, \mcL^\lc)$ obtained from the first step, we run a relative MMP with scaling to get a normal, ample test configuration $(\mcX^{\rm ac}, \mcL^{\rm ac})/\bP^1$ for $(X, -K_X)$ with 
$(\mcX^{\rm ac}, \mcX^{\rm ac}_0)$ log canonical such that $-K_{\mcX^{\rm ac}}\sim_{\bQ,\bC}\mcL^{\rm ac}$. More concretely, we take $q\gg 1$ such that $\mcH^\lc=\mcL^\lc-(q+1)^{-1} (\mcL^\lc+K_{\mcX^\lc})$ is relatively ample. Set 
$\mcX^0=\mcX^\lc$, $\mcL^0=\mcL^\lc$, $\mcH^0=\mcH^\lc$ and $\lambda_0=q+1$. 
Then $K_{\mcX^0}+\lambda_0 \mcH^0=q \mcL^0$. We run a sequence of $K_{\mcX^0}$-MMP over $\bC$ with scaling of $\mcH^0$.  Then we obtain a sequence of models 
$$
\mcX^0\dasharrow \mcX^1\dasharrow\cdots \dasharrow \mcX^k
$$
and a sequence of critical values 
$$
\lambda_{i+1}=\min\{\lambda; K_{\mcX^i}+\lambda \mcH^i \text{ is nef over } \bC\}
$$
with $q+1=\lambda_0\ge \lambda_1\ge \cdots\ge \lambda_k>\lambda_{k+1}=1$. For any $\lambda_i\ge \lambda\ge \lambda_{i+1}$,
we let $\mcH^i$ be the pushforward of $\mcH$ to $\mcX^i$ and set
\begin{equation}
\mcL^i_\lambda=\frac{1}{\lambda-1}\left(K_{\mcX^i}+\lambda\mcH^i\right)
=\frac{1}{\lambda-1}(K_{\mcX^i}+\mcH^i)+\mcH^i=:\frac{1}{\lambda-1}E+\mcH^i.
\end{equation}
Write $E=\sum_{j=1}^k e_j \mcX^i_{0,j}$ with $e_1\le e_2\le \cdots \le e_k$. Then we have $\frac{d}{d\lambda} \mcL^i_\lambda=-\frac{1}{(\lambda-1)^2}E$ and 
\begin{eqnarray*}
\bfL^\NA(\mcX^i, \mcL^i, a\eta^i/d_1)=\frac{a\lambda}{\lambda-1}e_1.
\end{eqnarray*}
So we can again use \eqref{eq-dbtS} to calculate:
\begin{eqnarray*}
\frac{d}{d\lambda} \bfH^\NA(\mcX^i, \mcL^i_\lambda, a\eta^i/d_1)&=&-\frac{a}{d_1(\lambda-1)^2}e_1+\frac{a}{d_1(\lambda-1)^2}\frac{\sum_i e_i \bfQ^{(a d_1^{-1})}_i}{\bfQ^{(a d_1^{-1})}}\\
&=&\frac{a}{d_1(\lambda-1)^2}\frac{\sum_i (e_i-e_1)\bfQ^{(a)}_i}{\bfQ^{(a)}}\ge 0.
\end{eqnarray*}
The last identity holds only if $e_i\equiv e_1$, which implies $(\mcX^\lc, \mcL^\lc)\cong (\mcX^\ac, \mcL^\ac+e_1 \mcX^\ac_0)$. 
%Again because $|-K_{\mcX^0}-\mcD^0|_\bQ$ is $\bG$-invariant and $\mcQ^0$ is a general divisor, the above MMP is $\bG$-equivariant. The same argument as in Step 1 verifies the more general inequality \eqref{eq-dec2}.

{\it Step 4:}
With the test configuration $(\mcX^\ac, \mcL^\ac)$ obtained from step 2, there exists 
$d_2\in \bZ_{>0}$ and a projective birational $T_\bC\times \bC^*$-equivariant birational map $(\mcX^\ac)^{(d_2)}\dasharrow \mcX^s$ over $\bP^1$ such that $(\mcX^s, -K_{\mcX^s})$ is a special test configuration and 
\begin{equation}
\bfH^\NA(\mcX^\ac, \mcL^\ac,a\eta/(d_1d_2))\ge \bfH^\NA(\mcX^s, \mcL^s, a \eta^s/(d_1d_2)).
\end{equation}
As in \cite{LX14}, this is achieved by doing a base change and run an MMP. 
Let $E=-K_{\mcX^s/\bP^1}-(-K_{\mcX'/\bP^1})$. Then $E\ge 0$ by the negativity lemma. $\mcL'_\lambda=-K_{\mcX'/\bP^1}+\lambda E$. 
\begin{equation}
\lct(\mcX', \mcL'_\lambda, a\eta'/d_1 d_2)=\frac{a}{d_1 d_2}\lambda e_1.
\end{equation}
So as before, we get:
\begin{eqnarray*}
\frac{d}{d\lambda} \bfH^\NA(\mcX', \mcL'_\lambda, a\eta'/d_1 d_2)&=&\frac{a}{d_1 d_2}e_1-\frac{a}{d_1 d_2}\frac{\sum_i e_i \bfQ^{(a d_1^{-1})}_i}{\bfQ^{(a d_1^{-1})}}\\
&=& -\frac{a}{d_1 d_2}\frac{\sum_i (e_i-e_1)\bfQ^{(a d_1^{-1})}_i}{\bfQ^{(a d_1^{-1})}}\le 0.
\end{eqnarray*}
The last identity holds only if $e_i\equiv e_1$ which implies $(\mcX^\ac, \mcL^\ac)\cong (\mcX^s, \mcL^s)$.

%So we have obtained special test configuration $(\mcX^s,  \mcL^s)$ of $(X, -K_X)$ which is $\bT$-equivariant. 
\end{proof}

\begin{cor}\label{cor-hspecial}
We have the identity:
\begin{equation}
h(X)=\inf_{(\mcX, \mcL, a\eta) \text{ special }} \bfH^\NA(\mcX, \mcL, a\eta).
\end{equation}
\end{cor}

\begin{lem}\label{lem-TC*}
For any normal test configuration $(\mcX, \mcL, \eta)$, there exists a unique $a_*>0$ such that 
\begin{equation}
\bfH^\NA(\mcX, \mcL, a_* \eta)=\inf_{c>0} \bfH^\NA(\mcX, \mcL, a \eta)=:\bfH^\NA_*(\mcX, \mcL).
\end{equation}
As a consequence, we have: 
\begin{equation}
h(X)=\inf_{(\mcX, -K_{\mcX}) \text{ special}}\bfH^\NA_*(\mcX, \mcL).
\end{equation}
\end{lem}
\begin{proof}
By resolution of singularities, we can assume that $(\mcX, \mcX^{\rm red}_0)$ is simple normal crossing and there is a $\bC^*$-equivariant dominating map $\rho: \mcX\rightarrow X_\bC$, and we have identities:
\begin{equation}
K_{\mcX}=\rho^*K_{X_\bC}+\sum_i a_i E_i, \quad \mcL=\rho^*(-K_{X_\bC})+\sum_i c_i E_i, \quad \mcX_0=\sum_i b_i E_i.
\end{equation}
Then $-(K_{\mcX}+\mcL)=-\sum_i (a_i+c_i)E_i$ and
\begin{equation}
\lct(\mcX, -(K_{\mcX}+\mcL); \mcX_0)= \min_i \frac{1+a_i+c_i}{b_i}.
\end{equation}
Because $\btD$ is translation invariant, by replacing adding a multiple of $\mcX_0$ to $\mcL$, we can normalize $\phi=\phi_{\mcF}$ to satisfy $\bfL^\NA(\phi)=0$. So 
\begin{equation}
c_i\ge b_i-1-a_i=b_i-A_i.
\end{equation}
and without loss of generality, $c_1=b_1-A_1$.
So 
\begin{equation}
\lambda_{\min}=\min_i \frac{c_i}{b_i}=\min_i (1-\frac{A_i}{b_i})\le 1-\frac{A_1}{b_1}=0.
\end{equation}
If $\lambda_{\min}<0$, then it is easy to see that the convex function $a\rightarrow \bfH^\NA(\mcX, \mcL, a\eta)$ is proper and hence admits a unique minimum. If $\lambda_{\min}=0$, then for all $i$, we have:
$c_i=b_i$ which implies that the normal test configuration $(\mcX, \mcL)$ is a trivial test configuration and hence $\bfH^\NA(\mcX, \mcL, a\eta)\equiv 0$.

\end{proof}

\section{A minimization problem for real valuations}
In this section, we will introduce a minimization problem for valuations analogous to the normalized volume functional in the local setting (\cite{Li18}).

\begin{defn}
For any $v\in \Val(X)$, define:
\begin{equation}
\tbeta(v)=
\left\{
\begin{array}{ll}
A_X(v)-\btS^\NA(\cF_v) & \text{ if } A_X(v)<+\infty, \\
+\infty& \text{ otherwise}
\end{array}
\right.
\end{equation}
\end{defn}
Note that by integration by parts we have:
\begin{eqnarray}\label{eq-bV1form}
e^{-\btS^\NA(\cF_v)}&=&\frac{1}{\bfV}\int_\bR e^{-x}\DHM(\cF_v)=\frac{1}{\bfV}\int_0^{+\infty}e^{-x} (-d\vol(\cF^{(x)}_v))\nonumber \\
&=&1-\frac{1}{\bfV}\int_0^{+\infty}\vol(\cF_v^{(x)}R_\bullet)e^{-x}dx\le 1
\end{eqnarray}
with identity if and only if $v$ is trivial.
So we can re-write $\tbeta(v)$ as:
\begin{equation}
\tbeta(v)=A_X(v)+\log \left(1-\frac{1}{\bfV}\int_0^{+\infty} e^{-x}\vol(\cF_v^{(x)}R_\bullet)dx\right).
\end{equation}

\begin{lem}\label{lem-tDtbeta}
For any $v\in X^{\rm div}_\bQ$, we have the inequality:
\begin{equation}
\bfH^\NA(\cF_v)\le \tbeta(v).
\end{equation}
Moreover if $(\mcX, -K_{\mcX}, a\eta)$ is a special test configuration, then the equality holds true for $v=a v_{\mcX_0}=a\cdot r(\ord_{\mcX_0})$ (see Definition \ref{def-TC}).
\end{lem}
\begin{proof}
The inequality follows immediately from
\begin{equation}
\inf_{w}(A(w)+\phi_v(w))\le A(v)+\phi_v(v)=A(v).
\end{equation}
When $(\mcX, -K_{\mcX}, a\eta)$ is a special test configuration and $v=a v_{\mcX_0}$, then 
\begin{equation}
\bfL^\NA(\mcX, -K_\mcX, a\eta)=a\bfL^\NA(\mcX, -K_{\mcX}, \eta)=a(\lct(\mcX; \mcX_0)-1)=0.
\end{equation}
On the other hand, by \eqref{eq-specialmin}, 
\begin{equation}
\bfL^\NA(\mcX, -K_{\mcX}, a\eta)=\bfL^\NA(\cF_v(-A(v))=\bfL^\NA(\mcF_v)-A(v).
\end{equation}
So we get
\begin{equation}
\bfH^\NA(\mcF_v)=\bfL^\NA(\cF_v)-\btS^\NA(\cF_v)=A(v)-\btS^\NA(v)=\tilde{\beta}(v).
\end{equation}
\end{proof}
\begin{lem}
For any $\phi=\phi_\cF$ and $v\in X^{\rm div}_\bQ$, we have the inequality:
\begin{equation}\label{eq-v2phi}
\btS(v)+\phi(v)\ge \btS^\NA(\phi).
\end{equation}
\end{lem}
\begin{proof}
We use the same argument as in \cite[4.2]{Li19}. Set $\gamma=\phi(v)$. Then by the argument there, we have $\lambda_{\min}=\lambda_{\min}(\cF)\le \gamma$ and we can then estimate:
\begin{eqnarray*}
e^{-\btS^\NA(\phi)}&=&\bfQ(\phi)=\frac{1}{\bfV}\int_\bR e^{-x}(-d\vol(\mcF^{(x)}))=e^{-\lambda_{\min}}-\frac{1}{\bfV}\int_{\lambda_{\min}}e^{-x}\vol(\cF^{(x)}R_\bullet)dx\\
&\ge &e^{-\gamma}-\frac{1}{\bfV}\int_{\gamma}^{+\infty} e^{-x}\vol(\cF^{(x)}R_\bullet)dx\ge e^{-\gamma}-\frac{1}{\bfV}\int_\gamma^{+\infty} e^{-x} \vol(\cF_v^{(x-a)})dx\\
&=&e^{-\gamma}-e^{-\gamma}\frac{1}{\bfV}\int_0^{+\infty}e^{-x}\vol(\cF_v^{(t)})dt=e^{-\gamma}\frac{1}{\bfV}\int_0^{+\infty}e^{-x}(-d\vol(\cF_v^{(t)}))\\
&=&e^{-\phi(v)}e^{-\btS^\NA(v)}.
\end{eqnarray*}
%So we get
%\begin{equation}
%\btS^\NA(\phi)\le \phi(v)+\btS^\NA(v).
%\end{equation}
\end{proof}

\begin{prop}
For any $\bQ$-Fano variety, we have the identity:
\begin{equation}\label{eq-hval}
h(X)=\inf_{v\in X^{\rm div}_\bQ}\tbeta(v).
\end{equation}
\end{prop}
\begin{proof}
For any test configuration $(\mcX, \mcL, a\eta)$, by Theorem \ref{thm-MMP} there exists a special test configuration $(\mcX^s, \mcL^s, a^s \eta^s)$ such that
\begin{equation}
\bfH^\NA(\mcX, \mcL, a\eta)\ge \bfH^\NA(\mcX^s, \mcL^s, a^s \eta^s)=\tbeta(a^s v_{\mcX^s_0}).
\end{equation}
The last identity is from Lemma \ref{lem-tDtbeta}. This together with Corollary \ref{cor-hspecial} implies identity \eqref{eq-hval}. 

Alternatively, recall that $\bfL^\NA(\phi)=\inf_{v\in X^{\rm div}_\bQ}(A_X(v)+\phi(v))$. So for any $\epsilon>0$ we can choose $v$ such that $A_X(v)+\phi(v)<\bfL^\NA(\phi)+\epsilon$. We can then use $v$ in \eqref{eq-v2phi} to get:
\begin{eqnarray*}
\bfL^\NA(\phi)-\btS^\NA(\phi)&\ge&A_X(v)+\phi(v)-\epsilon-(\phi(v)+\btS(v))=\tbeta(v)-\epsilon.
\end{eqnarray*}
Since $\epsilon$ is arbitrary, we can use \eqref{eq-defhX} to get the identity \eqref{eq-hval}.

\end{proof}
With the identity \eqref{eq-hval}, the comparison result of Xu-Zhuang in Proposition \ref{prop-XZcomp} gives:
\begin{cor}
For any $\bQ$-Fano variety, we have the equality:
\begin{equation}\label{eq-hbhD}
h(X)=\inf_{\cF}\bfH^\NA(\cF)=\inf_{\cF}\hat{\bfH}^\NA(\cF).
\end{equation}
\end{cor}

The next result should be compared to Lemma \ref{lem-TC*}.
\begin{prop}\label{prop-vscaling}
For any $v\in X^{\rm div}_\bQ$, there exists a unique $a_*=a_*(v)\ge 0$ such that 
\begin{equation}
\tbeta(a_* v)=\inf_{a>0}\tbeta(a v)=:\tbeta_*(v).
\end{equation}
When $\beta(v)\ge 0$, then $a_*=0$ so that $a_* v$ is the trivial valuation and $\tbeta_*(v)=0$. Otherwise $a_*(v)>0$ and $\tbeta_*(v)<0$.
\end{prop}
\begin{proof}
Fix $v=q\cdot \ord_E$ for a prime divisor $E$ over $X$ and $q>0$.
Consider the function on $\bR_{\ge 0}$:
\begin{eqnarray}\label{eq-frescaling}
f(a)&=&A(a v)-\btS^\NA(a v)=a A(v)+\log \left(\frac{1}{\bfV}\int_0^{+\infty}e^{-x}\DHM(\cF_{av})\right)\nonumber \\
&=&a A(v)+\log\left(\frac{1}{\bfV}\int_0^{+\infty}e^{-ax}\DHM(\cF_v) \right).
\end{eqnarray}
We will show that $a\mapsto f(a)$ is convex and goes to $+\infty$ as $a\rightarrow+\infty$. 
\begin{eqnarray*}
f'(a)&=&A(v)-\frac{\int_0^{+\infty} xe^{-ax}\DHM(\cF_v)}{\int_0^{+\infty}e^{-ax}\DHM(\cF_v)}, \\
 f''(a)&=&\frac{\int x^2 e^{-ax}\DHM}{\int e^{-ax}\DHM}-\frac{(\int x e^{-ax}\DHM)^2}{(\int e^{-ax}\DHM)^2}=\|x-\bar{x}\|^2_{L^2(d\nu)}\ge 0,
\end{eqnarray*}
where
\begin{equation}
d\nu=\frac{e^{-ax}\DHM}{\int e^{-ax}\DHM}, \quad \bar{x}=\int x d\nu.
\end{equation}
$f''(a)=0$ if and only if $av$ is trivial. Moreover $f'(0)=A(v)-\frac{1}{\bfV}\int_0^{+\infty}x \DHM(\cF_v)=\beta(v)$.

On the other hand, $f(0)=0$ and we claim that $\lim_{a\rightarrow+\infty}f(a)=+\infty$ which then implies the statement. 
To prove this divergence, we set $g(x)=V^{-1/n}\vol(\cF^{(x)}R_\bullet)^{1/n}$. Then $g(x)$ is decreasing, concave on $[0, \lambda_{\max}]$ (by Theorem \ref{thm-BoCh}) and differentiable by \cite{BFJ09, LM09}. Fix $0<\epsilon \ll \lambda_{\max}$ such that $g(\epsilon)<g(0)=1$. Set $C=-g'(\epsilon)>0$, $T=\frac{1+C\epsilon}{C}$ and define a function 
\begin{equation}
\hat{g}(x)=\left\{\begin{array}{ll}
1 & x\in [0, \epsilon]\\
1+C\epsilon-Cx & x\in (\epsilon, T]\\
0& x\in (T,+\infty).
\end{array}\right.
\end{equation}
Then $\hat{g}(x)\ge g(x)$ over $[0, +\infty)$ (by concavity). Then we calculate to get: 
\begin{equation}
a \int_0^{+\infty}\hat{g}^n(x)e^{-ax}dx=1-n C m_{n-1}.
\end{equation}
where $m_k=\int_{\epsilon}^T (1+C\epsilon-Cx)^k e^{-ax}dx$ satisfies:
\begin{eqnarray*}
m_k&=&\frac{1}{a}e^{-a\epsilon}-\frac{kC}{a}m_{k-1}=\frac{1}{a}e^{-a\epsilon}-\frac{kC}{a}\left(\frac{1}{a}e^{-a\epsilon}-\frac{(k-1)C}{a}m_{k-2}\right).
\end{eqnarray*}
Using induction we get $m_{n-1}=a^{-1}e^{-a\epsilon}(1+O(a^{-1}))$. So 
\begin{eqnarray*}
e^{-\btS^\NA}(\cF_{av})&=&1-a\int^{+\infty}_0g^n(x)e^{-ax}dx\ge 1-a\int_0^{+\infty}\hat{g}^n(x)e^{-ax}dx\\
&=&n C m_{n-1}=nC a^{-1}e^{-a\epsilon}(1+O(a^{-1})).
\end{eqnarray*}
So we get $-\btS^\NA(\cF_{av})\ge-\log a-a\epsilon+O(1)$, which gives:
\begin{equation}
f(a)=\tbeta(av)\ge (A(v)-\epsilon)a-\log a+O(1),
\end{equation}
which approaches $+\infty$ as $a\rightarrow+\infty$ if we choose $0<\epsilon<A(v)$.

\end{proof}
%\begin{rem}\label{rem-a*w}
%By the above proof, we can easily get the following estimate: for any $C_1>0$ there exists $C_2=C_2(C_1, v)>0$ such that for any $w\in \Val(X)$ with $w\le C_1 v$, we have:
%\begin{equation}\label{eq-a*w}
%a_*(w)\le \frac{C_2}{A(v)}.
%\end{equation}
%\end{rem}

\begin{cor}
We always have $h(X)\le 0$, with equality holds true if and only if $h(X)=0$. 
\end{cor}
\begin{proof}
By \cite{Fuj19a, Li17}, $X$ is K-semistable if and only if $\beta(v)\ge 0$, which implies $\tbeta_*(v)=0$. If $X$ is not K-semistable then there exists $v'$ such that $\beta(v')<0$. By Proposition \ref{prop-vscaling}, we then have $\tbeta_*(v')<0$ which implies $h(X)<0$.
\end{proof}

%\begin{lem}
%We have the identity:
%$X$ is $\btD$-semistable if and only if $\tbeta(v)\ge 0$ for any $v\in X^{\rm div}_\bQ$.
%\end{lem}
%\begin{proof}
%\begin{eqnarray*}
%\bfL^\NA(\phi)&=&\inf_{v}(A_X(v)+\phi(v))\ge \inf_{v} (\btS(v)+\phi(v)).
%\end{eqnarray*}
%$\lambda_{\min}\le a$.
%\end{proof}

%\section{Minimizer of $\tbeta$ and modified K-semistability}

%As in the case of normalized volume,

\begin{lem}\label{lem-lctunique}
If $v$ computes $h(X)$, then $v$ is the unique valuation, up to rescaling, that computes $\lct(\fa_\bullet(v))$. 
\end{lem}
\begin{proof}
Recall that $\lct(\fa_\bullet)=\inf_w \frac{A(w)}{w(\fa_\bullet(w))}$. For any $w\in \Val(X)$, assume that $w(\fa_\bullet(v))=a>0$. Then $a^{-1}w\ge  v$. By Proposition \ref{prop-strictinc}, the function $w\mapsto w\mapsto \btS^\NA(w)=-\log\frac{1}{\bfV}\int_\bR e^{-\lambda}\DHM(\cF_w)$ is strictly increasing on $\Val(X)$. %This can be proved using the same argument as in \cite[Proof of Proposition 3.15]{BlJ17} (which is based on the proof of \cite[Proposition 2.3]{LX16}) (see Proposition \ref{prop-strict}).
%it is easy to see that $w\mapsto \btS^\NA(w)=-\log\frac{1}{\bfV}\int_\bR e^{-\lambda}\DHM(\cF_w)$ is strictly increasing in $w$.
So we use the assumption to get:
\begin{eqnarray*}
\frac{A(w)}{w(\fa_\bullet(v))}&=& A(a^{-1}w)=A(a^{-1}w)-\btS^\NA(a^{-1}w)+\btS^\NA(a^{-1}w)\\
&\ge& A(v)-\btS^\NA(v)+\btS^\NA(v)=A(v)=\frac{A(v)}{v(\fa_\bullet(v))}.
\end{eqnarray*}
When the equality holds true, then $a^{-1}w=v$.
\end{proof}

We now observe that the method developed in \cite{BLX19} can be used to prove:
\begin{thm}\label{thm-exist}
For any $\bQ$-Fano variety, there exists a minimizing valuation of $\tbeta$ which is quasi-monomial.
\end{thm}
Since the argument is almost verbatim to \cite{BLX19} except for the continuity property of $\tbeta$, we just give a sketch of key points and explain the required continuity of $\tbeta$ in \ref{sec-cont}. Without the properties of $\tbeta(S)$ explained in section \ref{sec-cont}, the existence of a valuation calculating $h(X)$ (but without the quasi-monomial property) can also be obtained using the argument in \cite[section 6]{BlJ17}.
\begin{proof}

By Corollary \ref{cor-hspecial}, $h(X)=\inf_{E}\tbeta_*(E)$ where $E$ ranges over prime divisors over $X$ that induce special test configurations of $(X, -K_X)$. By \cite[Theorem A.2]{BLX19}, we know that such an $E$ is a lc place of an $N$-complement $D$ of $X$, where $N$ depends only on the dimension $n$ (this depends on the deep result of Birkar about the boundedness of $\bQ$-complements). So we have
\begin{equation}
h(X)=\inf_{v}\tbeta_*(v)
\end{equation}
where $v$ ranges over all divisorial valuations that are lc places of an $N$-complement. For such a valuation $v$, there exists $D\in \frac{1}{N}|-NK_X|$ such that $(X, D)$ is lc and $A_{(X,D)}(v)=0$. We then parametrize such $\bQ$-divisors as in \cite[Proof of Theorem 4.5]{BLX19}. Set $W=\bP(H^0(X, \mcO_X(-NK_X))$ and denote by $H$ the universal divisor on $X\times W$ parametrizing divisors in $|-NK_X|$ and set $D:=\frac{1}{N}H$. By the lower semicontinuity of log canonical thresholds, the locus $Z=\{w\in W; \lct(X_w; D_w)=1\}$ is locally closed in $W$. For each $z\in Z$, set $b_z:=\inf_{v}\tbeta(v)$, where $v$ ranges over all $v\in \Val(X)$ with $A_{(X,D_z)}(v)=0$. 

Let $g: Y_z\rightarrow X$ is a log resolution of $(X, D_z)$. Write $K_Y+D_{Y_z}=g^*(K_X+D_z)$.
Consider the section of the simplicial cone: $\mathcal{S}:={\rm QM}(Y_z, D_{Y_z})\bigcap \{v\in \Val(X); A(v)=1\}$.
By Proposition \ref{prop-vscaling}, we know that for each $v\in \mathcal{S}$ there exists $a_*(v)$ such that $\inf_{a>0} \tbeta(a v)=\tbeta(a_*(v)v)=:\tbeta_*(v)$.
By the Izmui's estimate (see \cite[Example 11.3.9]{Laz04}, \cite[Proposition 5.10]{JM12} for the smooth case, and 
\cite[section 3]{Li18} in the klt case), we know that there exists $C_1>0$ such that for any $v\in \mathcal{S}$ we have $v\le C_1\cdot\ord_F$ where $F=\cap_i D_{Y_z,i}$. 
Now by the proof of Proposition \ref{prop-vscaling}, we know that $a_*(v)$ is uniformly bounded for any $v\in \mathcal{S}$. 
By Proposition \ref{prop-contqm}, we know that $v\mapsto \tbeta(v)$ is continuous on ${\rm QM}(Y_z, D_{Y_z})$ and hence is uniformly continuous over compact subsets. We then get the continuity of $v\mapsto \tbeta_*(v)$ over the compact set $\mathcal{S}$. So we know that there exists $v_z^*\in \mathcal{S}$ such that $\tbeta_*(v_z^*)=\min_{v\in \mathcal{S}}\tbeta_*(v)$ and $a_*(v_z^*)\cdot v_z^*$ is then a minimizer of $\tbeta$ over ${\rm QM}(Y_z, D_{Y_z})$.

Then as \cite[Proof of Theorem 4.5]{BLX19}, choose a locally closed decomposition $Z=\cup_{i=1}^r Z_i$ so that $Z_i$ is smooth and there is an \'{e}tale map $Z'_i\rightarrow Z_i$ such that $(X_{Z'_i}, D_{Z'_i})$ admits fiberwise log resolutions. By the same arguments as \cite[Proof of Proposition 4.1, Proposition 4.2]{BLX19} which depend on the deformation invariance of log plurigenera in the work of Hacon-McKernan-Xu, we know that $b_z$ is independent of $z\in Z_i$. So $b_z$ takes finitely many values and there is $z_0\in Z$ such that $h(X)=\min_{z\in Z}b_z=b_{z_0}$ is computed by $v_{z_0}^*$.

\end{proof}

As in the case of normalized volume, we expect the following
\begin{conj}\label{conj-special}
The minimizer $v_*$ is unique, and is special which means that $\cF_{v_*}$ is a special $\bR$-test configuration; 
%(i.e. and ${\rm Proj}(\Gr(\cF))$ is a $\bQ$-Fano variety denoted by $W$.
\end{conj}

\begin{rem}
As \cite[Proposition 4.11]{BlJ17}, using Lemma \ref{lem-lctunique}, one can show that any divisorial (i.e. rational rank 1) minimizing valuation is primitive and plt.

Besides the case of stability threshold treated in \cite{BLX19}, 
in the local setting of normalized volumes, the existence of quasi-monomial minimizers is also known thanks to the work of Blum \cite{Blu18} and Xu \cite{Xu20}. 
Moreover one might also be able to adapt the techniques in Xu-Zhuang \cite{XZ20} to the current global setting to prove the uniqueness of minimizing valuations. 
We will prove in section \ref{sec-semiunique} the uniqueness of special minimizers (in the similar spirit as in \cite{LX16, LX18,  LWX18}). 
\end{rem}

%We leave these detailed study to other works. Here we just observe an immediate result. 
%Note that by the result of Xu \cite{Xu20}, any minimizer $v$ must be quasi-monomial.
%\begin{prop}
%If $v=a\cdot \ord_E$ computes $h(X)$ for a prime divisor $E$ over $X$, then $E$ is dreamy and induces a special test configuration of $(X, -K_X)$.
%\end{prop}
%\begin{proof}
%By the above lemma, $E$ computes $\fa_\bullet(\ord_E)$. By \cite{BCHM}, $E$ is primitive and dreamy. By Theorem \ref{thm-special}, we know that $E$ induces a special test %configuration.
%\end{proof}

\section{Initial term degeneration of filtrations}\label{sec-initial}

Let $\cF_0$ be a special $\bR$-test configuration of $(X, -K_X)$ with central fibre $(W:={\rm Proj}(\Gr(\cF_0)), \xi_0:=\xi_{\cF_0})$. Let $\cF_1$ be another filtration of $R$. We define a filtration on 
\begin{equation}
R':=R(W, -K_{W})=\bigoplus_{m\ge 0}\bigoplus_{\lambda \in \Gamma(\cF_0)}t^{-\lambda} \cF_0^{\lambda}R_m/\cF_0^{>\lambda}R_m=:\bigoplus_{m\ge 0}R'_m
\end{equation}
in the following way. Recall that we can write:
\begin{equation}
R'_m=\bigoplus_{\alpha\in M_\bZ} t^{-\la \alpha, \xi_0\ra}\cF_0^{\la \alpha,\xi_0\ra}R_m/ \cF_0^{>\la \alpha, \xi_0\ra}R_m.
\end{equation}
For any $f\in R_m$, set:
\begin{equation}
\bin_{\cF_0}(f)=(t^{-\la \alpha, \xi\ra}\bar{f})(0)=:f'\in \cF_0^{\la\alpha,\xi_0\ra} R_m/\cF_0^{>\la \alpha, \xi_0\ra} R_m \quad \text{ where }\quad \la \alpha, \xi_0\ra=v_{\cF_0}(f).
\end{equation}
For any $\lambda\in \bR$, take the Gr\"{o}bner base type degeneration:
\begin{equation}
\cF_1'^\lambda R'_m={\rm Span}_{\bC}\left(\bin_{\cF_0}(f), f\in \cF_1^\lambda R_m \right) \subseteq R'_m.
\end{equation}
Note that because $R'$ is integral, $\bin_{\cF_0}(fg)=\bin_{\cF_0}(f)\cdot \bin_{\cF_0}(g)$ if $f\in R_{m_1}$ and $g\in R_{m_2}$.
So in this way, we get a $\bT_0$-equivariant filtration:
\begin{equation}
\cF_1'^\lambda R'_m=\bigoplus_{\alpha\in \bZ^{r_0}}\cF'^{\lambda}_1 R'_{m,\alpha}.
\end{equation}
There is an equivalent way to describe $\cF_1'^\lambda R'_m$ as follows.
For any $f'\in R'_{m,\alpha}$, we choose $f\in R_m$ such that $f'=t^{-\la \alpha, \xi_0\ra}\bar{f}(0)$.
%there exists $\bT_0$-equivariant holomorphic section $\tilde{f'}=t^{-\alpha}\bar{f}$ satisfying $t\circ \tilde{f'}=t^\alpha \tilde{f}$. 
Then we have
\begin{equation}\label{eq-F''2}
\cF_1'^\lambda R'_{m,\alpha}=\{f'\in R'_{m,\alpha}; \exists h\in \cF_0^{>\la \alpha, \xi_0\ra}R_m \text{ s.t. } f+h \in \cF_1^{\lambda} R_m\}.
\end{equation}
This is well defined since $f$ is determined up to addition by elements from $\cF^{>\la \alpha, \xi_0\ra}_0R_m$.

Note that this construction allows us to find a basis $\cB=\{s_1,\dots, s_{N_m}\}$ of $R_m$ that is adapted to both $\cF_0 R_m$ and $\cF_1 R_m$. 
Recall that this means that for any $\lambda\in \bR$ and $i=0,1$, there exists a subset of $\cB$, which depends on $\lambda$ and $i$ and spans basis of $\cF_i^\lambda R_m$. 
To find such a basis, we can first find a basis $\cB'_\alpha$ of $R'_{m,\alpha}$ which is adapted to $\cF'_1 R_{m,\alpha}$. Then $\cB=\cup_\alpha \cB_\alpha=:\{f'_1,\dots, f'_{N_m}\}$ is a basis adapted to both $\cF'_1R_m$ and $\cF'_{\wt_{\xi_0}}R'_m$. For each $f'_k\in R'_{m,\alpha_k}$, there exists $\lambda_k\in \bR$ such that $f'_k\in \cF'^{\lambda_k}_1R'_m\setminus \cF'^{>\lambda_k}_1 R'_m$. Then by \eqref{eq-F''2}, there exists $h_k\in \cF_0^{>\la \alpha_k, \xi\ra}R_m$ such that $s_k:=f_k+h_k\in \cF^{\lambda_k}_1R_m$. Moreover,  we have $s_k\not\in \cF_1^{>\lambda_k}R_m$ since otherwise $\bin(s_k)=\bin(f_k)=f'_k \in \cF'^{>\lambda_k}R'_m$.
It is easy to verify that $\{s_k\}$ is the wanted basis. So the relative successive minima of $\cF_1$ with respect to $\cF_0$ (see \cite{BoJ18b}) is given by the set $\{\lambda_k-\la \alpha_k, \xi_0\ra\}$, which is the same as the relative successive minima of $\cF'_1:=\cF'_1R'$ with respect to $\cF'_0:=\cF'_{\wt_{\xi_0}}R'$. This immediately proves a useful fact:
\begin{lem}\label{eq-d2inv}
With the constructions and notations, we have the identity:
\begin{equation}
d_2^X(\cF_0, \cF_1)=d_2^{W}(\cF'_0, \cF'_1).
\end{equation}
\end{lem}
Since the initial term degeneration does not change the dimension of vector spaces, it is clear that the successive minima of $\cF_1$ and $\cF'_1$ coincides. As a consequence, we get:
\begin{equation}
\btS^\NA_X(\cF_1)=\btS^\NA_{W}(\cF'_1).
\end{equation}
On the other hand, consider the following $\bT_0$-equivariant graded filtration of the Rees algebra $\mcR':=\mcR(\cF_0)$ (see \eqref{eq-Rees}):
\begin{equation}
\cF'^\lambda\mcR'_{m,\alpha}=\{s=t^{-\la \alpha, \xi\ra}\bar{f}\in \mcR'_{m,\alpha}; t^{-\lambda} \bar{f} \in \mcR(\cF_1) \}.
\end{equation}
Then  $\cF'\mcR'$ coincides with $\cF R$ on the generic fibre and coincides with $\cF' R'$ on the central fibre. By the lower semicontinuity of lct for a family, it is easy to see that $\bhL^\NA$ in \eqref{eq-hatbfL} is also lower semicontinous for a family. This is standard if $\cF_0$ has rank 1 which corresponds to a special test configuration (see \cite[Lemma 8.1]{KP17} and \cite[Proof of Lemma 6.5]{BL18}).  In general, one can restrict to a generic curve passing through 0 in the family in Tessier's construction in the paragraph above Lemma \ref{lem-integral} (alternatively see Remark \ref{rem-rank1}).
So we can get:
\begin{equation}\label{eq-bfLdec}
\bhL^\NA_X(\cF_1)\ge \bhL^\NA_{W}(\cF'_1).
\end{equation}
\begin{rem}
This is the only point in our argument where we prefer $\bhL^\NA$ in \eqref{eq-hatbfL} to $\bfL^\NA$, due to the reason that the property of lower semicontinuity for $\bfL^\NA$, which involves sub-log-canonical instead of log canonical pairs, is not immediately available in the literature (although we believe it to be true).
\end{rem}
%To see this we canUsing the defining formula \eqref{eq-defLNA}, \eqref{eq-bfLdec} follows from the lower semicontinuity of lct for a family 
Combining the above discussion, we get the inequality:
\begin{equation}\label{eq-bfDdec}
\hat{\bfH}^\NA_X(\cF_1)\ge \hat{\bfH}^\NA_{W}(\cF'_1).
\end{equation} 

\begin{thm}\label{thm-minsemi}
Assume $v$ induces a special $\bR$-test configuration $\cF_v$ of $X$. Then $v$ is a special minimizer of $\tbeta$ over $\Val(X)$ if and only if $v$ is Ding-semistable (or equivalently K-semistable).
\end{thm}
\begin{proof}

For simplicity of notations, set $\cF_0=\cF_v$ and $(W, \xi_0):=(X_{\cF_v,0}, \xi_{\cF_0})$ and let $\bT_0$ be the torus generated by $\xi_0$. 

We first prove that minimizer is Ding-semistable.
Suppose $(W, \xi_0)$ is not Ding-semistable. Then by Theorem \ref{thm-special} from \cite{HL20}, there exists a $\bT$-equivariant special test configuration $(\mcW, -K_{\mcW})$ of $(W, -K_{W})$ with central fibre $Y:=\cW_0$ such that
\begin{equation}
\bfD^\NA_g(\mcW, -K_{\mcW})=\Fut_{Y,\xi}(\eta)<0.
\end{equation}
We can now construct a family of special valuations $\{v_\epsilon\}$ such that $v_\epsilon$ corresponds to a vector field $\xi_\epsilon=\xi_0+\epsilon \eta$ on $Y$. This can be done by using the cone construction to reduce to the situation in \cite[section 6]{LX16} or \cite[Proof of Theorem 2.64]{LX18}. Alternatively one can use an argument involving Hilbert scheme as in \cite[Proof of Lemma 3.1]{LWX18}. 

Here we will use the Chow variety to explain this construction. Recall that the Chow point of a cycle $Z\subset \bP^{N-1}$ of degree $d$ and dimension $n$ corresponds to a divisor in the Grassmannian $Gr(n+1, \bC^N)$ which is the zero scheme of a section :
$$\CH(Z)\in H^0(Gr(n+1, \bC^N), \mcO(d))=:\bbM.$$ 
$\CH(Z)$ is determined up to rescaling and we call it the Chow coordinate of $Z$. Let $\CH(X)$, $\CH(W)$ and $\CH(Y)$ be the Chow coordinates of $X$, $W$ and $Y$ respectively. Because the $\bT$-action on $\bP^{N-1}$ induces a weight decomposition $\bbM=\bigoplus_\alpha \bbM_\alpha$. We have:
\begin{equation}
\lim_{s\rightarrow+\infty} \sigma_\xi(s)\circ [\CH(X)]=[\CH(W)], \quad \lim_{s\rightarrow+\infty} \sigma_\eta(s)\circ [\CH(W)]=[\CH(Y)].
\end{equation}
If we set 
\begin{equation}
\chw_\xi(X)=\min\{\la \alpha, \xi\ra; \Ch(X)_\alpha \neq 0\}, \quad \chw_\eta(W)=\min\{\la \alpha, \xi\ra; \Ch(W)_\alpha\neq 0\}
\end{equation}
then:
\begin{eqnarray*}
{[\CH(W)]}&=&[\sum_{\alpha\in I_W}\CH(X)_\alpha] \text{ where $I_W=\{\alpha; \CH(X)_\alpha\neq 0, \la \alpha, \xi\ra=\chw_\xi(X) \}$ }\\
{[\CH(Y)]}&=&[\sum_{\alpha\in I_Y}\CH(W)_\alpha] \text{ where $I_Y=\{\alpha; \CH(W)_\alpha\neq 0, \la \alpha, \eta\ra=\chw_\eta(W) \}$ } .
\end{eqnarray*} 
%\begin{equation}
%\sigma_{\xi+\epsilon \eta}(t)\circ \CH(X)=\sum_{\alpha} t^{\la \alpha, \xi+\epsilon \eta\ra} \CH(X)_\alpha.
%\end{equation}
Note that $I_Y\subseteq I_W$.
For any $\alpha\in M_\bZ$ with $\CH(X)_\alpha\neq 0$, $\la \alpha, \xi\ra \ge \chw_\xi(X)$ with equality iff $\alpha\in I_W$. 
Similarly for any $\alpha\in M_\bZ$ with $\CH(W)_\alpha\neq 0$ (and hence $\CH(X)_\alpha\neq 0$), $\la \alpha, \eta\ra \ge \chw_\eta(W)$ with equality iff $\alpha\in I_Y$.
So when $0<\epsilon\ll 1$ and for any $\CH(X)_\alpha\neq 0$, $\la \alpha, \xi+\epsilon \eta\ra \ge \chw_\xi(X)+\epsilon \chw_\eta(W)$ with equality iff $\alpha\in I_Y$. So we get
\begin{eqnarray*}
\lim_{t\rightarrow 0} \sigma_{\xi+\epsilon \eta}(t) \circ [\CH(X)]=\lim_{t\rightarrow 0}\left[\sum_{\alpha} t^{\la \alpha, \xi+\epsilon \eta\ra} \CH(X)_\alpha\right]=[\CH(Y)].
\end{eqnarray*}
So for $0<\epsilon \ll 1$, $\xi+\epsilon \eta$ induces an $\bR$-test configuration that degenerates $X$ to $Y$. By Lemma \ref{lem-integral}, we get the corresponding valuations $v_\epsilon$.
%mention the ideas and refer to the above references for more details.
%First realize both $\cF_v$ and $\mcW$ in a common projective space $\bP^{N-1}$ for $N\gg 1$. Let $\mcI_X, \mcI_W, \mcI_Y$ be the homogeneous ideal of $\bC[Z_1,\dots Z_N]$ defining $X$, $W$ and $Y$ respectively. Then $\xi$ (resp. $\eta$) defines a valuation and hence a filtration (or an order) on the polynomial algebra such that the initial term of homogeneous generators of $\mcI_X$ (resp. $\mcI_W$) with respect to this filtration generates $\mcI_W$ (resp. $\mcI_Y$). By using the finite generation of these homogeneous ideals, with respect to the order defined by $\xi_\epsilon$ for $0<\epsilon\ll 1$, one can show that the initial term of generators of $\mcI_X$ also generates $\mcI_Y$. 

Now we use the identity \eqref{eq-der1} to get: 
\begin{equation}
\left.\frac{d}{d\epsilon}\right|_{\epsilon=0}\tbeta(v_\epsilon)=\frac{d}{d\epsilon}\hat{\bfH}^\NA_{Y}(\cF_{\wt_{\xi+\epsilon\eta}})=\Fut_{Y,\xi}(\eta)<0.
\end{equation}
But this contradicts the assumption that $v_0=v$ is the minimizer of $\tbeta$.

Conversely, we need to show that Ding-semistable valuation is a minimizer. Let $(\mcX, \mcL, a\eta)$ be any special test configuration of $(X, -K_X)$ and $\cF_1=\cF_{(\mcX, \mcL, a\eta)}$ be the associated filtration. We consider the initial term degeneration of $\cF_1$ with respect to $\cF_0$ defined as above. Then we can use \eqref{eq-bfDdec} to get:
\begin{eqnarray*}
\hat{\bfH}^\NA_X(\cF_1)\ge \hat{\bfH}^\NA_{W}(\cF'_1)\ge \hat{\bfH}^\NA_{W}(\cF_{\wt_{\xi_0}})=\hat{\bfH}^\NA(\cF_0)=\tbeta(v).
\end{eqnarray*}
where the second inequality follows from the results in Lemma \ref{lem-der} in the next section and the assumption that $(W, \xi_0)$ is Ding-semistable.

\end{proof}

\section{Uniqueness of minimizing special valuations}\label{sec-semiunique}

We prove Theorem \ref{thm-semi} in this section.
We first generalize the formula \eqref{eq-der1}. Let $(X, -K_X, \bT, \xi)$ be the data as before and $\cF$ be a $\bT$-equivariant filtration. We consider a family of $\bT$-equivariant filtrations
\begin{equation}\label{eq-cFs}
\cF_s=s \cF_{\frac{1-s}{s}\xi},\;\; s\in (0,1]; \quad \cF_0=\cF_{\wt_\xi}, \quad \cF_1=\cF
\end{equation} 
that interpolates $\cF_{\wt_\xi}$ and $\cF$.
\begin{lem}\label{lem-der}
For the family of filtrations \eqref{eq-cFs}, the following statements hold true:
\begin{enumerate}
\item
$s\mapsto \hat{\bfH}^\NA(\cF_s)$ is convex. It is affine if and only if $G_\cF$ is a multiple of $\la x, \xi\ra$.

\item
We have the derivative formula:
\begin{equation}\label{eq-der2}
\left.\frac{d}{ds}\right|_{s=0}\hat{\bfH}^\NA(\cF_s)=\beta_\xi(\cF_{-\xi}).
\end{equation}
\end{enumerate}
\end{lem}
To get \eqref{eq-der1} from \eqref{eq-der2}, we just need to set $\cF=\cF_{\wt_{\xi+\eta}}$ so that $\cF_s=\cF_{\xi+s\eta}$.
Moreover, we fix a faithful valuation that is adapted to the torus action (see Definition \ref{def-adapt}) and will freely use the associated Newton-Okounkov body $\Delta=\Delta(-K_X)$ of $(X, -K_X)$.

\begin{proof}

By Lemma \ref{lem-Gtwist} and \eqref{eq-Gbasic}, as functions on $\Delta=\Delta(-K_X)$, we have:
\begin{equation}
G(s,y):=G_{\cF_s}(y)=(1-s)\la y, \xi\ra+s G_\cF(y).
\end{equation}
So, by using Lemma \ref{lem-abtrans}, we get:
\begin{eqnarray}
\hat{\bfL}^\NA(\cF_s)&=&s \hat{\bfL}^\NA(\cF)\\
-\btS^\NA(\cF_s)&=&\log\left(\frac{n!}{\bfV}\int_{\Delta} e^{-G(s, y)}dy\right).
\end{eqnarray}
$\hat{\bfL}^\NA(\cF_s)$ is linear in $s$. By H\"{o}lder's inequality, $-\btS^\NA(\cF_s)$ is strictly convex in $s$ unless $G_\cF$ is a multiple of $\la x, \xi\ra$. This implies that $\hat{\bfH}^\NA(\cF_s)=\bhL^\NA-\bfS^\NA$ is convex in $s\in [0,1]$. 

To see \eqref{eq-der2}, we calculate:
\begin{eqnarray*}
\left.\frac{d}{ds}\right|_{s=0} \hat{\bfH}^\NA(\cF_s)&=&\hat{\bfL}^\NA(\cF)+\frac{\int_\Delta (\la y, \xi\ra-G_{\cF}(y))e^{-G(0,y)}dx }{\int_\Delta e^{-G(0, y)}dy}\\
&=&\hat{\bfL}^\NA(\cF_{-\xi})-\frac{n!}{\bfV_\xi}\int_\Delta G_{\cF_{-\xi}}(y)e^{-\la y, \xi\ra}dy\\
&=&\beta_\xi(\cF_{-\xi}).
\end{eqnarray*}
\end{proof}

Assume that there are two special $\bR$-test configurations $\cF_i=\{\cF_i R_m\}, i=0,1$ of $(X, -K_X)$ that minimize $\bfH^\NA$ or equivalently $\hat{\bfH}^\NA$. By Theorem \ref{thm-minsemi}, the central fibers 
$(W^{(i)}:={\rm Proj}(\Gr_{\cF_i}), \xi_i=\xi_{\cF_i})$ are both Ding-semistable. 
%For simplicity of notations, set $(X', \xi')=({\rm Proj}(\Gr_{\cF_0}), \xi'=\xi_{\cF_0})$.  %Without loss of generality we assume $\cF_i$ are both normalized (see Definition \ref{def-normfil}).
Now consider the initial term degeneration of $\cF_1$ with respect to $\cF_0$ as in the above section. We get a $\bT_0$-equivariant filtration $\cF'_1$ on $R'=R(W^{(0)}, -K_{W^{(0)}})$ and by \eqref{eq-bfDdec} that $\hat{\bfH}^\NA_X(\cF_1)\ge \hat{\bfH}^\NA_{W^{(0)}}(\cF'_1)$.

Now as in the beginning of this section, consider the family of filtrations that interpolates $\cF'_1$ and $\cF_{\wt_{\xi_0}}R'=:\cF'_{\wt_{\xi_0}}$:
\begin{equation}
\cF'_s:=s \cF'_{\frac{1-s}{s}\xi_0} R'.
\end{equation}
Applying Lemma \ref{lem-der} to $(W^{(0)}, \xi_0, \cF'_s)$, we know that $\bhD(s):=\hat{\bfH}^\NA(\cF'_s)$ is convex in $s\in [0,1]$. 
%By Lemma \ref{lem-Gtwist} and \eqref{eq-Gbasic}, as functions on $\Delta(-K_{W^{(0)}})=:\Delta'$, we have:
%\begin{equation}
%G(s,x):=G_{\cF'_s}(x)=(1-s)G_{\cF'}(x)+s\la x, \xi_0\ra.
%\end{equation}
%By Lemma \ref{lem-abtrans}, we have:
%\begin{eqnarray}
%\bfL^\NA(\cF'_s)&=&s \bfL^\NA(\cF')\\
%\btS^\NA(\cF'_s)&=&-\log\frac{1}{\bfV}\int_{\Delta'} e^{-G(s, x)}dx
%\end{eqnarray}
%$\bfL^\NA(\cF'_s)$ is linear in $s$. By H\"{o}lder's inequality, $\btS^\NA(\cF'_s)$ is concave in $s$. This implies that 
%Because the successive minima of $\cF_1$ and $\cF'=\cF'_1$ coincides, we get
%\begin{equation}
%\btS^\NA_X(\cF_1)=\btS^\NA_{W^{(0)}}(\cF'_1).
%\end{equation}
%Since $\lct$ is lower semicontinuous for a family, we get
%\begin{equation}
%\bfL^\NA_X(\cF_1)\ge \bfL^\NA_{W^{(0)}}(\cF'_1).
%\end{equation}
%So we get the inequality $\bfH^\NA_X(\cF_1)\ge \bfH^\NA_{W^{(0)}}(\cF'_1)=\btD(1)$.  
Moreover we have relation:
\begin{equation}
\bhD(0)=\hat{\bfH}^\NA_{W^{(0)}}(\cF_{\wt_{\xi_0}})=\hat{\bfH}^\NA_X(\cF_0)=\hat{\bfH}^\NA_X(\cF_1)\ge \hat{\bfH}^\NA(\cF'_1)=\bhD(1).
\end{equation}
The 3rd identity is by Theorem \ref{thm-minsemi} that $\cF_i, i=0,1$ both obtains the minimum of $\hat{\bfH}^\NA$.

On the other hand, by \eqref{eq-der2}
\begin{eqnarray*}
\left.\frac{d}{ds}\right|_{s=0} \hat{\bfH}^\NA(\cF'_s)%&=&\bfL^\NA(\cF')+\frac{\int_\Delta (\la x, \xi_0\ra-G_{\cF'}(x))e^{-\la x, \xi_0\ra}dx }{\int_X e^{-\la x, \xi_0\ra}dx}\\
%&=&\bfL^\NA(\cF'_{-\xi_0})-\frac{1}{V_{g_{\xi_0}}}\int_\Delta G_{\cF'_{-\xi_0}}(x)e^{-\la x, \xi_0\ra}dx\\
&=&\beta_{\xi_0}(\cF'_{-\xi_0})\ge 0.
\end{eqnarray*}
The last inequality is because $(W^{(0)}, \xi_0)$ is Ding-semistable.

By convexity of $\bhD(s)$, we conclude that $\bhD(s)$ is constant in $s$ and by Lemma \ref{lem-der} that $G_{\cF'_1}(y)\equiv \la y, \xi_0\ra$ for any $y\in \Delta'=\Delta(W^{(0)}, -K_{W^{(0)}})$ (the Okounkov body of $(W^{(0)}, -K_{W^{(0)}})$).

By the discussion in previous section, we know that %since $\cF'$ is $\bT_0$-invariant and has the same successive minima as $\cF_1$, it is easy to see that 
the relative successive minima of $\cF_1$ with respect $\cF_0$ is the same as the relative successive minima of $\cF'_1$ with respect to $\cF'_{\wt_{\xi_0}}$, which is the same as the successive minima of $\cF'_{-\xi_0}$ and is given by the difference $\lambda_k-\la \alpha_k, \xi_0\ra$ with the notations there. So we get by \eqref{eq-d2inv}
\begin{eqnarray*}
d_2(\cF_0, \cF_1)^2&=&d_2(\cF'_0, \cF'_1)=\lim_{m\rightarrow+\infty}\sum_k \frac{(\lambda_k-\la \alpha_k, \xi_0\ra)^2}{m^2}=\lim_{m\rightarrow+\infty}\sum_i \frac{\lambda_i^{(m)}(\cF'_{-\xi_0})^2}{m^2}\\
&=&\int_\bR \lambda^2 \DHM(\cF'_{-\xi_0})^2=\int_{\Delta'}G_{\cF'_{-\xi_0}}^2 dy\\
&=&\int_{\Delta'}(G_{\cF'}-\la y, \xi_0\ra)^2dy=0.
\end{eqnarray*}
By \cite{BoJ18b}, we know that $\cF_0$ is equivalent to $\cF_1$. By Lemma \ref{lem-integral} and Proposition \ref{prop-equiv}, we get $\cF_0=\cF_1$. 

\begin{rem}\label{rem-rank1}
Although here we are dealing  with filtration of arbitrary ranks, the unique result in this section (and minimization result in previous section) can also be proved by using $r:=\rk(\cF_0)$-step degenerations to reduce to the rank 1 case. To see this, we first choose $\{\eta_1,\dots, \eta_r\}\in N_\bQ\cong \bQ^{r}$ (where $N={\rm Hom}(\bC^*, \bT_0)$ as before) such that 
\begin{itemize}
\item ${\rm Span}_{\bR}\{\eta_1, \dots, \eta_r\}=N_\bR$.
\item For any $1\le k\le r$, $\eta_k$ induces a special test configuration whose central fibre is the same as $W^{(0)}$. This is achieved by choosing $\eta_k$ satisfying $|\eta_k-\xi_0|\ll 1$.
\end{itemize}
By abuse of notations, we denote by $\cF'_{\xi_0}$ (resp. $\cF'_{\eta_1}$) the filtration on $R=R(X, -K_X)$ corresponding to the $\bR$-test configuration induced by $\xi_0$ (resp. $\eta_1$), and also the filtration on $R'=R(W^{(0)}, -K_{W^{(0)}_0})$ corresponding to the weight filtration induced by $\xi_0$ (resp. $\eta_k$ for $2\le k\le r$). Set $\cF'^{(0)}_1=\cF_1$ and we define inductively $\cF'^{(k)}_1$ to be the initial term degeneration of $\cF'^{(k-1)}_1$ with respect to $\cF'_{\eta_k}$ for $1\le k\le r$. By \eqref{eq-bfDdec} for the rank 1 case, we have: 
\begin{equation}
\hat{\bfH}^\NA_X(\cF'^{(0)}_1)\ge \hat{\bfH}^\NA_{W^{(0)}}(\cF'^{(1)}_1), \quad 
\hat{\bfH}^\NA_{W^{(0)}}(\cF'^{(k-1)}_1)\ge \hat{\bfH}^\NA_{W^{(0)}}(\cF'^{(k)}_1), \quad 2\le k\le r.
\end{equation}
So if $\cF_1=\cF'^{(0)}_1$ obtains the minimum of $\hat{\bfH}^\NA_X$, then $\cF'^{(k)}$ for any $1\le k\le r$ also obtains the minimum of $\hat{\bfH}^\NA_{W^{(0)}}$.
Now because $\cF'^{(r)}$ is $\bT_0$-invariant and $\cF'_{\xi_0}=\cF'_{\wt_{\xi_0}}$ also obtains the minimum of $\hat{\bfH}^\NA_{W^{(0)}}$, we can use Lemma \ref{lem-twistconvex} to conclude that $\cF'^{(r)}=\cF'_{\xi_0}$.

On the other hand, by \eqref{eq-d2inv}, we get for $2\le k\le r$, 
\begin{eqnarray*}
d_2^X(\cF'^{(0)}_1, \cF'_{\eta_1})=d_2^{W^{(0)}}(\cF'^{(1)}_1, \cF'_{\eta_1}), \quad 
d_2^{W^{(0)}}(\cF'^{(k-1)}_1, \cF'_{\eta_k})=d_2^{W^{(0)}}(\cF'^{(k)}_1, \cF'_{\eta_k}).
\end{eqnarray*}
So for any $1\le k\le r$, we get, by omitting the upperscripts and using the triangle inequality,
\begin{eqnarray*}
d_2(\cF'^{(k-1)}_1, \cF'_{\xi_0})&\le& d_2(\cF'^{(k-1)}_1, \cF'_{\eta_k})+d_2(\cF'_{\eta_k}, \cF'_{\xi_0})\\
&=& d_2(\cF'^{(k)}_1, \cF'_{\eta_k})+d_2(\cF'_{\eta_k}, \cF'_{\xi_0})\\
&\le& d_2(\cF'^{(k)}_1, \cF'_{\xi_0})+2 d_2(\cF'_{\eta_k}, \cF'_{\xi_0})
\end{eqnarray*}
So we can inductively estimate:
\begin{eqnarray*}
d_2(\cF_1, \cF_0)&=&d_2(\cF'^{(0)}_1, \cF'_{\xi_0})\le d_2(\cF'^{(1)}_1, \cF'_{\xi_0})+2 d_2(\cF'_{\eta_1}, \cF'_{\xi_0})\\
&\le& d_2(\cF'^{(2)}_1, \cF'_{\xi_0})+2 \left(d_2(\cF'_{\eta_2}, \cF'_{\xi_0})+d_2(\cF'_{\eta_1}, \cF'_{\xi_0})\right)\\
&\le &\cdots \le  d_2(\cF'^{(r)}_1, \cF'_{\xi_0})+2 \sum_{k=1}^r d_2(\cF'_{\eta_k}, \cF'_{\xi_0})\\
&=&2 \sum_{k=1}^r d_2(\cF'_{\eta_k}, \cF'_{\xi_0}).
\end{eqnarray*}
Now we can choose $\eta_k$ such that $d_2(\cF'_{\eta_k}, \cF'_{\xi_0})$ is arbitrarily small for all $1\le k\le r$. So we indeed get $d_2(\cF_1, \cF_0)=0$ as desired.

\end{rem}

\section{Cone construction and $g$-normalized volume}

Let $X$ be an $n$-dimensional $\bQ$-Fano variety and for simplicity of notations assume that $-K_X$ is Cartier.
Recall that $R=\oplus_m R_m=\bigoplus H^0(X, m(-K_X))$. We define the cone 
\begin{equation}
C=C(X, -K_X)={\rm Spec}_\bC R, \quad o=\mathfrak{m}=\bigoplus_{m>0}R_m.
\end{equation}
Then $(C, o)$ is a klt cone singularity. Denote by $\Val_{C,o}$ be the space of real valuations that are centered at $o$. Since $X$ admits a $\bC^*\times \bT$-action, we have a decomposition of the coordinate ring of $R$:
\begin{equation}
R=\bigoplus_{m\ge 0}\bigoplus_{\alpha\in \bZ^r} R_{m,\alpha}.
\end{equation} 
For any $\bT$-invariant homogeneous primary ideal $\fa=\bigoplus_{m}\bigoplus_{\alpha} \fa_{m,\alpha}\subset R$, define the $g$-length and $g$-multiplicity of $\fa$:
\begin{eqnarray}
\col_g(\fa)&=&\sum_{m\ge 0} \sum_\alpha g(\frac{\alpha}{m}) \dim R_{m,\alpha}/\fa_{m,\alpha}\\
\mult_g(\fa)&=&\lim_{k\rightarrow+\infty}\frac{\col_g(\fa^k)}{k^{n+1}/(n+1)!}.
\end{eqnarray}
See \cite{Ros89} for the study of such equivariant multiplicity.
More generally, let $\fa_\bullet=\{\fa_k\}_{k\in \bN}$ be a graded sequence of $\bC^*\times\bT$-invariant ideals. We define:
\begin{equation}
\mult_g(\fa_\bullet)=\lim_{k\rightarrow+\infty} \frac{\col_g(\fa_k)}{k^{n+1}/(n+1)!}
\end{equation}

One can use the techniques of Newton-Okounkov bodies to show that the limit exists. To see this, we can adapt the argument in the work in \cite{KK14} as follows. First choose a valuation $\fv$ adapted to the $\bT$-action on $X$. We can construct a  $\bC^*\times\bT$-invariant $\bZ^{n+1}$-valuation on $C$: 
\begin{equation}
\text{for any } f\in R_m, \quad
\mathfrak{V}\left(f\right)=(m, \fv(f)).
\end{equation}
Denote by $\mathfrak{C}$ the strongly convex cone which is the closure of the convex hull of the value semigroup $\mathfrak{V}(R)$. To each graded sequence of $\bC^*\times\bT$-invariant ideals $\fa_\bullet$, one can associate a convex region $\bar{P}:=\bar{P}(\fa_\bullet)\subset \mathfrak{C}$ such that $\bar{P}^c:=\mathfrak{C}\setminus \bar{P}$ is bounded. If we still denote by $g(y)$ the pull-back of function $g$ by the projection $\bR^{n+1}=\bR\times\bR^n\rightarrow \bR^n$. Then $\mult_g$ is given by the weighted volume of the co-convex set $\bar{P}^c$:
\begin{equation}
\mult_g(\fa_\bullet)=(n+1)! \int_{\bar{P}^c} g(y)dy
\end{equation}
Let $\bar{v}\in \Val_{C,o}^{\bC^*\times\bT}$ be a $\bC^*\times \bT$-invariant valuation. Then for any $\lambda\in \bR$, $\fa_\lambda(\bar{v})=\{f\in R; \bar{v}(f)>m\}$ is a $\bT$-invariant homogeneous primary ideal. Set $\fa_\bullet(\bar{v})=\fa_m(\bar{v})$ and define (see \cite{ELS03} for the $g=1$ case)
\begin{eqnarray*}
\vol_g(\bar{v}):=\mult_g(\fa_\bullet(v))=\lim_{m\rightarrow+\infty}\frac{\col_g(\fa_\lambda(\bar{v}))}{\lambda^{n+1}/(n+1)!}.
\end{eqnarray*}

We define the following equivariant version of normalized volume \cite{Li18}:
\begin{eqnarray*}
\hvol_g: \Val^{\bC^*\times\bT}_{C,o}&\rightarrow&\bR_{>0}\cup \{+\infty\}\\
\hvol_g(\bar{v})&= &
\left\{
\begin{array}{ll}
A_C(\bar{v})^{n+1} \cdot \vol_g(\bar{v}) & \text{ when } A_C(\bar{v})<+\infty \\
+\infty & \text{ otherwise. }
\end{array}
\right.
\end{eqnarray*}
By using the same argument as in the study of normalized volumes, one can generalize almost all the results about normalized volume to work for the $g$-normalized volume functional. Here we just write down a few results that we need in the next section.
We have the following equivariant version of an identity from \cite{Liu18}.  
\begin{lem}
With the above notations, we have the following identity:
\begin{equation}
\inf_{\bar{v}}\hvol_g(\bar{v})=\inf_{\fa}\lct(\fa)^n\cdot \mult_g(\fa)=\inf_{\fa_\bullet} \lct(\fa_\bullet)^n\cdot \mult_g(\fa_\bullet),
\end{equation}
where $\bar{v}$ ranges over $\bC^*\times \bT$-invariant valuations, and  $\fa$ (resp. $\fa_\bullet$) ranges over $\bC^*\times\bT$-invariant ideals (resp. graded sequence of $\bC^*\times\bT$-invariant ideals).
\end{lem}
This is proved by using exactly the same argument. For the reader's convenience, we give the short proof.
\begin{proof}
For any $\bar{v}\in \rVal_{C,o}$, we have:
\begin{equation}
\lct(\fa_\bullet(\bar{v}))^n \cdot \mult_g(\fa_\bullet(\bar{v}))\le \left(\frac{A_C(\bar{v})}{\bar{v}(\fa_\bullet)}\right)^n \vol_g(\bar{v})=A_C(\bar{v})^n \vol_g(\bar{v}).
\end{equation}
Conversely, for any graded sequence of ideals $\fa_\bullet$. Let $\bar{w}\in \rVal_{C,o}$ be the valuation that calculates $\lct(\fa_\bullet)$ which exists by \cite{JM12}. By multiplying a constant, we can assume $1=\bar{w}(\fa_\bullet)=\inf_m \frac{1}{m}\bar{w}(\fa_m)$. So $\fa_m\subseteq \fa_m(\bar{w})$, which implies $\mult_g(\fa_\bullet)\ge \mult_g(\fa_\bullet(\bar{w}))=\vol_g(\bar{w})$.
Then we get:
\begin{equation}
\lct(\fa_\bullet)^n\cdot \mult_g(\fa_\bullet)=\left(\frac{A_C(\bar{w})}{\bar{w}(\fa_\bullet)}\right)^n\cdot \mult_g(\fa_\bullet)\ge A_C(\bar{w})^n\cdot \vol_g(\bar{w})=\hvol_g(\bar{w}).
\end{equation}
\end{proof}

For any $v\in X^{\rm div}_\bQ$ and $\tau>0$, we denote by $\bar{v}_\tau$ the $\bC^*$-invariant valuation on $C$ given by:
\begin{equation}
\bar{v}_\tau\left(\sum_i f_i t^i\right)=\min_i (v(f_i)+\tau i).
\end{equation}
By using the same calculation as in \cite{Li17}, we get:
\begin{thm}
We have the following formula for the $g$-volume of $\bar{v}_\tau$:
\begin{eqnarray}
\vol_g(\bar{v}_\tau)&=&\frac{1}{\tau^{n+1}}\bfV_g-(n+1)\int_0^{+\infty}\vol_g(\cF_v R^{(x)})\frac{dx}{(x+\tau)^{n+2}}.
\end{eqnarray}
\end{thm}
We have the following criterion for $g$-Ding-semistability, which generalizes the results in \cite{Li17, LL19, LX16} about normalized volumes.
\begin{thm}
$(X, \eta)$ is $g$-Ding-semistable if and only if $\ord_{X}$ obtains the minimum of $\hvol_g$ over $\Val_{C,o}^{\bC^*\times\bT}$.
\end{thm}

\begin{proof}
For any $v\in (X^{\rm div}_\bQ)^{\bT}$, consider $w_s:=\overline{(sv)}_{(1-s)A_X(v)}\in \Val_{C,o}^{\bC^*\times\bT}$. Then $w_0=A_X(v) \bar{v}_0$ and $w_1=v$.
$A_C(w_s)\equiv A_X(v)$.
Set 
\begin{eqnarray*}
f(s)&=&\hvol(w_s)=A_C(w_s)^{n+1}\vol_g(w_s)\\
&=&A_X(v)^{n+1}\left(\frac{\bfV_g}{(1-s)^{n+1}A_X(v)^{n+1}}-(n+1)\int_{0}^{+\infty} \vol_g(\cF_v R^{(x)})\frac{sdx}{(sx+(1-s)A_X(v))^{n+2}}\right)\\
&=&A_X(v)^{n+1} \int_0^{+\infty} \frac{-d \vol_g(\cF_v R^{(x)})}{(sx+(1-s)A_X(v))^{n+1}}.
\end{eqnarray*}
Then $f(s)$ is a convex function in $s\in [0,1]$. Its derivative at $s=0$ is given by:
\begin{eqnarray}
f'(0)&=&A_X(v)^{n+1}\left((n+1)\frac{\bfV_g}{A_X(v)^{n+1}}-(n+1)\int_0^{+\infty}\vol_g(\cF_v R^{(x)})dx \frac{1}{A_X(v)^{n+2}}\right)\nonumber \\
&=&\frac{n+1}{A_X(v)\bfV_g}\left(A_X(v)-\frac{1}{\bfV_g}\int_0^{+\infty} \vol_g(\cF_v R^{(x)})dx\right)\nonumber \\
&=&\frac{n+1}{A_X(v)\bfV_g}\cdot \beta_g(v). \label{eq-derhvg}
\end{eqnarray}
With this and Theorem \ref{thm-valgsemi}, we can easily derive the conclusion as in \cite{Li17}.
\end{proof}

\begin{rem}\label{rem-open}
By the same argument as in the case of normalized volume \cite{BL18a, Xu20}, one shows that $g$-Ding-semistability is Zariski openness for a $\bT$-equivariant family of Fano varieties.
\end{rem}

\section{Uniqueness of polystable degeneration}

In this section, prove Theorem \ref{thm-poly}. 
The proof is verbatim the same as the proof of the existence and uniqueness of K-polystable degenerations for any K-semistable $\bQ$-Fano varieties as proved in \cite{LWX18} (see also \cite{BLZ19}). Indeed we just need to carry out the same argument by using the equivariant version of normalized volume and the modified Futaki-invariant $\Fut_\xi$ etc. To avoid redundancy, we only sketch the key steps and refer to \cite{LWX18, BLZ19} for more details.

Assume that $(X, \xi)$ is semistable and admits two polystable degeneration via two special test configuration $(\mcX^{(i)}, -K_{\mcX^{(i)}}), i=0,1$. 
Take cones fibrewisely to get a special test configuration of Fano cones $(\cC^{(i)}), \zeta^{(i)})$ where $\zeta^{(i)}$ is the radial vector field. 

Let $E_k$ be the Koll\'{a}r component (see \cite{LX16} for the definition) obtained by blowing up the vertex of $\cC^{(0)}$ with weight $(k, 1)$. Then we have:
\begin{eqnarray*}
\hvol_g(E_k)=\hvol_g(\ord_X)+O(k^{-2}).
\end{eqnarray*} 
Set $\fa_\bullet=\{\fa_\ell(E_k)\}$. Then 
\begin{eqnarray*}
&&\lct(\fa_\bullet)=\frac{A(E_k)}{\ord_{E_k}(\fa_\bullet)}=A(E_k)=:c_k=O(k)\\
&&\lct(X, \fa_\bullet)^n\cdot \mult_g(\fa_\bullet)=\hvol_g(E_k). 
\end{eqnarray*}
Consider the initial degeneration of $\fa_\bullet$ with respect to $\cC^{(1)}$:
\begin{equation}
\bin(\fa_\ell)={\rm span}_\bC\{\bin(f), f\in \fa_\ell(E_k)\}.
\end{equation}
Using the preservation of co-length under initial term degeneration, we get:
\begin{eqnarray*}
\lct^n(\cC^{(1)}_0, \bin(\fa_\bullet))&\ge& \frac{\hvol_g(\ord_{\mcX^{(1)}_0})}{\mult_g(\bin(\fa_\bullet))}=\frac{\bfV_g}{\mult_g(\fa_\bullet)}\\
&=&\frac{\bfV_g}{\hvol_g(E_k)}\lct(\fa_\bullet)^n=\frac{\bfV_g}{\bfV_g+O(k^{-2})}\lct(\fa_\bullet)^n\\
&=&(1+O(k^{-2}))c_k=c_k+O(k^{-1}).
\end{eqnarray*}

Let $Z_k\rightarrow \cC^{(0)}$ be the extraction of $E_k$ and $Z_k\times \bC^*$ be the product along $\cC^{(1)}\setminus \cC^{(1)}_0\cong C\times\bC^*$ with exceptional divisor $\cE_k$. Let $\mathfrak{B}_\bullet=\{\mathfrak{B}_\ell\}$ be ideal on the total space $\cC^{(1)}$ obtained by the above degenerating $\fa_\ell$.  Then we have:
\begin{eqnarray}
A(\cC^{(1)}, c_k(1-\epsilon k^{-1}) \mathfrak{B}_\bullet, \cE_k)&=&A(C, c_k(1-\epsilon k^{-1}) \fa_\bullet, E_k)\\
&=& \epsilon k^{-1} c_k=\epsilon O(1),\\
\lct(\cC^{(1)}_0, c_k (1-\epsilon k^{-1})\bin(\fa_\bullet))&\ge&c_k^{-1} (1-\epsilon k^{-1})(c_k+O(k^{-1}))\nonumber \\
&=&=1-\epsilon k^{-1}+O(k^{-2}).
\end{eqnarray}
By inversion of adjunction $\lct(\cC^{(1)}, c_k(1-\epsilon k^{-1})\mathfrak{B}_\bullet)\ge 1-\epsilon k^{-1}+O(k^{-2})$. When $0<\epsilon \ll 1$, by \cite{BCHM}, we can extract the divisor $\cE_k$ over $\cC^{(1)}$. By the same argument as \cite{LWX18}, we get the commutative diagram:

\begin{equation}
\xymatrix @1 @R=1.2pc @C=1.2pc
{
C^{(1)}_0 \ar@{~>}_{\mcC'^{(1)}}[dddd]
 \ar@{-->}[dr] &   & &  & C \ar_{ \cC^{(1)}\longleftarrow \cZ^{(1)}_k\longleftarrow\cE^{(1)}_k}@{~>}[llll] \ar^{\cC^{(0)} \leftarrow \cZ_k\leftarrow \cE_k=E_k\times\bA^1}@{~>}[dddd]  \ar@{-->}[ld] & \ar[l] Z_k\leftarrow E_k
 \\
& X^{(1)}_0 \ar@{~>}_{\cX'^{(1)}}[dd] &  &  X \ar_<<<<<<<{\cX^{(1)}}@{~>}[ll] \ar^{\cX^{(0)}}@{~>}[dd] & & \\
&   &   &\\
&                              X'_0   &               & X^{(0)}_0  \ar@{~>}^{\stackrel{}{\cX'^{(0)}}}[ll] & & \\
C'_{0}  \ar@{-->}[ru] & & & & C^{(0)}_0 \ar@{~>}^{\cC'^{(0)}}[llll] \ar@{-->}[lu] & \ar[l] Z_{k,0}\leftarrow E_k .
}
\end{equation}
By the same argument as in \cite{LWX18}, we know that both test configurations $\mcX'^{(i)}, i=0,1$ are weakly special and have vanishing $\Fut_\xi$-invariant. By \cite{HL20}, we know that both of them are special and hence $X^{(1)}_0\cong X'_0\cong X^{(0)}_0$ by the polystability of $X^{(i)}_0$.  

The existence part can again be proved by the similar arguments as in \cite{LWX18} which deals with the case when $\xi=0$. We just sketch the arguments. If $(X, \xi)$ is K-polystable, then we are done. Otherwise, we can find a nontrivial $\bT$-equivariant special test configuration such that the central fibre (with the vector field $\xi$) has a vanishing $\Fut_\xi$-invariant. By \cite[Proof of Lemma 3.1]{LWX18}, we know that the central fibre is K-semistable, and has an effective action by a larger torus. If the central fibre is K-polystable, then we are done again. Otherwise, we can continue this process which must stop since the dimension of the torus is bounded by the dimension of $X$.

\begin{proof}[Proof of Corollary \ref{cor-conj}]
By the work of Chen-Sun-Wang in \cite{CSW18} which is based on the resolution of the Hamilton-Tian's conjecture \cite{CW14}, we get a special $\bR$-test configuration $\cF^{ss}$ with central fibre $(W, \xi)$, and a special test configuration of $(W, \xi)$ with central fibre $(X_\infty, \xi)$ which admits a K\"{a}hler-Ricci soliton and hence K-polystable. By the work of Dervan-Sz\'{e}kelyhidi, $\cF^{ss}$ obtains the minimum $h(X)$. The statement follows directly from Theorem \ref{thm-semi} and Theorem \ref{thm-poly}. 

\end{proof}
\begin{rem}
%In this setting, we know that $v_{\cF^{ss}}$ obtains the minimum of $\tbeta$. 
The fact that $\cF^{ss}$ obtains the minimum also follows from the K-semistability of $(W, \xi)$ and Theorem \ref{thm-minsemi}. The K-semistability of $(W, \xi)$ follows from the same degeneration argument as used in \cite{LX16} or the Zariski openness of K-semistability as pointed out in Remark \ref{rem-open}.
\end{rem}

\begin{rem}
As in the more general setting of \cite{LWX18} or \cite{Li19}, the algebraic results in this paper can be generalized to the log Fano case in a straightforward way.
\end{rem}

\appendix 

\section{Appendix: Properties of $\btS(v)$}\label{sec-cont}
 Recall that \eqref{eq-bV1form} that:
\begin{equation}
\bfQ(v):=\bfQ(\cF_v)=e^{-\btS^\NA(\cF_v)}=1-\frac{1}{\bfV}\int_0^{T(v)} e^{-x}\vol(\cF^{(x)}_v R_{\bullet})dx=:1-\Psi(v)
\end{equation}
where, for simplicity of notations, we denote 
\begin{equation}
T(v)=\lambda_{\max}(\cF_v), \quad 
\Psi(v)=\frac{1}{\bfV}\int_0^{T(v)}e^{-x}\vol(\cF_v^{(x)}R_\bullet)dx.
\end{equation}
\begin{prop}\label{prop-strictinc}
The function $v\mapsto \Psi(v)$ is strict increasing on $\Val(X)$. In other words, if $v\le w$, then $\Psi(v)\le \Psi(w)$ with the identity true only if $v=w$. As a consequence, $v\mapsto \btS(v)$ is strictly increasing on $\Val(X)$.
\end{prop}
This is proved as in \cite[Proof of Proposition 3.15]{BlJ17} (which is based on an argument in the local case from \cite{LX16}). We sketch the argument for the reader's convenience.
\begin{proof}
First, by using Proposition \ref{thm-BoCh} we can show that:
\begin{equation}
\Psi(v)=\lim_{m\rightarrow+\infty}\frac{1}{mN_m}\sum_{j\ge 1}e^{-\frac{j}{m}}\dim \cF^{j}_vR_m.
\end{equation}
Suppose $v\le w$ but $v\neq w$. Then by rescaling $v, w$ and $L=-K_X$, we can assume that there exists $s\in H^0(X, L)$ with $w(s)=p\in \bN^*$ and $v(s)\le p-1$. Then arguing as in \cite[Proof of Proposition 3.15]{BlJ17}, we have:
\begin{equation}
\dim(\cF^j_w R_m/\cF^j_v R_m)\ge \sum_{1\le i\le \min\{j/p, m\}} \dim(\cF^{j-ip}_vR_m/\cF^{j-ip+1}_vR_{m-i}).
\end{equation}
One the other hand, with $C=\max\{T(v), T(w)\}$, we get
\begin{eqnarray*}
\sum_{j\ge 1}\dim e^{-j/m}(\cF^j_wR_m-\cF^j_v R_m)&\ge& e^{-C} \sum_{j\ge 1}(\cF^j_w R_m-\cF^j_vR_m)\\
&\ge&e^{-C}\sum_{1\le i\le m}\sum_{j\ge pi}(\dim\cF^{j-ip}_vR_{m-i}-\cF^{j-ip+1}_v R_{m-i})\\
&=&e^{-C}\sum_{1\le i\le m}\dim R_{m-i}.
\end{eqnarray*}
So we conclude:
\begin{eqnarray*}
\Psi(v)-\Psi(w)\ge e^{-C}\lim_{m\rightarrow+\infty}\frac{1}{mN_m}\sum_{1\le i\le m}\dim R_{m-i}>0.
\end{eqnarray*}

\end{proof}

Let $\pi: Y\rightarrow X$ be a proper birational morphism with $Y$ a regular and $E=\sum_i E_i$ a reduced simple normal crossing divisor.

\begin{prop}\label{prop-contqm}
The function $v\mapsto \bfQ(v)$ is continuous on ${\rm QM}(Y, E)$. 
\end{prop}

We use the same strategy as \cite[Proposition 2.4]{BLX19}.
As noted in \cite{FO18}, for any $v\in \Val(X)$, we have $\frac{A(v)}{T(v)}\ge \alpha(X)>0$
which implies with $C=\alpha(X)^{-1}$, 
\begin{equation}\label{eq-TAv}
T(v)\le C A(v).
\end{equation}
\begin{lem}
For any $v\in \Val(X)$, we have the inequality:
\begin{equation}
\Psi(v)\le C A(v).
\end{equation}
\end{lem}
\begin{proof}
Because $\vol(\cF^{(x)} R_\bullet)\le \bfV$, we immediately get:
\begin{eqnarray*}
\Psi(v)\le \int_0^{T(v)}e^{-x}dx=1-e^{-T(v)}\le T(v)\le C A(v),
\end{eqnarray*}
where we used the inequality $1-e^{-x}\le x$ for any $x\in \bR_{\ge 0}$ and the inequality \eqref{eq-TAv}. 
\end{proof}
Similar to \cite{FO18, BlJ17}, we introduce the following approximation:
\begin{eqnarray}
\bfQ_m(\cF)&=&\frac{1}{N_m}\sum_{i} e^{-\frac{\lambda^{(m)}_i}{m}}=\frac{1}{N_m}\int_0^{+\infty} e^{-\frac{x}{m}} d\left(-\dim \cF^{x} R_m\right)
\\%, \quad \btS_m(\cF)=-\log \bfQ_m(\cF)\\
&=&1-\frac{1}{N_m}\int_0^{\frac{\lambda^{(m)}_{\max}(\cF)}{m}}e^{-x} \dim \cF^{xm} R_m dx\\
&=:&1-\Psi_m(\cF),
\end{eqnarray}
where we set:
\begin{eqnarray*}
\Psi_m(v)&=&\frac{1}{N_m}\int_0^{\lambda^{(m)}_{\max}}e^{-x} \dim \cF^{xm} R_m dx\\
&=&\frac{1}{N_m}\int_0^{T(v)}e^{-x} \dim \cF^{xm} R_m dx.
\end{eqnarray*}
Similar to \cite{FO18, BlJ17}, for any valuation $v\in \Val(X)$, we have the identity:
\begin{equation}
\bfQ_m(v)=\bfQ_m(\cF_v)=\min_{\{s_j\}} \frac{1}{N_m}\sum_{j=1}^{N_m} e^{-\frac{v(s_j)}{m}}
\end{equation}
where the minimum is taking over all bases $s_1,\dots, s_{N_m}$ of $H^0(X, -mK_X)$.

For any $\mathfrak{s}:=\{s_1,\dots, s_{N_m}\}\in H^0(X, -mK_X)^{N_m}$, define a function 
\begin{equation}
\vphi_{\mathfrak{s}}(v):=\sum_{j=1}^{N_m}e^{-\frac{v(s_j)}{m}}.
\end{equation}
By the same argument as in \cite[Proof of Lemma 2.5]{BLX19}, the set of functions $\{\vphi_{\mathfrak{s}}(v)| \mathfrak{s}\in R_m^{N_m}\}$ is finite. So $\bfQ_m$ is continuous on ${\rm QM}(Y, E)$.

As in \cite[Proof of Proposition 2.4]{BLX19}, the continuity of $\Psi$ and hence $\bfQ$ follows easily from the following proposition, which we  prove by using the techniques developed in \cite{BlJ17, BL18}:
\begin{lem}\label{prop-Qmconv}
\begin{enumerate}
\item For any $v\in \Val(X)$ with $A(v)<+\infty$, we have the convergence: 
\begin{equation}
\lim_{m\rightarrow+\infty} \Psi_m(v)=\Psi(v). %, \quad \lim_{m\rightarrow+\infty}\btS_m(\cF)=\btS^\NA(\cF).
\end{equation}

\item 
For any $\epsilon>0$ and any $C_1>0$, there exists $C_2>0$ and $m_0>0$ such that if $v\in \Val(X)$ satisfies $A(v)<C_1$, we have:
\begin{equation}
\left|\Psi_m(\cF_v)-\Psi(\cF_v)\right|\le \epsilon
\end{equation}
for all $m$ divisible by $m_0$.
\end{enumerate}
\end{lem}
\begin{proof}
The first statement follows from Theorem \ref{thm-BoCh}.2. We focus on the second statement. 

Note that $e^{-G}$ is convex and $0\le e^{-G}\le 1$. By \cite[Lemma 2.2]{BlJ17}, for any $\epsilon'>0$, there exists $m_0(\epsilon')$ such that for any $m\ge m_0$, 
\begin{equation}
\int_\Delta e^{-G} d\rho_m\ge \int_\Delta e^{-G}dy-\epsilon'.
\end{equation}
By the same argument as \cite[Proof of Lemma 2.9]{BlJ17}, we get
\begin{equation}
\bfQ_m(\cF_v)\ge \frac{m^n}{N_m}\int_\Delta e^{-G}d\rho_m.
\end{equation}
Note that $\lim_{m\rightarrow+\infty}\frac{m^n}{N_m}=V$.
So for any $\epsilon>0$ there exists $m_0$ such that for any $m\ge m_0$, 
\begin{equation}
\bfQ_m(\cF_v)\ge \frac{n!}{\bfV}\int_\Delta e^{-G}dy-\epsilon=\bfQ(\cF_v)-\epsilon.
\end{equation}

We need to prove the other direction of inequality.
Following \cite{BlJ17}, define a graded linear series:
\begin{equation}
\tilde{\cF}^{(t)}_{m, p}R_{mp}:=H^0\left(X, mpL\otimes \overline{\mathfrak{b}(|\cF^{mt} R_m |)^p}\right)
\end{equation}
where $\mathfrak{b}(|\cF^{mt}R_m|)$ is the base ideal of the sub-linear system $\cF^{tm} R_m$. Set:
\begin{equation}
\tilde{\Psi}_m(\cF)=\int_0^{T(v)}e^{-t}\vol(\tilde{\cF}^{(t)}_{m,\bullet})dt.
\end{equation}
By \cite[Proposition 5.13]{BlJ17}, there exists $a=a(X, -K_X)>0$ such that for all $t\in \bQ_{>0}$ with $mt>A(v)$, we have, with $t'=t-m^{-1}(at+A(v))$,
\begin{equation}
\left(\frac{m-a}{m}\right)^{n+1}\vol(\cF^{(t)}_{\bullet})\le \frac{1}{m^n}\vol(\tilde{\cF}^{(t')}_{m,\bullet}).
\end{equation}
So we can estimate as \cite[Proof of Proposition 5.15]{BlJ17} to get:
\begin{eqnarray*}
\tilde{\Psi}_m(v)&\ge& \left(\frac{m-a}{m}\right)^{n+1}\left(\Psi(v)-e^{\frac{a T(v)+A(v)}{m}}\int_0^{A(v)/(m-a)}\frac{\vol(\cF^{(t)})}{V}e^{-t}dt\right)\\
&\ge& \left(\frac{m-a}{m}\right)^{n+1}\left(\Psi(v)-e^{C A(v)/m}\frac{A(v)}{m-a}\right).
\end{eqnarray*}
From this it is easy to get:
\begin{eqnarray*}
\Psi(v)-\tilde{\Psi}_m(v)\le C \frac{A(v)}{m}.
\end{eqnarray*}
To compare with $\Psi_m$, we further set
\begin{equation}
\cF^{(x)}_{m,p}={\rm Im}\left(S^p \cF^{mx} R_m\rightarrow H^0(X, pmL) \right).
\end{equation}
By \cite[Proposition 5.14]{BL18} and \cite[3.2]{BL18}, there exists a positive constant $C>0$ independent of $v$ such that for all 
$x\le x\le T(v)-\frac{C A(v)}{m}$, $\vol(\cF^{(x)}_{m,\bullet})=\vol(\tilde{\cF}^{(x)}_{m,\bullet})$. So as \cite[Proof of Proposition 5.15]{BL18}, we get:
\begin{eqnarray*}
\Psi(v)&\le&\tilde{\Psi}_m(v)+C\frac{A(v)}{m}=\frac{1}{\bfV}\int_0^{T(v)}\frac{\vol(\tilde{\cF}^{(x)}_{m,\bullet})}{m^n}e^{-x}dx+\frac{C A(v)}{m}\\
&\le&\frac{1}{\bfV}\int_0^{T(v)-CA(v)/m} \frac{\vol(\cF^{(x)}_{m,\bullet})}{m^n}e^{-x}dx+\frac{C A(v)}{m}\\
&\le&\frac{1}{\bfV}\int_0^{T(v)}\frac{\vol(\cF^{(x)}_{m,\bullet})}{m^n}e^{-x}dx+\frac{C A(v)}{m}.
\end{eqnarray*}
For the second inequality we used the estimate that: as $m\rightarrow+\infty$,
\begin{eqnarray*}
\int_{T(v)-CA(v)/m}^{T(v)}e^{-x}dx=e^{-(T(v)-\frac{CA(v)}{m})}-e^{-T(v)}\le e^{\frac{C A(v)}{m}}-1=O(\frac{A(v)}{m}).
\end{eqnarray*}
Fix any $\epsilon>0$, by choosing $m\gg 1$ and $p\gg 1$, we have (see \cite[(5.6)]{BL18}):
\begin{equation}
\left|\frac{\vol(\cF^{(x)}_{m,\bullet})}{m^nV}-\frac{\dim \cF^{(x)}_{m,p}}{N_{mp}}\right|<\epsilon.
\end{equation}
Finally we can estimate as in \cite[Proof of Theorem 5.13]{BL18}: for $m\gg 1$,
\begin{eqnarray*}
\Psi(v)&\le&\frac{1}{\bfV}\int_0^{T(v)}\frac{\vol(\cF^{(x)}_{m,\bullet})}{m^n}e^{-x}dx+\frac{C A(v)}{m}\\
&\le&\int_0^{+\infty}\frac{\dim \cF^{(x)}_{m,p}}{N_{mp}}e^{-x}dx+\epsilon T(v)+\frac{C A(v)}{m}\\
&\le&\int_0^{+\infty}\frac{\dim \cF^{pmx}R_m}{N_{mp}}e^{-x}dx+2\epsilon A(v)\\
&=&\Psi(v)+2\epsilon A(v).
\end{eqnarray*}
In the 3rd inequality, we used again the inequality $1-e^{-T(v)}\le T(v)$. 
Since $\epsilon>0$ is arbitrary, we get the conclusion.

\end{proof}

\vskip 3mm

\noindent
Department of Mathematics, Purdue University, West Lafayette, IN 47907-2067

\noindent
{\it E-mail address: } han556@purdue.edu, li2285@purdue.edu

\vskip 2mm

%\noindent
%School of Mathematical Sciences and BICMR, Peking University, Yiheyuan Road 5, Beijing, P.R.China, 100871

%\noindent {\it E-mail address:} tian@math.princeton.edu

%\vskip 2mm

%\noindent
%School of Mathematical Sciences, Zhejiang University, Zheda Road 38, Hangzhou, Zhejiang,
%310027, P.R. China

%\noindent {\it E-mail address:} wfmath@zju.edu.cn

\end{document}